\documentclass{article}
\usepackage[a4paper,top=2cm,bottom=2cm,left=2cm,right=2cm]{geometry}
\usepackage{amsmath, amsthm}
\usepackage{amscd}
\usepackage{amssymb}
\usepackage{array}
\usepackage{soul}
\usepackage{color}
\usepackage{xcolor}
\usepackage{tikz-cd}

\usepackage{hyperref}
\usepackage[colorinlistoftodos]{todonotes}

\numberwithin{equation}{section}

\allowdisplaybreaks[4]

\newtheorem{thm}{Theorem}[section]
\newtheorem{theorem}[thm]{Theorem}

\newtheorem{corollary}[thm]{Corollary}
\newtheorem{lemma}[thm]{Lemma}
\newtheorem{proposition}[thm]{Proposition}

\theoremstyle{definition}
\newtheorem{definition}[thm]{Definition}
\newtheorem{example}[thm]{Example}

\newtheorem{observation}[thm]{Observation}

\newtheorem{remark}[thm]{Remark}

\newtheorem*{theorem*}{Theorem}

\begin{document}

\newcommand{\comment}[1]{{\color{blue}\rule[-0.5ex]{2pt}{2.5ex}}
\marginpar{\scriptsize\begin{flushleft}\color{blue}#1\end{flushleft}}}

\newcommand{\rosso}[1]{{\color{red}{\sl #1}}}
\newcommand{\blu}[1]{{\color{blue}#1}}
\newcommand{\viola}[1]{{\color{violet}{\sl #1}}}
\newcommand{\mar}[1]{{\color{brown}#1}}
\newcommand{\segnab}[1]{\blu{$\overline{\rule{0pt}{1pt}\smash[t]{\text{\color{black}#1}}}$}}
\newcommand{\msegnab}[1]{\blu{\overline{\rule{0pt}{1pt}\smash[t]{\color{black}#1}}}}
\newcommand{\segnar}[1]{\rosso{$\overline{\rule{0pt}{1pt}\smash[t]{\text{\color{black}#1}}}$}}
\newcommand{\msegnar}[1]{\rosso{\overline{\rule{0pt}{1pt}\smash[t]{\color{black}#1}}}}

\newcommand{\be}{\begin{equation}}
\newcommand{\ee}{\end{equation}}
\newcommand{\bea}{\begin{eqnarray}}
\newcommand{\eea}{\end{eqnarray}}
\newcommand{\bean}{\begin{eqnarray*}}
\newcommand{\eean}{\end{eqnarray*}}
\newcommand{\qi}{q^{-1}}
\newcommand{\ap}{\alpha}
\newcommand{\bt}{\beta}
\newcommand{\di}{\mathrm{d}}

\newcommand{\R}{\mathbb{R}}
\newcommand{\C}{\mathbb{C}}
\newcommand{\Z}{\mathbb{Z}}
\newcommand{\N}{\mathbb{N}}
\newcommand{\bM}{\mathbb{M}}
\newcommand{\g}{\mathfrak{G}}
\newcommand{\epsi}{\varepsilon}

\newcommand{\hs}{\hfill\square}
\newcommand{\hbs}{\hfill\blacksquare}

\newcommand{\bp}{\mathbf{p}}
\newcommand{\bmax}{\mathbf{m}}
\newcommand{\bT}{\mathbf{T}}
\newcommand{\bU}{\mathbf{U}}
\newcommand{\bP}{\mathbf{P}}
\newcommand{\bA}{\mathbf{A}}
\newcommand{\bm}{\mathbf{m}}
\newcommand{\bIP}{\mathbf{I_P}}

\newcommand{\cm}{\gamma}
\newcommand{\dt}{\delta}

\newcommand{\cA}{\mathcal{A}}
\newcommand{\cB}{\mathcal{B}}
\newcommand{\cN}{\mathcal{N}}
\newcommand{\cC}{\mathcal{C}}
\newcommand{\cI}{\mathcal{I}}
\newcommand{\Ogh}{\mathcal{O}(G/P)}
\newcommand{\cM}{\mathcal{M}}
\newcommand{\cO}{\mathcal{O}}
\newcommand{\cG}{\mathcal{G}}
\newcommand{\cU}{\mathcal{U}}
\newcommand{\cJ}{\mathcal{J}}
\newcommand{\cF}{\mathcal{F}}
\newcommand{\cV}{\mathcal{V}}
\newcommand{\cP}{\mathcal{P}}
\newcommand{\ep}{\mathcal{E}}
\newcommand{\E}{\mathcal{E}}
\newcommand{\cH}{\mathcal{O}}
\newcommand{\cPO}{\mathcal{PO}}
\newcommand{\cl}{\ell}
\newcommand{\cFG}{\mathcal{F}_{\mathrm{G}}}
\newcommand{\cHG}{\mathcal{H}_{\mathrm{G}}}
\newcommand{\Gal}{G_{\mathrm{al}}}
\newcommand{\cQ}{G_{\mathcal{Q}}}

\newcommand{\ri}{\mathrm{i}}
\newcommand{\re}{\mathrm{e}}
\newcommand{\rd}{\mathrm{d}}

\newcommand{\rGL}{\mathrm{GL}}
\newcommand{\rSU}{\mathrm{SU}}
\newcommand{\rSL}{\mathrm{SL}}
\newcommand{\rSO}{\mathrm{SO}}
\newcommand{\rOSp}{\mathrm{OSp}}
\newcommand{\rSpin}{\mathrm{Spin}}
\newcommand{\rsl}{\mathrm{sl}}
\newcommand{\rM}{\mathrm{M}}
\newcommand{\rdiag}{\mathrm{diag}}
\newcommand{\rP}{\mathrm{P}}
\newcommand{\rdeg}{\mathrm{deg}}
\newcommand{\Derg}{\mathrm{Der}_\Gamma}
\newcommand{\Der}{\mathrm{Der}}

\newcommand{\M}{\mathrm{M}}
\newcommand{\End}{\mathrm{End}}
\newcommand{\Hom}{\mathrm{Hom}}
\newcommand{\diag}{\mathrm{diag}}
\newcommand{\rspan}{\mathrm{span}}
\newcommand{\rank}{\mathrm{rank}}
\newcommand{\Gr}{\mathrm{Gr}}
\newcommand{\ber}{\mathrm{Ber}}
\newcommand{\coinv}{\mathrm{coinv}}

\newcommand{\fsl}{\mathfrak{sl}}
\newcommand{\fg}{\mathfrak{g}}
\newcommand{\ff}{\mathfrak{f}}
\newcommand{\fgl}{\mathfrak{gl}}
\newcommand{\fosp}{\mathfrak{osp}}
\newcommand{\fm}{\mathfrak{m}}

\newcommand{\str}{\mathrm{str}}
\newcommand{\Sym}{\mathrm{Sym}}
\newcommand{\tr}{\mathrm{tr}}
\newcommand{\defi}{\mathrm{def}}
\newcommand{\Ber}{\mathrm{Ber}}
\newcommand{\spec}{\mathrm{Spec}}
\newcommand{\sschemes}{\mathrm{(sschemes)}}
\newcommand{\sschemeaff}{\mathrm{ {( {sschemes}_{\mathrm{aff}} )} }}
\newcommand{\rings}{\mathrm{(rings)}}
\newcommand{\Top}{\mathrm{Top}}
\newcommand{\sarf}{ \mathrm{ {( {salg}_{rf} )} }}
\newcommand{\arf}{\mathrm{ {( {alg}_{rf} )} }}
\newcommand{\odd}{\mathrm{odd}}
\newcommand{\alg}{\mathrm{(alg)}}
\newcommand{\sa}{\mathrm{(salg)}}
\newcommand{\sets}{\mathrm{(sets)}}
\newcommand{\SA}{\mathrm{(salg)}}
\newcommand{\salg}{\mathrm{(salg)}}
\newcommand{\varaff}{ \mathrm{ {( {var}_{\mathrm{aff}} )} } }
\newcommand{\svaraff}{\mathrm{ {( {svar}_{\mathrm{aff}} )}  }}
\newcommand{\ad}{\mathrm{ad}}
\newcommand{\Ad}{\mathrm{Ad}}
\newcommand{\pol}{\mathrm{Pol}}
\newcommand{\Lie}{\mathrm{Lie}}
\newcommand{\Proj}{\mathrm{Proj}}
\newcommand{\rGr}{\mathrm{Gr}}
\newcommand{\rFl}{\mathrm{Fl}}
\newcommand{\rPol}{\mathrm{Pol}}
\newcommand{\rdef}{\mathrm{def}}

\newcommand{\uspec}{\underline{\mathrm{Spec}}}
\newcommand{\uproj}{\mathrm{\underline{Proj}}}

\newcommand{\sym}{\cong}
\newcommand{\al}{\alpha}
\newcommand{\lam}{\lambda}
\newcommand{\de}{\delta}
\newcommand{\dd}{\delta}
\newcommand{\D}{\Delta}
\newcommand{\s}{\sigma}
\newcommand{\sig}{\sigma}
\newcommand{\lra}{\longrightarrow}
\newcommand{\ga}{\gamma}
\newcommand{\ra}{\rightarrow}
\newcommand{\tE}{\widetilde E}

\newcommand{\NOTE}{\bigskip\hrule\medskip}

\newcommand{\beq}{\begin{equation}}
\newcommand{\eeq}{\end{equation}}
\newcommand{\ddet}{\mathrm{det}}
\newcommand{\proj}{\mathrm{proj}}
\newcommand{\pr}{\mathrm{pr}}
\newcommand{\Ubar}{\overline{U}}
\newcommand{\cT}{\mathcal{T}}
\newcommand{\cL}{\mathcal{L}}
\newcommand{\cE}{\mathcal{E}}

\newcommand{\id}{\mathrm{id}}
\newcommand{\coi}{\mathrm{co}}

\newcommand{\les}{{\,\raisebox{0.95 pt}{\scalebox{0.57}[0.57]{\mbox{$\leqslant$}}}}}

\medskip

\centerline{\Large \bf Differential Calculi on Quantum Principal Bundles} 

\medskip
\centerline{\Large  \bf over Projective Bases}

\medskip
\medskip
\medskip
\medskip
\centerline{\Large{{P. Aschieri$,^{\!1,2,3}$ R. Fioresi$,^{\!5,6}$ 
E. Latini$,^{4,6}$ T. Weber$^{\,2,7}$} }}

\vskip 0.5 cm
\centerline{$^1${\sl Dipartimento di Scienze e Innovazione 
Tecnologica, Universit\`a del Piemonte Orientale}}

\centerline{{\sl  
Viale T. Michel - 15121, Alessandria, Italy}}
\centerline{\texttt{paolo.aschieri@uniupo.it}}

\vskip 0.1 cm

\centerline{$^2${\sl Istituto Nazionale di Fisica Nucleare, Sezione di Torino,
  via P. Giuria 1, 10125 Torino}}

\vskip 0.1 cm
\centerline{$^3${\sl Regge Center, Torino, via P. Giuria 1, 10125, Torino, Italy}}

\vskip 0.5 cm

\centerline{
$^4${\sl Dipartimento di Matematica, Universit\`{a} di
Bologna}}

\centerline{\sl Piazza di Porta S. Donato 5, I-40126 Bologna, Italy}

\centerline{\texttt{emanuele.latini@UniBo.it}}
\bigskip

\centerline{
$^4${\sl FaBiT, Universit\`{a} di
Bologna}}

\centerline{\sl via S. Donato 15, I-40126 Bologna, Italy}

\centerline{\texttt{rita.fioresi@UniBo.it}}
\bigskip

\centerline{$^6${\sl Istituto Nazionale di Fisica Nucleare, Sezione di Bologna,}}
\centerline{via Berti Pichat 6/2, I-40127 Bologna}

\vskip 0.5 cm
\centerline{$^7${\sl Dipartimento di Matematica "Giuseppe Peano", Universit\`a degli Studi di Torino}}

\centerline{{\sl  
via Carlo Alberto 10, 10123 Torino, Italy}}
\centerline{\texttt{thomas.weber@unito.it}}

\bigskip

\begin{abstract}

We propose a sheaf-theoretic approach to the theory of differential
calculi on quantum principal bundles over non-affine bases.
After recalling the affine case we define differential calculi
on sheaves of comodule algebras as sheaves of covariant bimodules
together with a morphism of sheaves -the differential-  such that the Leibniz rule
and surjectivity hold locally. The main
class of examples is given by covariant calculi over quantum flag
manifolds, which we provide via an explicit Ore extension
construction. 
In a second step  we introduce principal covariant calculi 
by requiring a local compatibility of the calculi on the total sheaf,
base sheaf and the structure Hopf algebra in terms of exact
sequences.
In this case Hopf--Galois extensions of algebras lift to
Hopf--Galois extensions of exterior algebras with compatible
differentials.
In particular, the examples of principal (covariant)
calculi on the quantum principal bundles
$\mathcal{O}_q(\mathrm{SL}_2(\mathbb{C}))$ and
$\mathcal{O}_q(\mathrm{GL}_2(\mathbb{C}))$  over the projective space
$\mathrm{P}^1(\mathbb{C})$
are discussed in detail.
\end{abstract}

\setcounter{tocdepth}{2}

\tableofcontents

\section{Introduction} \label{intro-sec}

While in classical differential geometry there is a canonical
(functorial) construction of the differential calculus on a
differentiable manifold, a main feature of noncommutative differential
geometry is the non uniqueness of the calculus. Even on quantum
groups and requiring (co)invariance conditions there are interesting
different calculi \cite{wor}. In this perspective, the definition of a
noncommutative geometry requires also a characterization of its
differential calculus structure. Going beyond quantum groups we have quantum
homogeneous spaces and quantum principal bundles. Their study has been
very active in the last three decades when the base space is a
noncommutative deformation of an affine variety. In this case, as usual in
noncomutative geometry, the
affine geometric objects are replaced by a deformation of the algebra
of functions on them. The case of bundles
on projective varieties is intrinsically more difficult
since it requires patching  affine opens.
In noncommutative geometry, when it is necessary
to go beyond the affine setting, we  must take the ringed space approach
as expressed by Grothendieck:
\textit{to do geometry one does not need the space itself, 
but only the category of sheaves on that would-be space} (see
\cite{gr} and also \cite{sk2} with refs. therein 
for more insight on this point of view).

In this paper we propose an approach to noncommutative differential geometry
which employs the sheaf-theoretic language developed in
\cite{AFL}. This does not only cover the established affine theory but
also allows to consider differential structures
for quantum principal bundles over non-affine bases.

\medskip 
\medskip
\medskip

Important steps towards a sound foundation of noncommutative 
algebraic geometry were taken in the works \cite{az, kr, ro, vov}.
In \cite{az},  quantum projective schemes are understood as the
category of modules associated with a given graded noncommutative algebra,
while in \cite{ro}, an affine quantum space 
is viewed as the spectrum of a noncommutative
ring and quantum projective schemes are defined accordingly with
a gluing procedure in analogy with the classical $\mathrm{Proj}$
construction.
Indeed this is our perspective on non-affine noncommutative spaces: we 
shall define, following \cite{AFL}, a quantum space as a locally
ringed space and then build quantum principal bundles over quantum
spaces as sheaves with suitable (co)invariant properties  with respect
to a fixed quantum structure group.

While we have extensive studies on quantum principal bundles in the affine setting (see e.g . \cite{brz1, Dabrowski, bh, HKMZ}),
the more general non-affine case (considered e.g.~in \cite{hps}) lacks a
comprehensive and exhaustive treatment. In \cite{AFL}  we study this
more general setting combining the
sheaf approach to principal bundles with affine bases of
\cite{pflaum} (see also \cite{cipa} and \cite{acp}) with the
quantum projective
homogeneous spaces construction of \cite{cfg,fg}, where 
a graded quantum ring  is
built out of the key notion of {\sl quantum section}, 
associated to a quantum line bundle, with a gluing procedure similar to the classical
$\mathrm{Proj}$ one in \cite{ro}.  In particular, we study the
important example of quantum principal bundles on quantum flags of algebraic groups, the
total space algebra being a quantum group.

More specifically, the main construction  is based on quantizations
of the classical principal bundle $G \to G/P$, where $G$ is a
semisimple complex algebraic group and $P$ a parabolic subgroup, so that
$G/P$ is a projective variety. In this case the $P$-invariant ring
$\cO(G/P)$, i.e., the $\cO(P)$-coaction invariant (for short
coinvariant) ring $\cO(G/P)$ is trivially $\mathbb{C}$ and is then replaced by the homogeneous coordinate ring
$\tilde\cO(G/P)$ of $G/P$ with respect to a chosen projective
embedding, associated with a very ample line bundle $\mathcal{L}$. This line
bundle  $\mathcal L$ can be recovered more algebraically via a
character $\chi$ of $P$; the corresponding sections are the
$\cO(P)$-semi-coinvariant elements of $\cO(G)$ with respect to $\chi$
and generate the homogeneous coordinate ring $\tilde\cO(G/P)$ of
$G/P$. 
The relation
between this latter and the structure sheaf $\cO_{G/P}$ of $G/P$
is then as usual by considering projective
localizations (zero degree subalgebras of the localizations) of
$\tilde\cO( G / P )$.

Let $\cO_q(G)$ be a quantum deformation of $G$ and $\cO_q(P)$ 
a quantum deformation of a parabolic subgroup;
both $\cO_q(G)$ and $\cO_q(P)$ are Hopf algebras (see \cite{cfg} Sect.~3 
and \cite{AFL}). 
Via the datum of a {\sl quantum section} $s\in\cO_q(G)$, quantum
version of the lift to $\cO(G)$ of the character $\chi$ of $P$ defining the line
bundle $\mathcal{L}$, we obtain a quantization of $\mathcal L$.
Indeed, through the quantum section  we define a graded algebra $\tilde\cO_q(G/P)$
  (quantum homogeneous coordinate ring):
  each graded component consists of the elements of $\cO_q(G)$, which
are not coinvariant, but transform
according to powers of the projection of $s$ on $\cO_q(P)$.
Furthermore, the quantum section $s$ allows to 
define two sheaves $\cF_G$ and $\cO_M$ on $M=G/P$ 
and in \cite{AFL} we prove
that
$\cF_G$ is a sheaf of 
$\cO_q(P)$-comodule
algebras on the quantum ringed space $(M,\cO_M)$.  Given the locally principal
comodule algebra property (or faithfully flatness property) of
$\cF_G$, then this is a quantum
principal bundle. In particular the (coproduct of the) quantum section determines 
an open cover $\{V_i\}$ of $G$. The projected open cover $\{U_i:=p(V_i)\}$ of
$M$ (where $p:G\to M=G/P$) then lead to the intersections $U_I:=U_{i_1} \cap \dots \cap U_{i_r}$, 
$I=(i_1,\dots, i_r)$ which form the basis $\mathcal{B}$ for the topology used to define both the sheaves
$\cF_G$ and $\cO_M$.

\medskip 
\medskip
\medskip

In the present paper we develop a theory of differential calculi on
sheaves which is suited but not limited to affine and non-affine
examples such as quantum flags for quantum algebraic groups.
In particular we construct first order differential calculi on the quantum principal bundles $\cF_G$ that are
canonically given once a calculus on $\cO_q(G)$ is chosen.
For a different approach to
quantum differential calculi on quantum flags
we refer to \cite{HK1, HK2, ROB2}.

We start with the discussion of first order
differential calculus in the affine 
setting, that is, when the base and total spaces
are affine. Our focus is on the construction and induction of differential structures:
via tensor products, algebra homomorphisms and quotients. This also on the level of
covariant and bicovariant first order differential calculi.
In particular, we recall the construction of covariant calculi on smash
product algebras from \cite{pflaum2}, where a calculus on a module algebra
and a bicovariant calculus on the structure Hopf algebra shape the so-called
smash product calculus. This affine framework has been fruitfully investigated
in a series of works, including \cite{brz1, DD, KLS, lz, ROB, pflaum2}.

{We conclude our discussion of the affine setting with a
treatment of principal covariant calculi.
Associated with a commutative principal bundle we canonically have the exact
sequence of horizontal forms into forms on the total space onto
vertical forms. Even for principal comodule algebras (faithfully flat
Hopf--Galois extensions) exactness of this sequence in the
noncommutative case is an extra requirement. It is known that the
exterior algebra $\Omega^\bullet_H$ of a bicovariant calculus on a
quantum group $H$ is a graded Hopf algebra. We here for simplicity study the quotient
algebra $\Omega^{\les 1}_H=H\oplus\Gamma_H$ where forms of degree 2
and higher are set to zero. Similarly we can study when the (truncated) exterior
algebra $\Omega^{\leqslant 1}_A=A\oplus\Gamma_A$ is a $\Omega^{{\scriptstyle{\leqslant}}
  1}_H$-comodule algebra, and we can study the induced module of one forms $\Gamma_B$ on the
subalgebra $B=A^{\coi H}\subseteq A$ of coinvariants.
It was shown in \cite{SchHG,SchHGFF} that exactness of the sequence
of modules
$$
0\rightarrow A\otimes_B\Gamma_B\longrightarrow\Gamma_A
\longrightarrow
    A\square_H\Gamma_H\rightarrow 0
$$
(where the cotensor product  $A\square_H\Gamma_H$ equals $ A\otimes {}^{\coi H}\Gamma_H $,
with ${}^{\coi H}\Gamma_H$  the module of left coinvariant one-forms on
the quantum group $H$) is
equivalent to principality of the graded Hopf--Galois extension
$(\Omega^{\leqslant 1}_A)^{\coi\:\! \Omega^{\leqslant 1}_H}\subseteq \Omega^{\leqslant 1}_A$.
We revisit these results
considering differential graded algebras, where the above sequence
becomes that of noncommutative horizontal forms, forms on $A$,
vertical forms. The sequence then defines a principal covariant calculus on 
$A$. 
We show that these calculi correspond to graded Hopf--Galois extensions
where the $\Omega^{\leqslant 1}_H$-coaction  is differentiable (compatible with the
differentials on $A$ and on $A\otimes H$).
The main examples  of differential calculi on sheaves we shall
encounter satisfy such an exact sequence locally. In this section
dedicated to the affine setting we
further clarify the relation of principal covariant calculi to quantum principal bundles as
pioneered in \cite{brz1}.

\medskip

Building on the affine case results we proceed to introduce the sheaf approach to first order differential calculi. Namely, we define a first order differential calculus
$(\Upsilon,\mathrm{d})$ on a sheaf $\mathcal{F}$ of comodule algebras
as a sheaf $\Upsilon$ of covariant $\mathcal{F}$-bimodules together with
a morphism $\mathrm{d}\colon\mathcal{F}\rightarrow\Upsilon$ of sheaves
of comodules. We demand the Leibniz rule and surjectivity
of the $\mathcal{F}$-linear span of the differential only locally on
stalks. These are the two characterizing
properties of a first order differential calculus, and their holding true on stalks goes along with
the property that stalks of a quantum principal bundle are principal comodule algebras. 

The following result, proven in Theorem \ref{thm02}, shows that
the construction of the sheaf $\mathcal{F}_G$ from the Hopf algebra
$\mathcal{O}_q(G)$ developed in \cite{AFL} (with $G$ a semisimple complex
algebraic group, $P$ a parabolic subgroup) extends to the level of first order differential calculi.

\begin{theorem*} \it{Let $(\Gamma,\mathrm{d})$ be a right covariant first order differential calculus on the Hopf algebra $\cO_q(G)$.
\begin{enumerate}
\item[i.)] The assignment
\begin{equation*}
    \Upsilon_G\colon U_I\mapsto
    \Upsilon_G(U_I):= \cF_G(U_I) \otimes_{\cO_q(G)} \Gamma \otimes_{\cO_q(G)} \cF_G(U_I)
\end{equation*}
with $\mathrm{d}\colon U_I\mapsto$
$(\mathrm{d}_I\colon\cF(U_I)\rightarrow\Upsilon_G(U_I))$
induced by Ore extensions, defines a
right $\cO_q(P)$-covariant first order differential calculus
on the sheaf $\cF_G$.

\item[ii.)] The first order differential calculus $(\Upsilon_G,\mathrm{d})$ on the sheaf 
$\cF_G$  
induces a first order differential calculus $(\Upsilon_M,\mathrm{d}_M)$ on the sheaf
$\cO_M=\mathcal{F}_G^{\coi\!\: \cO_q(P)}$.
\item[iii.)] If $\cF_G$ is a quantum principal bundle, the sheaf of base forms is isomorphic, as a sheaf of $\cO_M$-bimodules, to the intersection
of that of horizontal and $\cO_q(P)$-coinvariant forms: $\Upsilon^{}_M\cong\Upsilon_G^{\mathrm{co}_{\:\!}\cO_q(P)}\cap\Upsilon_G^\mathrm{hor}
$.
\end{enumerate}}
\end{theorem*}
}
\noindent For the sheaf $\cF_G$, which is obtained from the Hopf algebra
$\cO_q(G)$, the definition of the calculus on $\cO_q(G)$ uniquely
determines the quotient calculus on $\cO_q(P)$, and from
Theorem~\ref{thm02}, the Ore extended calculus on
$\cF_G$ and that on the coinvariant subsheaf $\cO_M=\mathcal{F}_G^{\coi\!\: \cO_q(P)}$.
Thus, despite non uniqueness of the noncommutative differential
calculus, once defined the calculus on $\cO_q(G)$ the other calculi are canonically obtained.

We next proceed to extend our sheaf-theoretic treatment to the
notion of principal covariant calculus on a quantum principal bundle, which accounts for
compatibility of the calculus
$(\Upsilon,\mathrm{d})$ on the total sheaf $\mathcal{F}$, the calculus
$(\Upsilon_M,\mathrm{d}_M)$ on the base sheaf $\mathcal{O}_M$ and
a bicovariant calculus $(\Gamma_H,\mathrm{d}_H)$ on the structure Hopf algebra
$H$. 
For a principal covariant calculus we demand the sequence of stalks
\begin{equation}\label{zero}
    0\rightarrow\mathcal{F}_p\otimes_{(\mathcal{O}_M)^{}_p}(\Upsilon_M)^{}_p
    \rightarrow\Upsilon_p\longrightarrow
    \mathcal{F}_p\square_H
    \Gamma_H\rightarrow 0
  \end{equation}
to be exact for every $p\in M$, and the calculus on $H$ to be
bicovariant, cf. equation \eqref {eq109}.

\medskip

The rest of the article is devoted to providing three explicit
examples of this canonical construction. We study calculi on the quantum principal bundles
$\mathcal{O}_q(\mathrm{SL}_2(\mathbb{C}))$ and $\mathcal{O}_q(\mathrm{GL}_2(\mathbb{C}))$,
both over $\mathbf{P}^1(\mathbb{C})$. The first example is based on a
4D bicovariant
differential calculus on $\mathcal{O}_q(\mathrm{SL}_2(\mathbb{C}))$. The associated
sequence \eqref{zero} is well-defined but not exact.
Despite the base sheaf $\cO_{\mathbf{P}^1(\mathbb{C})}$ being commutative the
resulting calculus is noncommutative (the bimodule of one forms is
noncommutative). 

The second example
is based on a 3D calculus on $\mathcal{O}_q(\mathrm{SL}_2(\mathbb{C}))$ and the
associated sequence is exact but not a principal covariant sequence
since the induced calculus on the quantum parabolic subgroup
$\mathcal{O}_q(\mathrm{P})$ ($\mathrm{P}\subseteq \mathrm{SL}_2$ being
the Borel subgroup of upper triangular matrices) is not  bicovariant.

The third example has total space algebra
$\mathcal{O}_q(\mathrm{GL}_2(\mathbb{C}))$ with a 4D bicovariant
calculus. This induces a principal covariant calculus
on the quantum principal bundle
$\mathcal{F}_{\mathrm{GL}_q}$ over the base sheaf
$\cO_{\mathbf{P}^1(\mathbb{C})}$. Even though  the sheaf
$\mathcal{F}_{\mathrm{GL}_q}$ is locally trivial, i.e., on an open
cover $\{U_i\}$ of ${\mathbf{P}^1(\mathbb{C})}$ the algebras $\mathcal{F}_{\mathrm{GL}_q}(U_i)$ are isomorphic to smash product algebras, the associated smash product
calculi do not give a first order differential calculus on the sheaf
$\mathcal{F}_{\mathrm{GL}_q}$. 
This is an instance where, while the sheaf defining the quantum principal bundle is
locally trivial, the sheaf of one forms is not.

\section{Preliminary concepts}\label{prelim-sec}

In this section we establish our notation and provide the key
results we will need in the sequel. Though most of the material
of this section is known, 
we include it by completeness, so that the reader has readily available
the results in the form we need them later on.

\subsection{Notation on modules and comodules} 

Let $\Bbbk$ be a commutative unital ring (in later sections
specialized to $\mathbb{C}[q,q^{-1}]$ and to a field) and let $\otimes$
be the tensor product of $\Bbbk$-modules. 
Algebras over $\Bbbk$ will be associative and unital.
Let $H$ be a \textit{Hopf algebra} over $\Bbbk$. Throughout this paper
we denote the coproduct of $H$ by $\Delta$ and its counit by $\epsilon$.
Recall that those are algebra homomorphisms $\Delta\colon H\rightarrow H\otimes H$ and
$\epsilon\colon H\rightarrow\Bbbk$ such that
$(\Delta\otimes\mathrm{id}_H)\circ\Delta=(\mathrm{id}_H\otimes\Delta)\circ\Delta$
(coassociativity) and 
$(\epsilon\otimes\mathrm{id}_H)\circ\Delta=(\mathrm{id}_H\otimes\epsilon)\circ\Delta
=\mathrm{id}_H$
(counitality) hold.
We employ Sweedler's notation $\Delta(h)=h_1\otimes h_2$ for the coproduct of an element
$h\in H$, i.e.~we omit
summation symbols and summation indices. By the coassociativity we write
\begin{equation*}
    h_1\otimes h_2\otimes h_3
    :=(\Delta\otimes\mathrm{id}_H)(\Delta(h))
    =(\mathrm{id}_H\otimes\Delta)(\Delta(h))
\end{equation*}
for all $h\in H$ and similarly for higher coproducts. The antipode $S\colon H\rightarrow H$
of $H$ is an anti-bialgebra homomorphism, namely $S(hh')=S(h')S(h)$, $S(1)=1$, 
$\Delta(S(h))=S(h_2)\otimes S(h_1)$, $\epsilon\circ S=\epsilon$,
and furthermore we have $S(h_1)h_2=\epsilon(h)1=h_1S(h_2)$ for all $h,h'\in H$. 
We always assume that $H$ has {\it{invertible  antipode} } and denote its inverse by $\overline{S}\colon H\to H$. The latter is an anti-bialgebra
homomorphism such that $\overline{S}(h_2)h_1=\epsilon(h)1=h_2\overline{S}(h_1)$
for all $h\in H$.

A \textit{right $H$-comodule} is a $\Bbbk$-module $M$, together with a $\Bbbk$-linear map
$\Delta_R\colon M\rightarrow M\otimes H$
such that
\begin{equation}\label{eq03'}
    (\Delta_R\otimes\mathrm{id}_H)\circ\Delta_R
    =(\mathrm{id}_M\otimes\Delta)\circ\Delta_R
\end{equation}
and $(\mathrm{id}_M\otimes\epsilon)\circ\Delta_R=\mathrm{id}_M$,
while a \textit{left $H$-comodule} is a $\Bbbk$-module $M$ with a $\Bbbk$-linear map
$\Delta_L\colon M\rightarrow H\otimes M$ such that
\begin{equation}\label{eq04'}
    (\mathrm{id}_H\otimes\Delta_L)\circ\Delta_L
    =(\Delta\otimes\mathrm{id}_M)\circ\Delta_L
\end{equation}
and $(\epsilon\otimes\mathrm{id}_M)\circ\Delta_L=\mathrm{id}_M$.
The maps $\Delta_R$ and $\Delta_L$ are called right and left $H$-coaction,
respectively. We use the Sweedler's like notations
$\Delta_R(m)=m_0\otimes m_1$ and $\Delta_L(m)=m_{-1}\otimes m_0$,
where $m\in M$. In case $M$ is a right $H$-comodule, we denote the
$\Bbbk$-submodule of right $H$-coaction invariant elements,
simply called $H$-coinvariant elements,
by
\begin{equation*}
    M^{\mathrm{co}H}:=\{m\in M~|~\Delta_R(m)=m\otimes 1\}
\end{equation*}
Similarly, if $M$ is a left $H$-comodule, 
${}^{\mathrm{co}H}M:=\{m\in M~|~\Delta_L(m)=1\otimes m\}$ is the $\Bbbk$-submodule
of left $H$-coinvariant elements.
Accordingly to  (\ref{eq03'}) and (\ref{eq04'}), we further use the notations
\begin{equation*}
    m_0\otimes m_1\otimes m_2:=(\Delta_R\otimes\mathrm{id}_H)(\Delta_R(m))
    =(\mathrm{id}_M\otimes\Delta)(\Delta_R(m)),
\end{equation*}
$m_{-2}\otimes m_{-1}\otimes m_0
:=(\mathrm{id}_H\otimes\Delta_L)(\Delta_L(m))
=(\Delta\otimes\mathrm{id}_M)(\Delta_L(m))$ and similarly for higher coactions.
A right $H$-comodule $M$ which is also a left $H$-comodule is said to be an
\textit{$H$-bicomodule} if $\Delta_R$ and $\Delta_L$ are commuting coactions, i.e.,
\begin{equation*}
    (\Delta_L\otimes\mathrm{id}_H)\circ\Delta_R
    =(\mathrm{id}_H\otimes\Delta_R)\circ\Delta_L.
\end{equation*}
In Sweedler's like notation this commutativity  reads $m_{-1}\otimes
(m_0\otimes m_1)=(m_{-1}\otimes m_0)\otimes m_1$, and hence we simply
write $m_{-1}\otimes m_0\otimes m_1$.
If $M$ is an $H$-bicomodule then $M^{\mathrm{co}H}$ and $^{\mathrm{co}H}M$
are left and right $H$-subcomodules, respectively. A \textit{right $H$-comodule morphism} is a $\Bbbk$-linear map $\phi\colon M\rightarrow N$
between right $H$-comodules $(M,\Delta_M)$ and $(N,\Delta_N)$ such
that $\Delta_N\circ \phi =(\phi\otimes \id) \circ \Delta_M\,$, i.e., 
\begin{equation}\label{eq24}
    \phi(m)_0\otimes\phi(m)_1
    =\phi(m_0)\otimes m_1
\end{equation}
for all $m\in M$. A $\Bbbk$-linear map
satisfying (\ref{eq24}) is also called right $H$-colinear.
Similarly, left $H$-colinear and $H$-bicolinear maps are defined.

\subsection{Covariant bimodules}

Algebras in the category of comodules have a multiplication and unit that are
compatible with the coaction.
Since we are encountering several different coactions in the course of this
section we utilize the convention to denote the right $H$-coaction on a
right $H$-comodule algebra $A$ by $\delta_R\colon A\rightarrow A\otimes H$.
\begin{definition}
A \textit{right $H$-comodule algebra} is a right $H$-comodule $(A,\delta_R)$ together
with an associative product $\mu\colon A\otimes A\rightarrow A$ and a unit 
$\eta\colon\Bbbk\rightarrow A$, such that
\begin{equation*}
    \delta_R\circ\mu
    =(\mu\otimes\mu_H)\circ(\mathrm{id}_A\otimes\tau_{H,A}\otimes\mathrm{id}_H)
    \circ(\delta_R\otimes\delta_R)
\end{equation*}
and $\delta_R\circ\eta=\eta\otimes\eta_H$, where $\tau_{H,A}\colon H\otimes A
\rightarrow A\otimes H$, $h\otimes a\mapsto a\otimes h$ denotes the
flip of $\Bbbk$-modules and $(\mu_H,\eta_H)$ is the
algebra structure of $H$. Using the previously introduced short notation,
the compatibility conditions read
\begin{equation*}
    \delta_R(aa')
    =(aa')_0\otimes(aa')_1
    =a_0a'_0\otimes a_1a'_1
    =\delta_R(a)\delta_R(a')
\end{equation*}
for all $a,a'\in A$
and $\delta_R(1)=1_0\otimes 1_1=1\otimes 1_H$.
A \textit{right $H$-comodule algebra homomorphism} is a right $H$-comodule homomorphism which is
also an algebra homomorphism.
Similarly, one defines left $H$-comodule algebras
and $H$-bicomodule algebras and their morphisms.
\end{definition}
Fix a right $H$-comodule algebra $(A,\delta_R)$ in the following.
We now introduce $A$-bimodules in the category of right $H$-comodules.

\begin{definition}\label{def03}
An $A$-bimodule $M$ is called a \textit{right $H$-covariant $A$-bimodule} if
there is a right $H$-comodule action $\Delta_R\colon M\rightarrow M\otimes H$ on $M$,
such that
\begin{equation*}
    \Delta_R(a\cdot m\cdot a')
    =\delta_R(a)\Delta_R(m)\delta_R(a'),
\end{equation*}
or equivalently, in short notation, $(a\cdot m\cdot a')_0\otimes(a\cdot m\cdot a')_1
=a_0\cdot m_0\cdot a'_0\otimes a_1m_1a'_1$
for all $a,a'\in A$ and $m\in M$. A right \textit{$H$-covariant right (resp. left) $A$-module}
is a right (resp. left) $A$-module, where we only ask compatibility of a right
$H$-coaction with the right (resp. left) $A$-module structure.
Similarly, left $H$-covariant $A$-modules and $H$-bicovariant $A$-bimodules are defined,
if $A$ is a left $H$-comodule algebra or an $H$-bicomodule algebra, respectively.
\end{definition}

We are particularly interested in the case $A=H$, since $H$ is naturally
a right and a left $H$-comodule algebra with respect to the coactions given by
the coproduct.
\begin{definition}
A \textit{bicovariant $H$-bimodule} is an $H$-bimodule and an $H$-bicomodule $M$,
such that
\begin{equation*}
    \Delta_R(h\cdot m\cdot h')
    =\Delta(h)\Delta_R(m)\Delta(h')
    \text{ and }
    \Delta_L(h\cdot m\cdot h')
    =\Delta(h)\Delta_L(m)\Delta(h')
\end{equation*}
for all $h,h'\in H$ and $m\in M$.
\end{definition}
For $H$-covariant $H$-modules there is the following fundamental theorem
(c.f. \cite{mo}~Thm.~1.9.4, \cite{Schauenburg94}).
\begin{proposition}\label{fundthm}
For any right $H$-covariant right $H$-module $M$,
there is an isomorphism
\begin{equation}\label{eq104}
    M\to M^{\mathrm{co}H}\otimes H~,~~ m\mapsto m_0S(m_1)\otimes m_2~ ,
\end{equation}
of right $H$-covariant right $H$-modules.
If $M$ is a bicovariant $H$-bimodule, (\ref{eq104}) extends
to an isomorphism of bicovariant $H$-bimodules.

If $M$ is a left $H$-covariant left $H$-module,
there is an isomorphism
\begin{equation}\label{eq105}
    M\to H\otimes{}^{\mathrm{co}H}M~,~~ m\mapsto m_{-2}\otimes S(m_{-1})m_0~,
\end{equation}
of left $H$-covariant left $H$-modules.
If $M$ is a bicovariant $H$-bimodule, (\ref{eq105}) extends
to an isomorphism of bicovariant $H$-bimodules.
\end{proposition}
Note that in (\ref{eq104}) the tensor product 
$M^{\mathrm{co}H}\otimes H$ is a right $H$-covariant right $H$-module with respect to the
diagonal coaction and the right $H$-action 
$(\overline{m}\otimes h)\cdot h':=\overline{m}\otimes(hh')$, where 
$\overline{m}\in M^{\mathrm{co}H}$ and $h,h'\in H$.
If $M$ is furthermore a bicovariant $H$-bimodule
we endow $M^{\mathrm{co}H}\otimes H$
with the diagonal left $H$-coaction and the left $H$-action 
$h\cdot(\overline{m}\otimes h'):=h_1\overline{m}S(h_2)\otimes h_3h'$.
Similarly in (\ref{eq105}),
$H\otimes{}^{\mathrm{co}H}M$ becomes a left $H$-covariant left $H$-module via the
diagonal coaction and the left $H$-action given by
$h\cdot(h'\otimes\overline{m}):=(hh')\otimes\overline{m}$, where
$h,h'\in H$ and $\overline{m}\in{}^{\mathrm{co}H}M$.
If $M$ is furthermore a bicovariant $H$-bimodule
we endow $H\otimes{}^{\mathrm{co}H}M$
with the diagonal right $H$-coaction and the right $H$-action 
$(h\otimes\overline{m})\cdot h':=hh'_1\otimes S(h'_2)\overline{m}h'_3$.

Using the inverse  $\overline{S}$ of the antipode $S$ there are analogous isomorphism of
right $H$-covariant left $H$-modules and left $H$-covariant right $H$-modules, respectively.
The fundamental theorem allows us to write any element $m$ of a bicovariant $H$-bimodule
as $m=a_im^i$, where $\{m^i\}_i\in I\subseteq{}^{\mathrm{co}H}M$ is a basis
of (left) coinvariant elements and $a_i$ are coefficients in $H$.
This is particularly useful when one is dealing with $H$-linear maps of bicovariant bimodules.
Then, it is sufficient to specify properties on coinvariant elements.

\subsection{Hopf--Galois extensions}\label{SectHopfGalois}

In this section $H$ denotes a Hopf algebra and $A$ a right $H$-comodule algebra
with right $H$-comodule action $\delta_R\colon A\rightarrow A\otimes H$.
In particular $H$ is a right $H$-comodule algebra with coaction $\Delta$.
\begin{definition}\label{DefHG}
The algebra extension $B:=A^{\mathrm{co}H}\subseteq A$ is said to be
\begin{enumerate}
\item[i.)] a \textit{Hopf--Galois extension}, if the $\Bbbk$-linear map
\begin{equation*}
    \chi:=(\mu\otimes\mathrm{id}_H)\circ(\mathrm{id}_A\otimes_B\delta_R)
    \colon A\otimes_BA\rightarrow A\otimes H~,~~a\otimes a'\mapsto\chi(a\otimes_Ba')=aa'_0\otimes a'_1
\end{equation*}
is a bijection, where $\mu\colon A\otimes_B A\rightarrow A$ is the
product of $A$ induced on the balanced tensor product $A\otimes_B A$.

\item[ii.)] a \textit{cleft extension}, if there is a $\Bbbk$-linear map $j\colon H\rightarrow A$,
the \textit{cleaving map}, such that
\begin{enumerate}
    \item[1.)] $j$ is right $H$-colinear, i.e.,
    $j(h)_0\otimes j(h)_1=j(h_1)\otimes h_2$ for all $h\in H$,
    
    \item[2.)] $j$ is convolution invertible, i.e., there is a $\Bbbk$-linear map
    $j^{-1}\colon H\rightarrow A$ such that
    $j(h_1)j^{-1}(h_2)=\epsilon(h)1=j^{-1}(h_1)j(h_2)$ for all $h\in H$,
    
    \item[3.)] $j$ respects the units, i.e., $j(1_H)=1$.
\end{enumerate}

\item[iii.)] a \textit{trivial extension}, if there is a right $H$-comodule algebra
homomorphism $j\colon H\rightarrow A$.
\end{enumerate}
\end{definition}
A trivial extension $A^{\mathrm{co}H}\subseteq A$ is automatically a cleft extension since an 
$H$-comodule algebra map $j\colon H\to A$ is right $H$-colinear and unital by assumption
and its convolution inverse is $j^{-1}=j\circ S$. This trivially
implies that the composition $j^{-1}\circ\overline{S}
\colon H\rightarrow A$ is right $H$-colinear. In the following lemma we prove that
this $H$-colinearity property holds for any cleft extension.
\begin{lemma}\label{lemma01}
The inverse of the cleaving map satisfies
\begin{equation}\label{eq11}
    \delta_R\circ j^{-1}
    =(j^{-1}\otimes S)\circ\tau_{H,H}\circ\Delta~,
\end{equation}
where $\tau_{H,H}$ is the flip.  In particular, $j^{-1}\circ \overline{S}$ is right $H$-colinear. 
\end{lemma}
\begin{proof}
Equation (\ref{eq11}) is proven in \cite{mo}~Lem.~7.2.6, 1). Then,
\begin{align*}
    \delta_R(j^{-1}(\overline{S}(h)))
    =j^{-1}(\overline{S}(h)_2)\otimes S(\overline{S}(h)_1)
    =j^{-1}(\overline{S}(h_1))\otimes S(\overline{S}(h_2))
    =j^{-1}(\overline{S}(h_1))\otimes h_2
\end{align*}
for all $h\in H$, where we used that $\overline{S}$ is an anti-bialgebra homomorphism.
\end{proof}
To complete the hierarchy of Definition~\ref{DefHG}, we note that cleft extensions
are in particular Hopf--Galois extensions. We shall frequently
  call them cleft Hopf--Galois extensions.
\begin{proposition}\label{prop04} $B:=A^{\mathrm{co}H}\subseteq A$ is a cleft
extension if and only if $B\subseteq A$ is a Hopf--Galois extension and
there is a left $B$-module and right $H$-comodule isomorphism $B\otimes H\cong A$.
\end{proposition}
The proof of this proposition (see e.g. \cite{mo}~Thm. 8.2.4)
relies on the construction of the following left $B$-module and right $H$-comodule map
\begin{equation*}
    \theta\colon B\otimes H \to  A\,,\,\,\,\,
    b\otimes h\mapsto bj(h)
\end{equation*}
with inverse given by
\begin{equation*}
    \theta^{-1}\colon  A\to B \otimes H\,,\,\,\,\,
    a\mapsto a_0j^{-1}(a_1)\otimes a_2.
\end{equation*}
Note that $\theta^{-1}$ is well-defined, i.e. $a_0j^{-1}(a_1)\in B$, since
\begin{align*}
    \delta_R(a_0j^{-1}(a_1))
    =a_0j^{-1}(a_2)_0\otimes a_1j^{-1}(a_2)_1
    =a_0j^{-1}(a_3)\otimes a_1S(a_2)
    =a_0j^{-1}(a_2)\otimes\epsilon(a_1)1
    =a_0j^{-1}(a_1)\otimes 1,
\end{align*}
where we employed Lemma~\ref{lemma01} (see also \cite{mo}~Lem.~7.2.6,
2.).
\medskip

We conclude this section by recalling the notion of principal comodule algebra
from \cite{Dabrowski} and Schneider theorem for faithfully flat Hopf--Galois extensions.
\begin{definition}\label{PrinComAl}
A right $H$-comodule algebra $A$ is said to be a principal comodule algebra
if $B:=A^{\mathrm{co}H}\subseteq A$ is a Hopf--Galois extension and if
$A$ is right $H$-equivariantly projective as a left $B$-module.
The latter means that there exists a left $B$-linear and right $H$-colinear map
$s\colon A\rightarrow B\otimes A$ such that $m\circ s=\mathrm{id}_A$,
where $m\colon B\otimes A\rightarrow A$ denotes the restricted product of $A$.
\end{definition}

Any cleft extension is a principal comodule algebra with
section $s\colon A\rightarrow B\otimes A$, $a\mapsto
s(a)=a_0j^{-1}(a_1)\otimes j(a_2)$.
Let $B:=A^{\mathrm{co}H}\subseteq A$ be a Hopf--Galois extension.
If $\Bbbk$ is a field, and recalling that we always consider Hopf algebras $H$
with invertible antipode,
the right $H$-equivariant projectivity of $A$ is equivalent 
to
the existence of a strong connection,  it is also equivalent to
faithfully flatness of $A$ as a left $B$-module, see e.g. \cite{BRZ} Part VII, Thm. 6.16, 6.19, 6.20. 
For faithfully flat Hopf--Galois extensions we will later use the following
main result.

\begin{theorem}[Schneider, \cite{schn}~Thm.~1] \label{ta-schn}
  Let $B:=A^{\coi H}\subseteq A$ be a faithfully flat Hopf--Galois extension.  We have the
  equivalence of categories:
  $$
  \begin{array}{c}
    \Phi:{}_B\mathcal{M} \lra {}_A\mathcal{M}^H, \quad \Phi(M)=A\otimes_BM; \qquad
    \Psi:{}_A\mathcal{M}^H \lra {}_B\mathcal{M}, \qquad \Psi(N)=N^{\coi H}
\end{array}
  $$
  where ${}_A\mathcal{M}^H$ denotes the category of
  right $H$-covariant left $A$-modules and ${}_B\mathcal{M}$ the category of left $B$-modules.
 \end{theorem}

\subsubsection{Smash products}\label{sec:sp}

A \textit{left $H$-module algebra} $(B,\rhd)$ is a left $H$-module
structure $\rhd\colon H\otimes B\rightarrow B$ on an algebra $B$ such that
$h\rhd(bb')=(h_1\rhd b)(h_2\rhd b')$ and $h\rhd 1_B=\epsilon(h)1_B$ for all
$h\in H$ and $b,b'\in B$.

\begin{definition}
For a left $H$-module algebra $(B,\rhd)$
the \textit{smash product} $B\# H$ is defined as
the $\Bbbk$-module $B\otimes H$ endowed with
the multiplication
\begin{equation*}
    (b\otimes h)\cdot_{\#}(b'\otimes h')
    :=b(h_1\rhd b')\otimes h_2h'
\end{equation*}
for all $b,b'\in B$ and $h,h'\in H$.
\end{definition}
In the following we write
$b\# h$ instead of $b\otimes h$ if we want to view $b\otimes h$ as an element of
$B\# H$. Moreover, we set $(b\# h)(b'\# h')
:=(b\otimes h)\cdot_{\#}(b'\otimes h')$.
The algebra $B \# H$ is naturally endowed with the right $H$-comodule structure
\begin{equation*}
    B\# H\to  B\# H\otimes H~,~~b\# h\mapsto(b\# h_1)\otimes h_2~.
\end{equation*}
The compatibility of this coaction with the product $\cdot_\#$ is
easily checked and thus $B\# H$
is a right $H$-comodule algebra (see also \cite{mo}~Sect.~4.1).
The smash product algebra $ B\# H$ is the trivial Hopf--Galois
  extension $B=(B\# H)^{\coi H}\subseteq B\# H$ with cleaving map
defined by the algebra inclusion $H\to B\# H$. Vice versa trivial
Hopf--Galois extensions are isomorphic (as $H$-comodule algebras) to smashed products. 

For a trivial Hopf Galois extension $B=A^{co H}\subseteq A$, 
conjugation via the cleaving map $j\colon H\rightarrow A$ defines the
left $H$-module algebra action on $B$
\begin{equation*}
    h\rhd b:=j(h_1)bj^{-1}(h_2)~.
  \end{equation*}
  This is easily verified using that $j$ is an algebra map
and that its convolution inverse $j^{-1}=j\circ S$ is an anti-algebra map.
Furthermore, $j(h_1)bj^{-1}(h_2)\in B$ follows from the right $H$-colinearity
of $j$ and Lemma~\ref{lemma01}. The compatibility of $\rhd$ with the algebra
structure holds because $j^{-1}$ is the convolution inverse of $j$ and $j$ is
unital.
\begin{proposition}\label{prop05}
If $B:=A^{\mathrm{co}H}\subseteq A$ is a trivial extension the map
$\theta\colon B\# H\rightarrow A$, $\theta(b\# h):=bj(h)$  is an
isomorphism of right $H$-comodule algebras.
The inverse is
$\theta^{-1}(a)=a_0j^{-1}(a_1)\# a_2$.
\end{proposition}
\begin{proof}
In Proposition~\ref{prop04} it has been proven that $\theta$ is an
isomorphism of right $H$-comodules. It remains to prove that $\theta$ is
an algebra homomorphism. Trivially we have $\theta(1_B\# 1_H)=1_A$ and also
\begin{align*}
    \theta((b\# h)(b'\# h'))
    &=\theta(b(h_1\rhd b')\# h_2h')
    =b(h_1\rhd b')j(h_2h')
    =bj(h_1)b'j^{-1}(h_2)j(h_3)j(h')\\
    &=bj(h)b'j(h')
    =\theta(b\# h)\theta(b'\# h'),
\end{align*}
which concludes the proof.
\end{proof}

\section{Covariant differential calculi on Hopf--Galois extensions}\label{diff-sec}

We begin this section with the definition of the category of noncommutative differential calculi
and discuss pullback and quotient calculi. 
We then specialize to covariant differential calculi and
bicovariant calculi, building on previous works of \cite{herm} and \cite{wor}.
Examples of (bi)covariant calculi on the quantum groups $\cO_q(\mathrm{SL}_2(\mathbb{C}))$
and $\cO_q(\mathrm{GL}_2(\mathbb{C}))$ are recalled and the induced calculi on
their parabolic quantum subgroups are presented. 
We next revisit a result 
of Pflaum and Schauenburg  \cite{pflaum2}, where an $H$-module calculus and a
bicovariant calculus merge to the \textit{smash product calculus}: a
covariant calculus on the smash product algebra. This latter is a
trivial principal bundle and the smash product calculus is a 
FODC on it.

In Section \ref{pc-sec}
we show that for right $H$-covariant FODCi on principal comodule
algebras base forms are the intersection of horizontal and coinvariant
forms.   Furthermore, when the injection of horizontal forms into total
space forms is completed in an exact sequence with vertical forms,
we obtain a graded Hopf--Galois extension with compatible
differentials  on the total space algebra and on the quantum structure
group.

\subsection{Noncommutative differential calculi}

In this section we give the definition and some results on
noncommutative differential calculi. Though this material
is well established in the context of bicovariant calculi on
Hopf algebras, since we take a slightly more general point of view,
we prefer to review the main points. We refer the reader to
 \cite{wor}, \cite[Chpt. 12]{KSbook}, \cite[Chpt. 2]{bm} for more details.

We give the definition of a differential calculus.
\begin{definition}\label{def01}
  A \textit{first order differential calculus} (FODC) on a (noncommutative) algebra $A$ is a couple
  $(\Gamma,\mathrm{d})$, where
\begin{enumerate}
\item[i.)] $\Gamma$ is an $A$-bimodule,

\item[ii.)] $\mathrm{d}\colon A\rightarrow\Gamma$ is a $\Bbbk$-linear map which satisfies the
Leibniz rule $\mathrm{d}(ab)=(\mathrm{d}a)b+a\mathrm{d}b$ for all $a,b\in A$,

\item[iii.)]
  $\Gamma=A\mathrm{d}A:=\mathrm{span}_\Bbbk\{a\mathrm{d}b~|~a,b\in
  A\}$.
  \end{enumerate}
  We say that $\Gamma$ is generated (as a left
$A$-module) by exact forms and refer to iii.) as the {\it{surjectivity
property}} of the FODC.
A morphism between a FODC $(\Gamma,\mathrm{d})$ on $A$ and a FODC $(\Gamma',\mathrm{d}')$
on another algebra $A'$ is a couple $(\phi,\Phi)$, where 
$\phi\colon A\rightarrow A'$ is an algebra homomorphism and
$\Phi\colon\Gamma\rightarrow\Gamma'$ is a $\Bbbk$-linear map such that
\begin{equation*}
    \Phi(a\cdot\omega\cdot b)
    =\phi(a)\cdot'\Phi(\omega)\cdot'\phi(b)
  \end{equation*}
  and
  \begin{equation*}
    \Phi\circ\mathrm{d}=\mathrm{d}'\circ\phi
\end{equation*}
for all $a,b\in A, \omega\in \Gamma$. 
Two FODCi are called \textit{equivalent} if there is an isomorphism
of FODCi between them. 
\end{definition}
The classical example of a FODC comes from differential geometry. 
Given a smooth manifold $M$, the $C^\infty(M)$-bimodule $\Omega^1(M)$ of
differential $1$-forms on $M$ together with the de Rham differential is a FODC on
$C^\infty(M)$.
On any algebra $A$ there is the so-called 
\textit{universal FODC} $(\Gamma_u,\mathrm{d}_u)$, where $\Gamma_u:=\ker\mu_A\subseteq A\otimes A$ is
the kernel of the multiplication $\mu_A\colon A\otimes A\rightarrow A$ and
$\mathrm{d}_ua:=1\otimes a-a\otimes 1$ for all $a\in A$. The left and right $A$-module
action on $\Gamma_u$ is given by the multiplication on the first and the
second tensor factor, respectively.  This FODC is universal in the sense that every FODC on
$A$ is isomorphic to a quotient of $(\Gamma_u,\mathrm{d}_u)$.
(c.f. e.g. \cite{wor}).

It is well-known that the tensor product $A\otimes A'$ of two
algebras $A$ and $A'$ is an algebra with associative product $(a\otimes a')(b\otimes b')
=ab\otimes a'b'$ and unit $1_A\otimes 1_{A'}$. This construction extends to the level
of FODCi (see also \cite{pflaum2} Thm.~2.2).

\begin{proposition}\label{prop01}
Given a FODC $(\Gamma,\mathrm{d})$ on $A$ and a FODC $(\Gamma',\mathrm{d}')$ on
$A'$, there is a FODC $(\Gamma_{A\otimes A'},\mathrm{d}_{A\otimes A'})$ on
$A\otimes A'$, where $\Gamma_{A\otimes A'}=\Gamma\otimes A'\oplus A\otimes\Gamma'$ and
\begin{equation}\label{eq01}
    \mathrm{d}_{A\otimes A'}\colon A\otimes A'\to \Gamma_{A\otimes A'}~,~~a\otimes a'\mapsto
    \mathrm{d}a\otimes a'+a\otimes\mathrm{d}'a'~.
\end{equation}
The $A\otimes A'$-bimodule structure on $\Gamma_{A\otimes A'}$ is
\begin{equation}\label{eq02}
    (a\otimes a')\cdot(\omega\otimes b'+b\otimes\omega')\cdot(c\otimes c')
    =a\omega c\otimes a'b'c'+abc\otimes a'\omega'c',
\end{equation}
where $a,b,c\in A$, $a',b',c'\in A'$, $\omega\in\Gamma$ and $\omega'\in\Gamma'$.
This construction is associative, i.e. $(\Gamma_{(A\otimes A')\otimes A''},
\mathrm{d}_{(A\otimes A')\otimes A''})=(\Gamma_{A\otimes(A'\otimes A'')},
\mathrm{d}_{A\otimes(A'\otimes A'')})$ for another algebra $A''$.
\end{proposition}
\begin{proof}
Clearly, (\ref{eq02}) defines an $A\otimes A'$-bimodule structure and (\ref{eq01})  a
$\Bbbk$-linear map. This latter satisfies the Leibniz rule
\begin{align*}
    \mathrm{d}_{A\otimes A'}((a\otimes a')(b\otimes b'))
    &=\mathrm{d}_{A\otimes A'}(ab\otimes a'b')\\
    &=\mathrm{d}(ab)\otimes a'b'+ab\otimes\mathrm{d}'(a'b')\\
    &=(\mathrm{d}a)b\otimes a'b'+a\mathrm{d}b\otimes a'b'
    +ab\otimes(\mathrm{d}'a')b'+ab\otimes a'\mathrm{d}'b'\\
    &=(\mathrm{d}a\otimes a'+a\otimes\mathrm{d}'a')\cdot(b\otimes b')
    +(a\otimes a')\cdot(\mathrm{d}b\otimes b'+b\otimes\mathrm{d}'b')\\
    &=\mathrm{d}_{A\otimes A'}(a\otimes a')\cdot(b\otimes b')
    +(a\otimes a')\cdot\mathrm{d}_{A\otimes A'}(b\otimes b')
\end{align*}
for all $a,b\in A$ and $a',b'\in A'$. Furthermore, $\Gamma_{A\otimes A'}$ is generated
by $A\otimes A'$ and $\mathrm{d}_{A\otimes A'}$, since
\begin{align*}
    (A\otimes A')\cdot\mathrm{d}_{A\otimes A'}(A\otimes A')
    &=(A\otimes A')\cdot(\mathrm{d}A\otimes A'\oplus A\otimes\mathrm{d}'A')\\
    &=A\mathrm{d}A\otimes A'\oplus A\otimes A'\mathrm{d}'A'\\
    &=\Gamma\otimes A'\oplus A\otimes\Gamma'\\
    &=\Gamma_{A\otimes A'}
\end{align*}
as an equality of sets. Finally, associativity follows from the
equality of $A\otimes A'\otimes A''$-bimodules
\begin{align*}
    \Gamma_{(A\otimes A')\otimes A''}
    &=\Gamma_{A\otimes A'}\otimes A''
    \oplus(A\otimes A')\otimes\Gamma''\\
    &=(\Gamma\otimes A'\oplus A\otimes\Gamma')\otimes A''
    \, \oplus \, A\otimes A'\otimes\Gamma''\\
    &=\Gamma\otimes A'\otimes A''
  \,  \oplus\, A\otimes \Gamma'\otimes A''\,\oplus \,A\otimes A'\otimes\Gamma''\\
    &=\Gamma\otimes(A'\otimes A'')
    \oplus A\otimes\Gamma_{A'\otimes A''}\\
    &=\Gamma_{A\otimes(A'\otimes A'')}
\end{align*}
and from $\mathrm{d}_{(A\oplus A')\oplus A''}= \mathrm{d}_{A\oplus( A'\oplus A'')}$. 
\end{proof}

\begin{remark}\label{rem01}
In other words, Proposition~\ref{prop01} proves that FODCi form a monoidal category.
The monoidal product is
  $$  (\Gamma,\mathrm{d},A)\otimes(\Gamma',\mathrm{d}',A')
    :=(\Gamma_{A\otimes A'},\mathrm{d}_{A\otimes A'},A\otimes A')
    $$
 and the monoidal unit is the trivial calculus $(\Gamma=\{0\},\mathrm{d}=0)$ on the
algebra $A=\Bbbk$.
\end{remark}

Given a FODC on $A$ we can induce a FODC on any subalgebra and quotient algebra of $A$.
For a surjective algebra map $\pi\colon A\rightarrow A'$,
we identify $A'$ with $A/I$ where $I:=\ker\pi$ and denote by $[a]:=a+I$
the class of  $a\in A$.
\begin{proposition}\label{prop02}
Let $(\Gamma,\mathrm{d})$ be a FODC on $A$.
\begin{enumerate}
\item[i.)] An algebra map $\iota\colon A'\rightarrow A$ induces a FODC
$(\Gamma_\iota,\mathrm{d}_\iota)$ on $A'$, where
\begin{equation}\label{eq119}
    \Gamma_\iota:=\iota(A')\mathrm{d}\iota(A')\subseteq\Gamma
\end{equation}
and $\mathrm{d}_\iota:=\mathrm{d}\circ \iota \colon
A'\to\Gamma_\iota$, $ a'\mapsto\mathrm{d}\iota(a')$.

\item[ii.)] A surjective algebra map $\pi\colon A\rightarrow A'$ induces a FODC
$(\Gamma_\pi,\mathrm{d}_\pi)$ on $A'$, where the $A'$-bimodule 
\begin{equation*}
    \Gamma_\pi:=\Gamma/\Gamma_I
\end{equation*}
is the quotient with $\Gamma_I:=I\mathrm{d}A+A\mathrm{d}I$, where $I:=\ker\pi\subseteq A$ and
$\mathrm{d}_\pi\colon A'\to\Gamma_\pi, \pi(a)\mapsto [\mathrm{d}a]$.

\item[iii.)] If the exact sequence of algebras $0\to \ker\pi\to
  A\stackrel{\pi}\rightarrow A'\to 0$ splits via $\iota\colon
  A'\rightarrow A$, so that  $\pi \circ\iota=\mathrm{id}_{A'}$,
 then $(\Gamma_\iota,\mathrm{d}_\iota)$
is equivalent to $(\Gamma_\pi,\mathrm{d}_\pi)$ with isomorphism $(\phi,\Phi)$ given by

\begin{equation}\label{eq121}
\phi=\mathrm{id}_{A'}~~,~~~~    \Phi\colon\Gamma_\iota\to \Gamma_\pi\,,~\omega\mapsto
    [\omega]~.
\end{equation}
\end{enumerate}
\end{proposition}
\begin{proof}
i.) The $\Bbbk$-submodule $\Gamma_\iota$ in (\ref{eq119}) is structured as an $A'$-bimodule via
\begin{equation}\label{eq120'}
    a'\cdot\omega\cdot b':=\iota(a')\omega\iota(b')
\end{equation}
for all $a',b'\in A'$ and $\omega\in\Gamma_\iota\subseteq\Gamma$. To see this, we first
verify that the maps defined in  (\ref{eq120'}) close in $\Gamma_\iota$. By definition,
any $\omega\in\Gamma_\iota$ is a finite sum $\omega=\iota(a'^i)\mathrm{d}\iota(b'_i)$,
for $a'^i,b'_i\in A'$. Then, for all $a',b'\in A'$,
\begin{align*}
    a'\cdot\omega\cdot b'
    &=\iota(a')\iota(a'^i)(\mathrm{d}\iota(b'_i))\iota(b')\\
    &=\iota(a'a'^i)(\mathrm{d}(\iota(b'_i)\iota(b'))-\iota(b'_i)\mathrm{d}\iota(b'))\\
    &=\iota(a'a'^i)\mathrm{d}\iota(b'_ib')-\iota(a'a'^ib'_i)\mathrm{d}\iota(b')
    \in\Gamma_\iota,
\end{align*}
where we used that $\iota$ is an algebra map. The maps in (\ref{eq120'}) are
left and right $A'$-actions because $\iota$ is an algebra map. They
are trivially commuting. Next, the $\Bbbk$-linear map
$\mathrm{d}_\iota$  satisfies the Leibniz rule
\begin{align*}
    \mathrm{d}_\iota(a'b')
    =\mathrm{d}\iota(a'b')
    =\mathrm{d}(\iota(a')\iota(b'))
    =(\mathrm{d}\iota(a'))\iota(b')+\iota(a')\mathrm{d}\iota(b')
    =(\mathrm{d}_\iota a')\cdot b'+a'\cdot\mathrm{d}_\iota b'
\end{align*}
for all $a',b'\in A'$. Since by definition of
$\Gamma_\iota$ any $\omega\in\Gamma_\iota$ is of the form
$\omega=\iota(a'^i)\mathrm{d}\iota(b'_i)$ for some $a'^i,b'_i\in A'$
(finite sum understood), we have
$\omega=\iota(a'^i)\mathrm{d}\iota(b'_i)=a'^i\cdot\mathrm{d}_\iota b'_i$, proving that
$(\Gamma_\iota,\mathrm{d}_\iota)$ is a FODC on $A'$.

ii.) Since $I=\ker\pi\subseteq A$ is an ideal, by the
Leibniz rule of $\mathrm{d}$ it immediately follows that $\Gamma_I
=I\mathrm{d}A+A\mathrm{d}I\subseteq\Gamma$ is an $A$-subbimodule. Hence,
$\Gamma/\Gamma_I$ is an $A$-bimodule. Since $I\cdot\Gamma\subseteq
\Gamma_I$, $\Gamma\cdot I\subseteq
\Gamma_I$, the $A$-actions on $\Gamma$ descend to  $A'=A/I$-actions on $\Gamma/\Gamma_I$,
\begin{equation*}
    \pi(a)\cdot[\omega]\cdot\pi(b)
    :=[a\omega b]
\end{equation*}
for all $a,b\in A$ and $\omega\in\Gamma$.
It follows that the map $\mathrm{d}$ induced on the quotient,
$\mathrm{d}_\pi\colon A'\to \Gamma_\pi, \pi(a)\mapsto[\mathrm{d}a]$,
is $\Bbbk$-linear and satisfies the Leibniz rule:
\begin{align*}
    \mathrm{d}_\pi(\pi(a)\pi(b))
    &=\mathrm{d}_\pi(\pi(ab))
    =[\mathrm{d}(ab)]
    =[(\mathrm{d}a)b+a\mathrm{d}b]\\
    &=[\mathrm{d}a]\cdot\pi(b)+\pi(a)\cdot[\mathrm{d}b]
    =(\mathrm{d}_\pi\pi(a))\cdot\pi(b)+\pi(a)\cdot\mathrm{d}_\pi(\pi(b))
\end{align*}
for all $a,b\in A$. Finally,  since
any $\omega\in\Gamma$ is of the form  $\omega=a^i\mathrm{d}b_i$ (sum understood)
we have
$[\omega]=[a^i\mathrm{d}b_i]=\pi(a^i)\cdot\mathrm{d}_\pi(\pi(b_i))$ so
that  $\Gamma_\pi= A'\mathrm{d}_\pi A'$. 
We conclude that $(\Gamma_\pi,\mathrm{d}_\pi)$ is a FODC on $A'$.

iii.) According to i.) and ii.) $(\Gamma_\iota,\mathrm{d}_\iota)$ and
$(\Gamma_\pi,\mathrm{d}_\pi)$ are FODCi on $A'$. The splitting
$\iota:A'\to A$ of the
exact sequence $0\to \ker\pi\to
  A\stackrel{\pi}\rightarrow A'\to 0$ implies the direct sum of $A'$-bimodules
$A=\iota(A')\oplus I$ where the $A'$-action is via the algebra
embedding $A'\to \iota(A')\subseteq A$. This induces the direct sum of
$A'$-bimodules
$$\Gamma=A\mathrm{d} A=\iota(A')\mathrm{d}\iota(A')\oplus
\big(I\mathrm{d}A+A\mathrm{d}I\big)=\Gamma_\iota\oplus\Gamma_I~.$$
Considering the quotient with respect to $\Gamma_I$ gives the isomorphism
$\Gamma_\iota\to \Gamma/\Gamma_I$, $\omega\mapsto
\pi(\omega)$, which equals the $\Phi$ morphism defined in 
(\ref{eq121}).
We are left to prove the compatibility of $(\phi=\mathrm{id}_{A'},\Phi)$ with the exterior derivatives, for all $a'\in A'$
$$\Phi( \mathrm{d}_\iota a')= \Phi( \mathrm{d} \iota (a'))=[ \mathrm{d} \iota (a')]=\mathrm{d}_\pi\pi(\iota(a'))= \mathrm{d}_\pi a'=\mathrm{d}_\pi \phi(a')~.$$
\end{proof}
We call $(\Gamma_\iota,\mathrm{d}_\iota)$ the {\it pullback calculus}, while we call
$(\Gamma_\pi,\mathrm{d}_\pi)$ the {\it quotient calculus}. 
The above proposition provides a useful tool to produce examples of
noncommutative differential calculi. In the context of Drinfel'd twist
deformation quantization it has been employed to describe noncommutative
differential geometry on submanifolds of $\mathbb{R}^n$ given by
generators and relations \cite{FioreWeber}.
More abstractly, the braided Cartan calculus of a braided commutative
algebra with triangular Hopf algebra symmetry is related to the braided
Cartan calculus on a submanifold algebra in the above sense \cite{Weber}.

\subsection{Covariant calculi on comodule algebras}

In this section we discuss the theory for covariant calculi
on comodule algebras, following the perspectives of
\cite{herm,wor} and \cite{pflaum2}.
Let $H$ be a Hopf algebra.
\begin{definition}\label{rHc}
\begin{enumerate}
\item[i.)] A FODC $(\Gamma,\mathrm{d})$ on a right $H$-comodule algebra $(A,\delta_R)$
is said to be \textit{right $H$-covariant} if 
\begin{equation}\label{eq08}
    a\mathrm{d}b\mapsto a_0\mathrm{d}b_0\otimes a_1b_1,
\end{equation}
for all $a,b\in A$,
extends to a well-defined $\Bbbk$-linear map 
$\Gamma\rightarrow\Gamma\otimes H$. 

\item[ii.)] A FODC $(\Gamma,\mathrm{d})$ on a left $H$-comodule algebra $(A,\delta_L)$
is said to be \textit{left $H$-covariant} if 
\begin{equation}\label{eq08'}
    a\mathrm{d}b\mapsto a_{-1}ab_{-1}\otimes a_0\mathrm{d}b_0,
\end{equation}
for all $a,b\in A$,
extends to a well-defined $\Bbbk$-linear map 
$\Gamma\rightarrow H\otimes\Gamma$. 

\item[iii.)] A FODC $(\Gamma,\mathrm{d})$ on an $H$-bicomodule algebra $(A,\delta_R,\delta_L)$
is said to be \textit{$H$-bicovariant} if 
it is right and left $H$-covariant.
\end{enumerate}
\end{definition}

If $A=H$ is an $H$-comodule algebra with respect to the coproduct
we call an $H$-covariant FODC simply covariant. We recall the
following proposition.

\begin{proposition}[\cite{wor}]\label{lemma03}
\begin{enumerate}
\item[i.)] A FODC $(\Gamma,\mathrm{d})$ on a right $H$-comodule algebra
$(A,\delta_R)$ is right $H$-covariant if and only if $(\Gamma,\Delta_R)$ is a 
right $H$-covariant $A$-bimodule and $\mathrm{d}$ is right $H$-colinear.
In this case $\Delta_R$ is determined by (\ref{eq08}).

\item[ii.)] A FODC $(\Gamma,\mathrm{d})$ on a left $H$-comodule algebra 
$(A,\delta_L)$ is left $H$-covariant if and only if 
$(\Gamma,\Delta_L)$ is a left $H$-covariant $A$-bimodule and $\mathrm{d}$ is
left $H$-colinear. In this case $\Delta_L$ is determined by (\ref{eq08'}).

\item[iii.)] A FODC $(\Gamma,\mathrm{d})$ on an $H$-bicomodule algebra
$(A,\delta_R,\delta_L)$ is an $H$-bicovariant FODC if and only if 
$(\Gamma,\Delta_R,\Delta_L)$ is an $H$-bicovariant $A$-bimodule and $\mathrm{d}$ is
$H$-bicolinear.
\end{enumerate}
\end{proposition}

We define a morphism of right/left/bi $H$-covariant
FODCi $(\Gamma,\mathrm{d})$ and $(\Gamma',\mathrm{d}')$ on right/left/bi $H$-comodule
algebras $A,A'$ to be a morphism $(\phi,\Phi)$ between FODCi such that $\phi$ and $\Phi$
are right/left/bi $H$-colinear maps. Since $\Phi$, if it exists, is  determined by
$\phi$ via $\Phi(a^i\mathrm{d}b_i)=\phi(a^i)\mathrm{d}\phi(b_i)$ for
$a^i, b_i\in A$ (finite sum understood), we see that the FODC morphism
$(\phi,\Phi)$ is  right/left/bi $H$-colinear if and only if
$\phi:A\to A'$ is  right/left/bi $H$-colinear, indeed, for all $a\in A$,
$$\Delta_R(\Phi(\mathrm{d}a))=\Delta_R\mathrm{d}\phi(a)=\mathrm{d}\phi(a)_0\otimes
\phi(a)_1=\mathrm{d}\phi(a_0)\otimes
a_1=\Phi(\mathrm{d}a_0)\otimes a_1=(\Phi\otimes \mathrm{id}_H)\Delta_R(\mathrm{d}a)~.$$
The natural question arises if
Proposition~\ref{prop02} generalizes to the $H$-covariant setting.
\begin{proposition}\label{prop06}
Let $(\Gamma,\mathrm{d})$ be a right $H$-covariant FODC on a right $H$-comodule
algebra $A$.
\begin{enumerate}
\item[i.)] If $\iota\colon A'\rightarrow A$ is a right $H$-comodule
algebra homomorphism the pullback calculus 
$(\Gamma_\iota,\mathrm{d}_\iota)$ is a right $H$-covariant FODC on $A'$.

\item[ii.)] If $\pi\colon A\rightarrow A'$ is a surjective right $H$-comodule
algebra homomorphism the quotient calculus $(\Gamma_\pi,\mathrm{d}_\pi)$ is 
a right $H$-covariant FODC on $A'$.

\item[iii.)] 
 If the exact sequence of $H$-comodule algebras $0\to \ker\pi\to
  A\stackrel{\pi}\rightarrow A'\to 0$ splits via $\iota\colon
  A'\rightarrow A$, then $(\Gamma_\iota,\mathrm{d}_\iota)$
is equivalent to $(\Gamma_\pi,\mathrm{d}_\pi)$ as a right $H$-covariant FODC.
\end{enumerate}
Analogous statements hold for left $H$-covariant and $H$-bicovariant FODCi on
left $H$-comodule algebras and $H$-bicomodule algebras, respectively.
\end{proposition}
\begin{proof}
From Proposition~\ref{prop02} it follows that $(\Gamma_\iota,\mathrm{d}_\iota)$
and $(\Gamma_\pi,\mathrm{d}_\pi)$ are FODCi on $A'$. We prove $i.)$ by observing 
that for all $a^i,b^i\in A$ (and understanding finite sums on the
indices $i$)
$$
\iota(a^i)_0\mathrm{d}(\iota(b^i)_0)\otimes\iota(a^i)_1\iota(b^i)_1
=\iota(a^i_0)\mathrm{d}(\iota(b^i_0))\otimes a^i_1b^i_1
\in\Gamma_\iota\otimes H
$$
by the right $H$-colinearity of $\iota$. Thus, the right $H$-coaction
$\Delta_R\colon\Gamma\rightarrow\Gamma\otimes H$ restricts to a $\Bbbk$-linear map
$\Delta_R|_{\Gamma_\iota}\colon\Gamma_\iota\rightarrow\Gamma_\iota\otimes H$
and consequently $(\Gamma_\iota,\mathrm{d}_\iota)$ is a
right $H$-covariant FODC. For $ii.)$ we first note that
$I=\ker\pi\subseteq A$ satisfies $\delta_R(I)\subseteq I\otimes H$ since
$\pi$ is right $H$-colinear. Then, by the right $H$-covariance of
$(\Gamma,\mathrm{d})$ we have that $\Gamma_I\subseteq \Gamma$ is a
right $H$-subcomodule:
$\Delta_R(\Gamma_I)\subseteq\Gamma_I\otimes H$, so that $\Delta_R\colon\Gamma\rightarrow\Gamma
\otimes H$ induces  a well-defined right $H$-coaction on $\Gamma_\pi$,
which reads
$$
\Delta_R([a\mathrm{d}b])=[a_0\mathrm{d}b_0]\otimes a_1ab_1
=\pi(a_0)\cdot\mathrm{d}_\pi\pi(b_0)\otimes a_1b_1
$$
for all $a,b\in A$. This shows that $(\Gamma_\pi,\mathrm{d}_\pi)$ is
right $H$-covariant.
The third statement trivially holds recalling that for the morphism $(\phi,\Phi)$ of FODCi $\Phi$ is
$H$-colinear if $\phi$ is $H$-colinear, which is the case since
$\phi=\mathrm{id}_{A'}$.
\end{proof}

We have the following corollary  of Proposition \eqref{prop06} ii.).
\begin{corollary}\label{corollary06}
Let $\pi\colon A\to H$ be a Hopf algebra epimorphism. Any right/left/bicovariant
FODC $(\Gamma,\mathrm{d})$ on the Hopf algebra $A$ 
gives a right/left $H$-covariant or $H$-bicovariant FODC on $A$ 
and projects to a
right/left/bicovariant FODC $(\Gamma_\pi,\mathrm{d}_\pi)$ on the Hopf
algebra $H$.
\end{corollary}
\begin{proof}
Here we prove the case of a bicovariant FODC $(\Gamma,\mathrm{d})$ on $A$.
The others are similarly proven.
The left and right $H$-coactions on $A$ defined by
  \begin{equation*}
    \begin{split}
  \delta_L&:=(\pi\otimes\mathrm{id}_A)\circ\Delta
    \colon A\to  H\otimes A\,,~ a\mapsto a_{-1}\otimes a_0:=[a_1]\otimes a_2\\[.2em]
    \delta_R&:=(\mathrm{id}_A\otimes\pi)\circ\Delta
    \colon A\to  A\otimes H\,,~ a\mapsto a_0\otimes a_1:=a_1\otimes[a_2]
  \end{split}
\end{equation*}
and those on $\Gamma$ defined by
\begin{equation*}
  \begin{split}
    \Delta^H_L:=(\pi\otimes
    \mathrm{id}_A)\circ\Delta_L&\colon\Gamma\to H\otimes \Gamma\,,
     ~ a\mathrm{d}b\mapsto a_{-1}b_{-1}\otimes a_0\mathrm{d}b_0
    \\[.2em]
  \Delta^H_R:=(\mathrm{id}_A\otimes\pi)    \circ \Delta_R&\colon\Gamma\to
  \Gamma\otimes H\,,
  ~ a\mathrm{d}b\mapsto a_0\mathrm{d}b_0\otimes a_1b_1~,\\
\end{split}
\end{equation*}
where $a,b\in A$,
structure $(\Gamma, \mathrm{d}_\pi)$ as an $H$-bicovariant FODC on
$A$.
The projection $\pi:A\to H$ is left and right $H$-colinear so that from 
Proposition \eqref{prop06}  the
induced FODC $(\Gamma_\pi, \mathrm{d}_\pi)$ on $H$ is $H$-bicovariant
with $H$-coaction given by the $H$-coproduct. 
\end{proof}

\subsection{Examples of covariant calculi}\label{ExSection}

In this section we describe covariant calculi on the quantum groups
$\mathcal{O}_q(\mathrm{GL}_2(\mathbb{C}))$,
$\mathcal{O}_q(\mathrm{SL}_2(\mathbb{C}))$
and on their parabolic subalgebras.

\begin{example}[\textit{4D calculus on $\mathcal{O}_q(\mathrm{GL}_2(\mathbb{C}))$}]\label{eq04}
Let $q\in\mathbb{C}$ be a non-zero complex number which is not a root of unity. 
The free $\mathbb{C}$-algebra $\mathbb{C}\langle\alpha,\beta,
\gamma,\delta\rangle$ generated by indeterminates $\alpha,\beta,\gamma,
\delta$ modulo the relations
\begin{equation}\label{qslq2}
\begin{array}{c}
\ap\bt=q^{-1}\bt\ap, \quad \ap\cm=q^{-1}\cm\ap, \quad \bt\dt=q^{-1}\dt\bt, \quad \cm\dt=q^{-1}\dt\cm, \\ \\
\bt\cm=\cm\bt, \qquad \ap\dt-\dt\ap=(q^{-1}-q)\bt\cm 
\end{array}
\end{equation}
is denoted by $\mathcal{O}_q(\mathrm{M}_2(\mathbb{C})):=\mathbb{C}\langle\alpha,\beta,
\gamma,\delta\rangle/I_M$, where $I_M$ is the ideal in $\mathbb{C}\langle\alpha,\beta,
\gamma,\delta\rangle$ generated by the (\ref{qslq2}). The quotient algebra
$\mathcal{O}_q(\mathrm{M}_2(\mathbb{C}))$ is a bialgebra with coalgebra structure
\begin{equation}\label{eq135}
    \Delta
    \begin{pmatrix}
    \alpha & \beta\\
    \gamma & \delta
    \end{pmatrix}
    =\begin{pmatrix}
    \alpha & \beta\\
    \gamma & \delta
    \end{pmatrix}\dot\otimes
    \begin{pmatrix}
    \alpha & \beta\\
    \gamma & \delta
    \end{pmatrix},~
    \epsilon\begin{pmatrix}
    \alpha & \beta\\
    \gamma & \delta
    \end{pmatrix}
    =\begin{pmatrix}
    1 & 0\\
    0 & 1
    \end{pmatrix},
\end{equation}
where $\dot\otimes$ denotes the tensor product and matrix multiplication. Furthermore,
there is a central element $\mathrm{det}_q:=\alpha\delta-q^{-1}\beta\gamma
\in\mathcal{O}_q(\mathrm{M}_2(\mathbb{C}))$, satisfying
$\Delta(\mathrm{det}_q)=\mathrm{det}_q\otimes\mathrm{det}_q$ and $\epsilon(\mathrm{det}_q)=1$.
On the quotient algebra $\mathcal{O}_q(\mathrm{GL}_2(\mathbb{C})):=
\mathcal{O}_q(\mathrm{M}_2(\mathbb{C}))[r]/\langle r~\mathrm{det}_q-1\rangle$
we induce a Hopf algebra
structure by (\ref{eq135}) and $\Delta(r)=r\otimes r$, $\epsilon(r)=1$, together with the
antipode
\begin{equation}\label{eq136}
    S\begin{pmatrix}
    \alpha & \beta\\
    \gamma & \delta
    \end{pmatrix}
    =r\begin{pmatrix}
    \delta & -q\beta\\
    -q^{-1}\gamma & \alpha
    \end{pmatrix},~~~~~~~~~~
    S(r)=\mathrm{det}_q.
\end{equation}
On the quantum group $A:=\mathcal{O}_q(\mathrm{GL}_2(\mathbb{C}))$ there is a
$4$-dimensional bicovariant FODC
$(\Gamma_{\mathrm{GL}},\mathrm{d}_{\mathrm{GL}})$ which we are
going to describe following the explicit exposition of
\cite{AschieriCastellani}\footnote{Note that in our
  convention we used $q^{-1}$ instead of $q$ and we rescaled
  $\omega^2,\omega^3$ by $q^{-1}$ for later use in Example \ref{sl2}}.
$\Gamma$ is a free left $A$-module generated by a basis 
$\{\omega^1,\omega^2,\omega^3,\omega^4\}$ of left coinvariant $1$-forms obeying the 
commutation relations
\begin{equation}\label{commrel}
\begin{split}
    \omega^1\alpha&=q\alpha\omega^1,\\
    \omega^2\alpha&=\alpha\omega^2,\\
    \omega^3\alpha&=-q^{-1}\lambda\beta\omega^1+\alpha\omega^3,\\
    \omega^4\alpha&=-\lambda\beta\omega^2+q^{-1}\alpha\omega^4,\\
    \omega^1\gamma&=q\gamma\omega^1,\\
    \omega^2\gamma&=\gamma\omega^2,\\
    \omega^3\gamma&=\gamma\omega^3-q^{-1}\lambda\delta\omega^1,\\
    \omega^4\gamma&=-\lambda\delta\omega^2+q^{-1}\gamma\omega^4,
\end{split}
\hspace{1,5cm}
\begin{split}
    \omega^1\beta&=q^{-1}\beta\omega^1,\\
    \omega^2\beta&=-q^{-1}\lambda\alpha\omega^1+\beta\omega^2,\\
    \omega^3\beta&=\beta\omega^3,\\
    \omega^4\beta&=q^{-1}\lambda^2\beta\omega^1-\lambda\alpha\omega^3
    +q\beta\omega^4,\\
    \omega^1\delta&=q^{-1}\delta\omega^1,\\
    \omega^2\delta&=\delta\omega^2-q^{-1}\lambda\gamma\omega^1,\\
    \omega^3\delta&=\delta\omega^3,\\
    \omega^4\delta&=q^{-1}\lambda^2\delta\omega^1-\lambda\gamma\omega^3
    +q\delta\omega^4
\end{split}
\end{equation}
with $\lambda:=q^{-1}-q$.
The differential is given in terms of this basis by
\begin{equation}\label{4D+}
\begin{array}{rcl}
\mathrm{d}_{\mathrm{GL}} \ap&=&\frac{q-1}{\lambda}\, \ap \omega^1 +\frac{q^{-1}-1}{\lambda}\,\ap \omega^4-\bt\omega^2 \\
\mathrm{d}_{\mathrm{GL}} \bt&=&Q \bt \omega^1 +\frac{q-1}{\lambda} \,\bt \omega^4-\ap\omega^3\\
\mathrm{d}_{\mathrm{GL}} \cm&=&\frac{q-1}{\lambda} \,\cm \omega^1 +\frac{q^{-1}-1}{\lambda}\, \cm \omega^4-\dt\omega^2 \\
\mathrm{d}_{\mathrm{GL}} \dt&=&Q \dt \omega^1 +\frac{q-1}{\lambda}\, \dt \omega^4-\cm\omega^3,
\end{array}
\end{equation}
where
$Q:=\frac{q^{-1}(\lambda^2+1)-1}{\lambda}=\frac{q^2+q+1}{q^2(q+1)}-1$. The
basis $1$-forms are then expressed explicitly as
\begin{equation}\label{eq138}
\begin{split}
    \omega^1&=Q'\bigg(
    (q^{-2}-q^{-1})(S(\alpha)\mathrm{d}_{\mathrm{GL}}\alpha
    +S(\beta)\mathrm{d}_{\mathrm{GL}}\gamma)
    +q^{-2}(q^{-1}-1)(S(\gamma)\mathrm{d}_{\mathrm{GL}}\delta
    +S(\delta)\mathrm{d}_{\mathrm{GL}}\delta)
    \bigg)\\
    \omega^2&=-q^2(S(\gamma)\mathrm{d}_{\mathrm{GL}}\alpha
    +S(\delta)\mathrm{d}_{\mathrm{GL}}\gamma)\\
    \omega^3&=-q^2(S(\alpha)\mathrm{d}_{\mathrm{GL}}\beta
    +S(\beta)\mathrm{d}_{\mathrm{GL}}\delta)\\
    \omega^4&=Q'\bigg(
    (q^{-5}-q^{-3}-q^{-2}+q^{-1})
    (S(\alpha)\mathrm{d}_{\mathrm{GL}}\alpha
    +S(\beta)\mathrm{d}_{\mathrm{GL}}\gamma)
    +(q^{-2}-q^{-1})(S(\gamma)\mathrm{d}_{\mathrm{GL}}\delta
    +S(\delta)\mathrm{d}_{\mathrm{GL}}\delta)
    \bigg),
\end{split}
\end{equation}
where $Q':=\frac{1}{q^{-5}-q^{-4}-q^{-2}+q^{-1}}$.

~\\

\noindent We now construct a bicovariant FODC
$(\Gamma_{P_{\mathrm{GL}}},\mathrm{d}_{P_{\mathrm{GL}}})$ on
$\mathcal{O}_q(P_{\mathrm{GL}}):=A/\langle\gamma\rangle$. This is done
via the Hopf algebra quotient $\pi_{P_{\mathrm{GL}}}\colon
A\rightarrow \mathcal{O}_q(P_{\mathrm{GL}}),~ a\mapsto
\pi_{P_{\mathrm{GL}}}(a)=:[a]$,  using Corollary~\ref{corollary06}. We recall that 
$\Gamma_{P_{\mathrm{GL}}}:=\Gamma_{\mathrm{GL}}/\Gamma_{I_\gamma}$,
where $\Gamma_{I_\gamma}:=A\cdot\mathrm{d}_{\mathrm{GL}}\langle\gamma\rangle
+\langle\gamma\rangle\mathrm{d}_{\mathrm{GL}}A\subseteq\Gamma_{\mathrm{GL}}$.
Writing $[\omega]$ for the equivalence class of
$\omega\in\Gamma_{\mathrm{GL}}$ on the quotient $\pi^\Gamma_{P_{\mathrm{GL}}}\colon\Gamma_{\mathrm{GL}}\rightarrow
\Gamma_{P_{\mathrm{GL}}}$, the
differential on $\Gamma_{P_{\mathrm{GL}}}$ is defined by
$\mathrm{d}_{P_{\mathrm{GL}}}\colon\mathcal{O}_q(P_{\mathrm{GL}})
\rightarrow\Gamma_{P_{\mathrm{GL}}}$, $\mathrm{d}_{P_{\mathrm{GL}}}[a]:=
[\mathrm{d}_{\mathrm{GL}}a]$ for all
$a\in\mathcal{O}_q(\mathrm{GL}_2(\mathbb{C}))$.
As a free left $H:=\mathcal{O}_q(P_{\mathrm{GL}})$-module
$\Gamma_{P_{\mathrm{GL}}}$ is $3$-dimensional with basis
$\{[\omega^1],[\omega^3],[\omega^4]\}$\footnote{Classically, the
dimension of $\Gamma_{P_\mathrm{GL}}$ is $3$, so its $q$-deformation is at least $3$-dimensional.  
From (\ref{eq138})
it immediately follows that $[\omega^2]=0$, which proves the claim.}.
Denoting the projected generators by  
\begin{equation*}
    \pi_{P_{\mathrm{GL}}}\begin{pmatrix}
    \alpha & \beta\\
    \gamma & \delta
    \end{pmatrix}=:
    \begin{pmatrix}
    t & p\\
    0 & s
    \end{pmatrix},~~~~~\pi_{\mathrm{GL}}(r)=:r',
\end{equation*}
the right $H$-action
and differential are explicitly given by
\begin{equation}\label{commrel2}
\begin{split}
    [\omega^1]t&=qt[\omega^1],\\
    [\omega^3]t&=-q^{-1}\lambda p[\omega^1]+t[\omega^3],\\
    [\omega^4]t&=q^{-1}t[\omega^4],
\end{split}
\hspace{0,5cm}
\begin{split}
    [\omega^1]p&=q^{-1}p[\omega^1],\\
    [\omega^3]p&=p[\omega^3],\\
    [\omega^4]p&=q^{-1}\lambda^2p[\omega^1]-\lambda t[\omega^3]
    +qp[\omega^4],
\end{split}
\hspace{0,5cm}
\begin{split}
    [\omega^1]s&=q^{-1}s[\omega^1],\\
    [\omega^3]s&=s[\omega^3],\\
    [\omega^4]s&=q^{-1}\lambda^2s[\omega^1]+qs[\omega^4]
\end{split}
\end{equation}
and
\begin{equation}\label{diff2}
    \mathrm{d}_{P_{\mathrm{GL}}}t=\frac{q-1}{\lambda}t[\omega^1]+\frac{q^{-1}-1}{\lambda}t[\omega^4],~~
    \mathrm{d}_{P_{\mathrm{GL}}}p=Qp[\omega^1]+\frac{q-1}{\lambda}p[\omega^4]-t[\omega^3],~~
    \mathrm{d}_{P_{\mathrm{GL}}}s=Qs[\omega^1]+\frac{q-1}{\lambda}s[\omega^4].
\end{equation}
\end{example}

\begin{example}[\textit{Bicovariant FODC on }$\mathcal{O}_q(\mathrm{SL}_2(\mathbb{C}))$ and its parabolic quotient]\label{sl2}
Consider the quotient Hopf algebra
$\cO_q(\rSL_2(\C)):=\mathcal{O}_q(\mathrm{GL}_2(\mathbb{C}))/\langle\mathrm{det}_q-1\rangle$
with coalgebra structure and antipode induced from (\ref{eq135}) and (\ref{eq136}).
We denote the associated 
Hopf algebra map projection by
$\mathrm{pr}_{\mathrm{SL}}\colon\mathcal{O}_q(\mathrm{GL}_2(\mathbb{C}))
\rightarrow\mathcal{O}_q(\mathrm{SL}_2(\mathbb{C}))$.
The quantum group $\cO_q(\rSL_2(\C))$
is the  Manin deformation (see \cite{ma1}) 
of the ring of algebraic functions on the
complex special linear group $\rSL_2(\C)$. 
It is the deformed algebra of regular functions on the complex
special linear group $\rSL_2(\C)$.
Let  $\cO_q(P)$ be the deformed algebra of functions on the Borel subgroup $P\subset 
\rSL_2(\C)$. We identify it with the quotient $\cO_q(\rSL_2(\C))/I_P$
where $I_P\subseteq\cO_q(\rSL_2(\C))$ is the Hopf ideal generated by the element $\cm$.
On generators the Hopf algebra projection is given by
\begin{equation}\label{quotient}
    \pi_P\colon\mathcal{O}_q(\mathrm{SL}_2(\mathbb{C}))\rightarrow\mathcal{O}_q(P),~~~~~
    \begin{pmatrix}
    \alpha & \beta\\
    \gamma & \delta
    \end{pmatrix}
    \mapsto\begin{pmatrix}
    t & p\\
    0 & t^{-1}
    \end{pmatrix}
\end{equation}
or, in other words, $\cO_q(P) =\cO_q(\rSL_2(\C))/I_P
=\C \langle t, t^{-1}, p \rangle /\langle tp-q^{-1}pt\rangle$.
Note that 
$\mathcal{O}_q(P)=\mathcal{O}_q(P_{\mathrm{GL}})/\langle\mathrm{det}_q-1\rangle
\cong\mathcal{O}_q(\mathrm{GL}_2(\mathbb{C}))/\langle\gamma,\mathrm{det}_q\rangle$
and the corresponding projection $\mathrm{pr}_P\colon
\mathcal{O}_q(P_{\mathrm{GL}})\rightarrow\mathcal{O}_q(P)$ makes the
diagram
\begin{equation}\label{eq137}
\begin{tikzcd}
\mathcal{O}_q(\mathrm{GL}_2(\mathbb{C}))
\ar{r}{\mathrm{pr}_{\mathrm{SL}}}
\ar{d}[swap]{\pi_{P_{\mathrm{GL}}}}
& \mathcal{O}_q(\mathrm{SL}_2(\mathbb{C}))
\ar{d}{\pi_P} \\
\mathcal{O}_q(P_{\mathrm{GL}})
\ar{r}{\mathrm{pr}_P}
& \mathcal{O}_q(P)
\end{tikzcd}
\end{equation}
commute.

There are bicovariant FODCi
$(\Gamma^+_{\mathrm{SL}},\mathrm{d}^+_{\mathrm{SL}})$ on 
$\mathcal{O}_q(\mathrm{SL}_2(\mathbb{C}))$
and $(\Gamma^+_P,\mathrm{d}_P^+)$
on $\mathcal{O}_q(P)$ induced from Example~\ref{eq04} as the quotient calculi
\begin{equation*}
\begin{tikzcd}
\Gamma_{\mathrm{GL}}
\arrow{rr}{\mathrm{pr}^\Gamma_{\mathrm{SL}}}
& & \Gamma^+_{\mathrm{SL}}\\
\mathcal{O}_q(\mathrm{GL}_2(\mathbb{C}))
\arrow{u}{\mathrm{d}_{\mathrm{GL}}}
\arrow{rr}{\mathrm{pr}_{\mathrm{SL}}}
& & \mathcal{O}_q(\mathrm{SL}_2(\mathbb{C}))
\arrow{u}[swap]{\mathrm{d}^+_{\mathrm{SL}}}
\end{tikzcd}
\hspace{1,5cm}
\begin{tikzcd}
\Gamma_{P_{\mathrm{GL}}}
\arrow{r}{\mathrm{pr}^\Gamma_P}
& \Gamma^+_P\\
\mathcal{O}_q(P_{\mathrm{GL}})
\arrow{u}{\mathrm{d}_{P_{\mathrm{GL}}}}
\arrow{r}{\mathrm{pr}_P}
& \mathcal{O}_q(P)
\arrow{u}[swap]{\mathrm{d}^+_P}~~~.
\end{tikzcd}
\end{equation*}
By the commutativity of the square diagram in (\ref{eq137}) and recalling that the modules of $1$-forms are generated by the algebras and differentials we obtain the commutative cube \begin{equation*}
\begin{tikzcd}[row sep={40,between origins}, column sep={40,between origins}]
      & \Gamma_{P_{\mathrm{GL}}} \ar{rr}{\mathrm{pr}^\Gamma_P} & & 
      \Gamma_P^+ \\
    \Gamma_{\mathrm{GL}}
    \ar[crossing over]{rr}{~\qquad\mathrm{pr}^\Gamma_{\mathrm{SL}}}
    \ar{ur}{\pi^\Gamma_{P_{\mathrm{GL}}}} & &
    \Gamma^+_{\mathrm{SL}}
    \ar{ur}{\pi^\Gamma_P} \\
      & \mathcal{O}_q(P_{\mathrm{GL}}) \ar{uu}
      \ar{rr}{~~~~~\mathrm{pr}_P} & & 
      \mathcal{O}_q(P) \vphantom{\times_{S_1}} \ar{uu}[swap]{\mathrm{d}^+_P} \\
    \mathcal{O}_q(\mathrm{GL}_2(\mathbb{C})) \ar{uu}{\mathrm{d}_{\mathrm{GL}}}
    \ar{rr}[swap]{\mathrm{pr}_{\mathrm{SL}}} \ar{ur}[swap]{\pi_{P_{\mathrm{GL}}}} & &
    \mathcal{O}_q(\mathrm{SL}_2(\mathbb{C})) 
    \ar[uu]
    \vphantom{\times_{S_1}} \ar{ur}[swap]{\pi_P}
\end{tikzcd}
\end{equation*}
To be more explicit, the bicovariant FODC
$(\Gamma^+_{\mathrm{SL}},\mathrm{d}^+_{\mathrm{SL}})$ on 
$\mathcal{O}_q(\mathrm{SL}_2(\mathbb{C}))$ is
$4$-dimensional with basis of left coinvariant $1$-forms $\{\omega^1,\omega^2,
\omega^3,\omega^4\}$. Up to the identification $\mathrm{det}_q=1$ the commutation relations
and differentials coincide with (\ref{commrel}) and (\ref{4D+}).
The bicovariant calculus $\Gamma^+_P$
is only $2$-dimensional: one of the three basis vectors $[\omega^1],[\omega^3],
[\omega^4]$ of $\Gamma_{P_{\mathrm{GL}}}$
vanishes on $\Gamma^+_P$,
namely $[\omega^1]$. Indeed, denoting equivalence classes on $\Gamma^+_P$ under
$\mathrm{pr}^\Gamma_P\colon\Gamma_{P_\mathrm{GL}}\rightarrow\Gamma^+_P$
by $[\cdot]'$ and using \eqref{qslq2} and \eqref{4D+} we obtain
\begin{equation*}
    0=[\mathrm{d}(\alpha\delta)]'
    =[\mathrm{d}(\alpha)\delta]'+[\alpha\mathrm{d}\delta]'
    =\bigg[\bigg(\frac{q-1}{\lambda}+qQ\bigg)\omega^1\bigg]',
\end{equation*}
which implies $[\omega^1]'=0$. Then, the commutation relations (\ref{commrel2})
and the differentials (\ref{diff2}) project to
\begin{equation*}
\begin{split}
    [{\omega}^3]' t=t[{\omega}^3]'\,,\,\,  [{\omega}^3]' p=p[{\omega}^3]'\,,\,\,[{\omega}^4]' t=q^{-1}t[{\omega}^4]'\,,\,\,[{\omega}^4]'p=qp[{\omega}^4]'-\lambda t[{\omega}^3]'~,\\[.2em]
    \mathrm{d}^+_P t=\mbox{$\frac{q^{-1}-1}{\lambda}$}
    \,t\,[{\omega}^4]'\,,\,\,\,\, \mathrm{d}^+_P p=-t \,[{\omega}^3]'+\mbox{$\frac{q-1}{\lambda}$}\,p\,[{\omega}^4]'~~~~~~~~~~~~~~~~~~~~~~~
\end{split}
\end{equation*}
on $\mathcal{O}_q(P)$, where we identified $t,p$ as generators in
$\mathcal{O}_q(P)$.
\end{example}

\begin{example}[\textit{$3$-dim calculus on $\mathcal{O}_q(\mathrm{SL}_2(\mathbb{C}))$ and its parabolic quotient $\mathcal{O}_q(P)$}]\label{ex01}
Let $A:=\mathcal{O}_q(\mathrm{SL}_2(\mathbb{C}))$ and
$H:=\mathcal{O}_q(P)$ be the Hopf algebras from Example~\ref{sl2}.
There is a $3$-dimensional left covariant FODC 
$(\Gamma_{\mathrm{SL}},\mathrm{d}_{\mathrm{SL}})$
on $A$ described in \cite{KSbook}~Sect.~14.1.4.
$\Gamma_\mathrm{SL}$ is the free left $A$-module generated by the basis
$\{\omega^0,\omega^1,\omega^2\}$ of left coinvariant $1$-forms
with commutation relations
\begin{equation}\label{eq124}
\begin{split}
    &\omega^j\alpha=q^3\alpha\omega^j,~~~~~
    \omega^j\beta=q^{-3}\beta\omega^j,\\
    &\omega^j\gamma=q^3\gamma\omega^j,~~~~~
    \omega^j\delta=q^{-3}\delta\omega^j,
\end{split}
\end{equation}
for $j=0,2$ and
\begin{equation}\label{eq125'}
\begin{split}
    &\omega^1\alpha=q^2\alpha\omega^1+(q^2-1)\beta\omega^2,~~~~~
    \omega^1\beta=q^{-2}\beta\omega^1+(q^2-1)\alpha\omega^0,\\
    &\omega^1\gamma=q^2\gamma\omega^1+(q^2-1)\delta\omega^2,~~~~~~
    \omega^1\delta=q^{-2}\delta\omega^1+(q^2-1)\gamma\omega^0.
\end{split}
\end{equation}
The differential $\mathrm{d}_{\mathrm{SL}}\colon A
\rightarrow\Gamma_{\mathrm{SL}}$ is determined by
\begin{equation*}
\begin{split}
    &\mathrm{d}_{\mathrm{SL}}\alpha=\alpha\omega^1+\beta\omega^2,~~~~~
    \mathrm{d}_{\mathrm{SL}}\beta=\alpha\omega^0-q^{-2}\beta\omega^1,\\
    &\mathrm{d}_{\mathrm{SL}}\gamma=\gamma\omega^1+\delta\omega^2,~~~~~~
    \mathrm{d}_{\mathrm{SL}}\delta=\gamma\omega^0-q^{-2}\delta\omega^1
\end{split}
\end{equation*}
and thus
\begin{equation*}
\begin{split}
    \omega^0=\delta\mathrm{d}_{\mathrm{SL}}\beta
    -q\beta\mathrm{d}_{\mathrm{SL}}\delta\,,~~~~~
    \omega^1=\delta\mathrm{d}_{\mathrm{SL}}\alpha
    -q\beta\mathrm{d}_{\mathrm{SL}}\gamma\,,~~~~~
    \omega^2=-q^{-1}\gamma\mathrm{d}_{\mathrm{SL}}\alpha
    +\alpha\mathrm{d}_{\mathrm{SL}}\gamma\,.
\end{split}
\end{equation*}

As in Example~\ref{sl2}
elements in the equivalence class $H$ are denoted by $[\omega]$ with a
representative $\omega\in\Gamma_{\mathrm{SL}}$.
The induced quotient calculus $(\Gamma_P,\mathrm{d}_P)$ on $H$
is the $2$-dimensional
left covariant FODC with $\Gamma_P$ being the free left $H$-module generated by
the basis $\{[\omega^0],[\omega^1]\}$ of left coinvariant elements (notice that
$[\omega^2]=0$, while $[\omega^0],[\omega^1]$ are linearly independent). 
The resulting commutation relations are
\begin{equation*}
\begin{split}
    &[\omega^0]t=q^3t[\omega^0],~~~~~~~
    [\omega^0]p=q^{-3}p[\omega^0],\\
    &[\omega^1]t=q^2t[\omega^1],~~~~~~~
    [\omega^1]p=q^{-2}p[\omega^1]+(q^2-1)t[\omega^0]
\end{split}
\end{equation*}
and the differential reads
\begin{equation*}
    \mathrm{d}_Pt=t[\omega^1],~~~~~
    \mathrm{d}_Pp=t[\omega^0]-q^{-2}p[\omega^1],~~~~~
    \mathrm{d}_Pt^{-1}=-q^{-2}t^{-1}[\omega^1].
\end{equation*}
\end{example}

\subsection{The smash product calculus}\label{smash}
In this section we recall the construction of a covariant differential calculus
on the smash product algebra $B\#H$ from an $H$-module calculus on an
$H$-module algebra $B$ and a bicovariant
calculus on the Hopf algebra $H$, given in \cite{pflaum2}.
If the $H$-action on $B$ is trivial  we recover the tensor product differential
calculus on $B\otimes H$ described in Proposition \ref{prop01}.

\medskip
Let $B$ be a left $H$-module algebra, with action $\rhd: H\otimes B\to
B$. Let  $M$ be  a $B$-bimodule, with actions $\cdot\colon B\otimes M\to M$,
$\cdot\colon M\otimes B\to M$, and a left $H$-module with action that with
slight abuse we denote  $\rhd\colon H\otimes M\to
M$.  We say that $M$ is a {\it relative $H$-module
$B$-bimodule} if the $H$ and $B$ actions have the compatibility, for
all $h\in H, b,b'\in B$, $m\in M$, $$h\rhd (b \cdot m\cdot
b')=(h_1\rhd b)\cdot (h_2\rhd m)\cdot (h_3\rhd b')~.$$
Similar definitions are given if $M$ is just a left or a right $B$-module.

\begin{definition}
Let $B$ be a left $H$-module algebra with action $\rhd: H\otimes B\to B$. A FODC $(\Gamma_B,{\mathrm{d}_B})$ on
$B$ is said to be an {\it $H$-module FODC} if for any  $b^i,b_i\in B$,
$i=1,2,...n$, ($n\in \mathbb{N}$) and
$h\in H$ we have
\begin{equation}\label{defHFODC}
\sum_ib^i\mathrm{d}_Bb_i=0 \,\Rightarrow\,  \sum_i(h_1\rhd
b^i)\mathrm{d}_B(h_2\rhd b_i)=0 ~.
\end{equation}
\end{definition}
The rationale of this definition is in the following proposition.
\begin{proposition}
$(\Gamma_B,{\mathrm{d}_B})$ is an $H$-module FODC  if and only if
  $\Gamma_B$ is a relative $H$-module $B$-bimodule and $\mathrm{d}_B\colon B\to
  \Gamma_B$ is an $H$-module map: for all $h\in H$, $b\in B$, $h\rhd \mathrm{d}_Bb=\mathrm{d}_B(h\rhd b)$.
  \end{proposition}
  \begin{proof}
Since \eqref{defHFODC} holds we have a well-defined $H$-action given
 by,
 $$
H\otimes \Gamma_B\to \Gamma_B ~,~~ h\rhd\sum_i(b^i\mathrm{d}_Bb_i):=\sum_i(h_1\rhd
b^i)\mathrm{d}_B(h_2\rhd b_i)$$
where we used that $\Gamma_B=B\mathrm{d}_BB$, so that
$\sum_i b^i\mathrm{d}_Bb_i$ is a generic element of $\Gamma_B$. Trivially  $\mathrm{d}_B: B\to
  \Gamma_B$ is an $H$-module map. Compatibility of this action with
  the left $B$-action is immediate. Compatibility with the right $B$-action:
$h\rhd((\sum_i b^i\mathrm{d}_Bb_i) b)=(h_1\rhd\sum_i
b^i\mathrm{d}_Bb_i)(h_2\rhd b)$, follows writing $(\sum_i
b^i\mathrm{d}_Bb_i) b=\sum_i b^i\mathrm{d}_B(b_i b)-\sum_i
b^ib_i\mathrm{d}_Bb$.

Vice versa, if
  $\Gamma_B$ is a relative $H$-module $B$-bimodule and $\mathrm{d}_B: B\to
  \Gamma_B$ is an $H$-module map, the  implication \eqref{defHFODC} follows
  from its equivalence with $\sum_ib^i\mathrm{d}_Bb_i=0 \,\Rightarrow\,  h\rhd\sum_i(b^i\mathrm{d}_Bb_i)=0$.
\end{proof}

Given an $H$-module algebra $B$ and a FODC
$(\Gamma_{B},\mathrm{d}_B)$ and
a left covariant FODC
$(\Gamma_{H},\mathrm{d}_H)$,
we consider the $\Bbbk$-module
\begin{equation}\label{eq126}
    \Gamma_\#:=\Gamma_B\otimes H\,\oplus\, B\otimes\Gamma_H~
\end{equation}
and study when there is a FODC  $(\Gamma_\#,\mathrm{d}_\#)$ on
$B\#H$.
The $\Bbbk$-module $\Gamma_\#$  in \eqref{eq126}
is a direct sum of tensor products of left $H$-modules hence it carries a
left $H$-action canonically induced from the $H$-actions on the
$H$-modules $\Gamma_B$,  $H$,  $B$, $\Gamma_H$: 
for all $h\in H$, $\omega_B\otimes h'+b'\otimes \omega_H\in   \Gamma_\#$,
\begin{equation*}
h\cdot  (\omega_B\otimes h'+b'\otimes \omega_H)=h_1\rhd\omega_B\otimes h_2h'+h_1\rhd b'\otimes h_2\omega_H~,
\end{equation*}
extended linearly to all $\Gamma_\#$.
Defining the left $B$-action on $\Gamma_\#$ as the $B$-action on the first
factors in the tensor products $\Gamma_B\otimes H$ and
$B\otimes\Gamma_H$ we obtain the left $B\#H$-action on $\Gamma_\#$
using that $b\# h=(b\# 1_H)(1_B\# h)$:
\begin{equation}\label{leftBHaction}
  (b\# h)\cdot (\omega_B\otimes h'+b'\otimes
  \omega_H):=b(h_1\rhd\omega_B)\otimes h_2h'+b(h_1\rhd b)'\otimes
  h_2\omega_H~.
  \end{equation}
The proof that this indeed defines  an action of the algebra  $B\#H$
on $\Gamma_\#$, actually a  $B\#H$-action on the submodules
$\Gamma_B\otimes H$ and $B\otimes\Gamma_H$, uses the same steps of the proof of associativity of the
multiplication in $B\#H$.
We define 
\begin{equation*}
  (\omega_B\otimes h'+b'\otimes \omega_H)\cdot b:=
  \omega_B(h'_1\rhd b)\otimes h'_2+b'((\omega_H)^{}_{-1}\rhd b)\otimes (\omega_H)^{}_{0}
\end{equation*}
and extend it linearly to all $\Gamma_\#$. It is easy to prove that this is
a right $B$-action on  $\Gamma_\#$. Defining the right $H$-action on
$\Gamma_\#$ as the right $H$-action on the second
factors in the tensor products $\Gamma_B\otimes H$ and
$B\otimes\Gamma_H$ we obtain the right $B\#H$-action on $\Gamma_\#$
\begin{equation}\label{rightBHaction}
(\omega_B\otimes h'+b'\otimes \omega_H)\cdot (b\#h):=
  \omega_B(h'_1\rhd b)\otimes
  h'_2h + b'((\omega_H)^{}_{-1}\rhd b)\otimes (\omega_H)^{}_{0}h~.
\end{equation}
We prove commutativity of the left and right $B\#H$-actions on
$\Gamma_B\otimes H$:
\begin{equation*}
  \begin{split}
(b\# h)\cdot ((\omega_B\otimes h')\cdot (\tilde b\# \tilde h))&=
(b\# h)\cdot (\omega_B( h'_1\rhd\tilde b)\otimes h'_2\tilde h)\\
&=
b\cdot(h_1\rhd(\omega_B( h'_1\rhd\tilde b))\otimes h_2h'_2\tilde h)\\
&=b(h_1\rhd\omega_B) (h_2 h'_1\rhd\tilde b)\otimes h_3h'_2\tilde h\\
&=(b(h_1\rhd\omega_B) \otimes h_2 h')\cdot(\tilde b\# \tilde h)\\
&=((b\# h)\cdot (\omega_B\otimes h'))\cdot (\tilde b\# \tilde h)~.
\end{split}
\end{equation*}
Commutativity of the left and right $B\#H$-actions on
$B\otimes \Gamma_H$ is similarly proven using the left $H$-comodule
structure of $\Gamma_H$. This shows that $\Gamma_\#$ is a
$B\#H$-bimodule.
We sometimes write $\Gamma_B\#H$ and $B\#\Gamma_H$ in order to stress that
we consider them as bimodules over $B\#H$ instead of $B\otimes H$.
As in \cite{pflaum2}~Thm.~2.7 we have a
FODC on $B\#H$ with $B\#H$-bimodule $\Gamma_\#$.
\begin{theorem}\label{prop03}
Let $H$ be a Hopf algebra and $B$ a left $H$-module algebra.
  Given an $H$-module FODC $(\Gamma_B,\mathrm{d}_B)$ on $B$ and a
left covariant FODC $(\Gamma_H,\mathrm{d}_H)$ on $H$ there is a FODC $(\Gamma_\#,\mathrm{d}_\#)$ on $B\#H$,
where the $\Bbbk$-module 
$\Gamma_\#:=\Gamma_B\otimes H\oplus B\otimes\Gamma_H$ is endowed with the
$B\#H$-bimodule actions \eqref{leftBHaction}, \eqref{rightBHaction}
and the exterior derivative  $\mathrm{d}_\#\colon B\#H\rightarrow\Gamma_\#$ is defined by
\begin{equation*}
    \mathrm{d}_\#(b\#h)
    :=\mathrm{d}_Bb\otimes h+b\otimes\mathrm{d}_Hh
\end{equation*}
for all $b\in B$ and $h\in H$.
\begin{proof}
  We show that $\mathrm{d}_\#\colon B\#H\rightarrow\Gamma_\#$
  satisfies the Leibniz rule:
\begin{align*}
    \mathrm{d}_\#((b\#h)(b'\#h'))
    &=\mathrm{d}_\#(b(h_1\rhd b')\#h_2h')\\
    &=\mathrm{d}_B(b(h_1\rhd b'))\otimes h_2h'
    +b(h_1\rhd b')\otimes\mathrm{d}_H(h_2h')\\
    &=(\mathrm{d}_Bb)(h_1\rhd b')\otimes h_2h'
    +\,b\!\;\mathrm{d}_B(h_1\rhd b')\otimes h_2h'\\
    &~~~~+b(h_1\rhd b')\otimes(\mathrm{d}_Hh_2)h'
    \,+\, b(h_1\rhd b')\otimes h_2\!\:\mathrm{d}_Hh'\\
    &=\mathrm{d}_\#(b\#h)\cdot(b'\#h')
    +(b\#h)\cdot\mathrm{d}_\#(b'\#h')
\end{align*}
for all $b,b'\in B$ and $h,h'\in H$.
We are left to prove that 
$\Gamma_\#=(B\#H)\cdot\mathrm{d}_\#(B\#H)$. Let $b,b'\in B$ and $h,h'\in H$.
Then
\begin{align*}
    b\!\;\mathrm{d}_Bb'\otimes h
    =b\!\;\mathrm{d}_Bb'\otimes h
    +bb'\otimes\mathrm{d}_Hh
    -bb'\otimes\mathrm{d}_Hh
    =(b\#1)\cdot\mathrm{d}_\#(b'\#h)-(bb'\#1)\cdot\mathrm{d}_\#(1\#h)
\end{align*}
and $b\otimes h\!\:\mathrm{d}_Hh'=(b\#h)\cdot\mathrm{d}_\#(1\#h')$
establish the equality in question.
\end{proof}
\end{theorem}
The smash product construction  of differential calculi
is compatible with right $H$-coactions.
\begin{corollary}\label{CorSmash}
Let $H$ be a Hopf algebra, $B$ a left $H$-module algebra,
$(\Gamma_B,\mathrm{d}_B)$ an $H$-module FODC  on $B$ and 
$(\Gamma_H,\mathrm{d}_H)$ a bicovariant FODC on $H$. The FODC
$(\Gamma_\#,\mathrm{d}_\#)$ of Theorem \ref{prop03} is then right $H$-covariant.
\end{corollary}
\begin{proof}
 Define a right $H$-coaction on $\Gamma_\#$ via
\begin{equation*}
\begin{split}
    \Gamma_\#&\xrightarrow{\Delta_\#}\Gamma_\#\otimes H\\
    \omega_B\otimes h+b\otimes\omega_H
    &\:\!\longmapsto\,\omega_B\otimes h_1\otimes h_2
    +b\otimes(\omega_H)^{}_0\otimes(\omega_H)^{}_1.
\end{split}
\end{equation*}
We prove that the calculus is right $H$-covariant by showing right
$H$-colinearity of the differential, cf. Proposition~\ref{lemma03}~i.).
For all $ b\#h\in B\# H$,
\begin{align*}
    \Delta_\#(\mathrm{d}_\#(b\#h))
    &=\Delta_\#(\mathrm{d}_Bb\otimes h
    +b\otimes \mathrm{d}_Hh)\\
    &=\mathrm{d}_Bb\otimes h_1\otimes h_2
      +b\otimes \mathrm{d}_Hh_1\otimes h_2\\
  &=\mathrm{d}_\#(b\#h_1)\otimes
      h_2\\
      &=(\mathrm{d}_\#\otimes \mathrm{id}_A)\Delta_\#(b\#h)~.
\end{align*}
\\[-3.4em]\end{proof}

\subsection{Base forms, horizontal forms and principal covariant calculi}\label{pc-sec}
In this section we study differential calculi on noncommutative
principal bundles over affine bases. We assume the ground ring $\Bbbk$ to be a
field and recall that the Hopf algebra (quantum structure group) $H$
is assumed to have invertible antipode. We have seen that in this setting a principal comodule algebra
$B=A^{\coi H}\subseteq A$ is equivalently a faithfully flat Hopf Galois extension $B=A^{\coi H}\subseteq A$.

\begin{definition}\label{horforms-def}
  Let $(\Gamma_A, \rm d_A)$ be a FODC on a right $H$-comodule algebra
$A$. We call 
the pullback calculus $(\Gamma_B,\mathrm{d}_B):=(B\rd_A|^{\phantom{jj}}_BB,\rd_A|^{\phantom{jj}}_B)$ on
$B:=A^{\mathrm{co}H}\subseteq A$ the
FODC of {\textit{base forms}}. We further call
$\Gamma^\mathrm{hor}:=A\Gamma_B$, the $(A,B)$-bimodule of \textit{horizontal forms}.
\end{definition}

If $(\Gamma_A,\mathrm{d}_A)$ is a right $H$-covariant FODC, base and
horizontal forms can be further characterized.

\begin{theorem}\label{prop-ff}
Let $A$ be a principal comodule algebra, $B:=A^{\coi H}$
   and $(\Gamma_A,\mathrm{d}_A)$ a right
   $H$-covariant FODC on $A$.
The natural map
  $
  A \otimes_B \Gamma_B \to\Gamma_A,~ a \otimes \omega \mapsto a\omega
  $
  is injective and  gives the left $A$-module isomorphism
  $$
  A \otimes_B\Gamma_B\cong A\Gamma_B ~.
  $$ 
The $B$-bimodule of base forms is the intersection of those of horizontal and coinvariant forms
\begin{equation*}
    \Gamma_B=\Gamma_A^\mathrm{hor}\cap\Gamma_A^{\mathrm{co}H}.
\end{equation*}
\end{theorem}

\begin{proof}
  The inclusion $A\Gamma_B \subseteq \Gamma_A$
  implies $(A\Gamma_B)^{\coi H} \subseteq (\Gamma_A)^{\coi H}$ and, since $A$
  is a flat $B$-module, we have the inclusion
  \begin{equation}\label{ff-1}
  A \otimes_B (A\Gamma_B)^{\coi H} \lra A \otimes_B (\Gamma_A)^{\coi H}
  \cong \Gamma_A
  \end{equation}
  where for the isomorphism $A \otimes_B (\Gamma_A)^{\coi H}\cong \Gamma_A$
  we used Theorem \ref{ta-schn}.
  The inclusion $\Gamma_B \subseteq   (A\Gamma_B)^{\coi H}\,$ and flatness of $A$
  over $B$ imply the inclusion
  \begin{equation}\label{ff-2}
  A \otimes_B \Gamma_B \lra A \otimes_B (A\Gamma_B)^{\coi H}
  \end{equation}
  that composed with (\ref{ff-1}) gives injectivity of the natural map
  $
  A \otimes_B \Gamma_B \to\Gamma_A$
  and hence the isomorphism $A \otimes_B \Gamma_B\cong A\Gamma_B$.

  From  Theorem \ref{ta-schn} we then have $
    \Gamma_B\cong (A \otimes_B \Gamma_B)^{\coi H}\cong
    (A\Gamma_B)^{\coi H}$, that is,
    $\Gamma_B=\Gamma_A^\mathrm{hor}\cap\Gamma_A^{\mathrm{co}H}$.
 \end{proof}

We now consider a FODC $(\Gamma_H,\mathrm{d}_H)$ on $H$, with
$(\Gamma_A,\mathrm{d}_A)$ that is not necessarily right $H$-covariant.

\begin{definition}\label{PrincipalDC}
A FODC $(\Gamma_A,\mathrm{d}_A)$
on a principal comodule algebra $A$ together with a left covariant FODC $(\Gamma_H,\mathrm{d}_H)$ on $H$,
is called a \textit{principal calculus}
on $A$
if we have the short exact sequence
\begin{equation}\label{eq134}
    0\rightarrow A\otimes_B\Gamma_B\rightarrow\Gamma_A\xrightarrow{\mathrm{ver}}
    A\square_H\Gamma_H\rightarrow 0
\end{equation}
where
 $
    A\square_H\Gamma_H:=\mathrm{span}_\Bbbk\{
    a\otimes\omega^H\in A\otimes\Gamma_H~|~\delta_A(a)\otimes\omega^H
    =a\otimes\Delta_L^{\Gamma_H}(\omega^H)\}
$
    is the \textit{cotensor product} of $A$ and $\Gamma_H$, and the
    \textit{vertical map}
     $\mathrm{ver}:\Gamma_A\to A\square_H\Gamma_H$ is well-defined as a $\Bbbk$-linear map by
\begin{equation}\label{ver}
    \mathrm{ver}(a\mathrm{d}_Aa'):=a_0a'_0\otimes a_1\mathrm{d}_Ha'_1~.
\end{equation}
If, in addition, the FODC $(\Gamma_A,\mathrm{d}_A)$ is right $H$-covariant
and the FODC $(\Gamma_H,\mathrm{d}_H)$ is bicovariant, we say we have
a \textit{principal covariant calculus}
on the principal comodule algebra $A$.
\end{definition}

We can easily check that $\mathrm{ver}$, if it is well-defined, it satisfies
\begin{equation}\label{eq133}
    \mathrm{ver}(a\cdot\omega\cdot a')
    =\delta_A(a)\mathrm{ver}(\omega)\delta_A(a')
\end{equation}
for all $a,a'\in A$ and $\omega\in\Gamma_A$.
The sequence is therefore a sequence of left $A$-modules right $B$-modules.

\medskip
The following lemma provides a sufficient condition for the existence
of the vertical map.

\begin{lemma}\label{lemma-ver}
Let $\pi\colon A\rightarrow H$ be a Hopf algebra quotient
and $A$ be a principal comodule algebra. For any left covariant
calculus $(\Gamma_A,\mathrm{d}_A)$ on $A$ and induced left covariant
quotient calculus $(\Gamma_H,\mathrm{d}_H)$ on $H$ the vertical map
is well-defined.
\end{lemma}
\begin{proof}
Let us denote the left $A$-coaction on $\Gamma_A$ by
$\Delta_L^{\Gamma_A}\colon\Gamma_A\rightarrow A\otimes\Gamma_A$,
$\omega^A\mapsto\omega^A_{-1}\otimes\omega^A_0$,
and the quotient map of forms by
$\pi_\Gamma\colon\Gamma_A\rightarrow\Gamma_H$.
Then
\begin{equation*}
\begin{split}
    (\mathrm{id}_A\otimes\pi_\Gamma)\circ\Delta_L^{\Gamma_A}\colon\Gamma_A
    &\rightarrow A\otimes\Gamma_H\\
    \omega^A&\mapsto\omega^A_{-1}\otimes[\omega^A_0]
\end{split}
\end{equation*}
is a well-defined map. We prove that this map
coincides with $\mathrm{ver}$. Recall that the induced right $H$-coaction
on $A$ is given by 
$\delta_A:=(\mathrm{id}_A\otimes\pi)\circ\Delta
\colon A\rightarrow A\otimes H$, $a\mapsto a_0\otimes a_1:=a_1\otimes[a_2]$.
For $a,a'\in A$ we obtain by the
left covariance of $(\Gamma_A,\mathrm{d}_A)$
\begin{align*}
    (\mathrm{id}_A\otimes\pi_\Gamma)(\Delta_L^{\Gamma_A}(a\mathrm{d}_Aa'))
    =(a\mathrm{d}_Aa')_{-1}\otimes[(a\mathrm{d}_Aa')_0]
    =a_1a'_1\otimes[a_2\mathrm{d}_Aa'_2]
    =a_1a'_1\otimes[a_2]\mathrm{d}_H[a'_2]
    =a_0a'_0\otimes a_1\mathrm{d}_Ha'_1,
\end{align*}
i.e. $(\mathrm{id}_A\otimes\pi_\Gamma)\circ\Delta_L^{\Gamma_A}=\mathrm{ver}$,
as claimed.
\end{proof}

For a right $H$-covariant calculus on a principal comodule algebra the second arrow
in \eqref{eq134} is injective by Theorem \ref{prop-ff}.
In this context of principal comodule algebras the notion of principal
covariant calculus of Definition \ref{PrincipalDC} is then equivalent to 
that of strong quantum principal bundle in \cite{brz1}~Def.~4.9 and
\cite{bm}~Sect.~5.4. This follows from the canonical identifications
$A \otimes_B\Gamma_B\cong A\Gamma_B$, cf. Theorem \ref{prop-ff}, and
$A\square_H\Gamma_H\cong A\otimes {}^{\coi H}\Gamma_H$,
where  ${}^{\coi H}\Gamma_H$ denotes the module of left coinvariant one forms characterizing the covariant FODC
$(\Gamma_H,\mathrm{d}_H)$ (see \cite{wor}). 
In this light Theorem \ref{prop-ff}, which does not assume an exact
sequence, can be read as generalizing the conditions of
\cite{bm}~Cor.~5.53 for the equality
$\Gamma_B=\Gamma_A^\mathrm{hor}\cap\Gamma_A^{\mathrm{co}H}$ to hold true.

\medskip

We present two examples of principal covariant calculus, further
examples including a principal calculus are in Section \ref{sectionQDCProjSpace}.

\begin{example}[\textit{Smash product calculus}]\label{ex02}
Let $(\Gamma_H,\mathrm{d}_H)$ be a bicovariant FODC on a Hopf algebra $H$
and $(\Gamma_B,\mathrm{d}_B)$ an $H$-module FODC on a left $H$-module algebra
$B$. Then, the smash product calculus $(\Gamma_\#,\mathrm{d}_\#)$ of
Section \ref{smash} is a principal covariant calculus on
$B\#H$. 
\begin{proof}
Recall from Section \ref{sec:sp} that $B=(B\#H)^{\mathrm{co}H}\subseteq B\#H$ is a trivial
Hopf--Galois extension.  We show that the sequence  in \eqref{eq134} is equivalent to
the exact sequence
$$
0\rightarrow\Gamma_B\#H\rightarrow\Gamma_B\#H\oplus B\#\Gamma_H\xrightarrow{\mathrm{pr}}B\#\Gamma_H\rightarrow 0~.
$$
From\cite{pflaum2}~Thm.~4.1 and Lem.~4.2
we have the isomorphisms of right $H$-covariant $B\#H$-bimodules
\begin{equation}\label{eq125''}
\begin{split}
    (B\#H)\otimes_B\Gamma_B\cong
    \Gamma_B\#H,~~~~~~~~&
    (b\#h)\otimes_B\omega^B\mapsto b(h_1\rhd\omega^B)\#h_2,\\
    (B\#H)\square_H\Gamma_H\cong
    B\#\Gamma_H,~~~~~~~~&
    (b\#h)\otimes\omega^H\mapsto b\#\epsilon(h)\omega^H~.
\end{split}
\end{equation}
Their inverses are given by $\omega^B\#h\mapsto(1\#h_2)\otimes_B(
\overline{S}(h_1)\rhd\omega^B)$ and
$b\#\omega^H\mapsto(b\#\omega^H_{-1})\otimes\omega^H_0$,
respectively.
Under this first isomorphism the second arrow in \eqref{eq134}
becomes the inclusion $\Gamma_B\#H\to \Gamma_B\#H\oplus B\#\Gamma_H$.
Finally, the vertical map equals the projection
$\mathrm{pr}\colon\Gamma_B\#H\oplus B\#\Gamma_H\rightarrow
B\#\Gamma_H$ up to the second isomorphism in \eqref{eq125''}
(cf. e.g. \cite{pflaum2}~Thm.~2.9).
\end{proof}
\end{example}

\begin{example}[\textit{3D calculus on the q-monopole fibration}]
Consider $A=\mathcal{O}_q(\mathrm{SL}_2(\mathbb{C}))$ ($q$ real,
$q\not=\pm 1$) with the
Hopf algebra quotient
\begin{equation*}
    \pi\colon A\rightarrow H 
,~~~~~~
    \begin{pmatrix}
    \alpha & \beta\\
    \gamma & \delta
    \end{pmatrix}\mapsto\begin{pmatrix}
    \alpha & 0\\
    0 & \delta
    \end{pmatrix}
    =\begin{pmatrix}
    \phi & 0\\
    0 & \phi^{-1}
    \end{pmatrix}.
\end{equation*}
The projection $\pi$ induces a right $H$-coaction:
\begin{equation*}
    \delta_A:=(\mathrm{id}\otimes\pi)\circ\Delta\colon A\rightarrow A\otimes H,~~~~~~~
    \begin{pmatrix}
    \alpha & \beta\\
    \gamma & \delta
    \end{pmatrix}
    \mapsto\begin{pmatrix}
    \alpha & \beta\\
    \gamma & \delta
    \end{pmatrix}\dot\otimes
    \begin{pmatrix}
    \phi & 0\\
    0 & \phi^{-1}
    \end{pmatrix}.
\end{equation*}
The subalgebra $B:=A^{\mathrm{co}H}$ of coinvariants is given by the (complex) Podle\'s
sphere $\mathcal{O}_q(\mathbb{S}^2)$ with generators
$B_-=\alpha\beta$, $B_+=\gamma\delta$, $B_0=\beta\gamma$ and
commutation relations
\begin{equation*}
    B_\pm B_0=q^{\pm 2}B_0B_\pm,~~~~~~
    B_-B_+=q^{-2}B_+B_-+q^{-2}(q^{-1}-q)B_0^2.
\end{equation*}
(see \cite{podles} and refs. therein).
It is known that $B\subseteq A$ is a faithfully flat Hopf--Galois extension,
i.e. that $A$ is a principal comodule algebra \cite{MSchneider},
(cf. also \cite{BRZ}~Ex.~6.26). The $3$-dimensional left covariant
FODC on $A$ introduced in \cite{wor3} is $H$-bicovariant. It
induces a $2$-dimensional FODC
on $B$ via the algebra embedding $B\rightarrow A$ and a $1$-dimensional
bicovariant FODC on $H$ via the quotient map
$\pi$. This data is a principal covariant calculus on $A$, see \cite{bm}~Ex.~5.51.                 
\end{example}
\medskip

We next show that for a principal covariant calculus 
on a principal comodule algebra $A$
the Hopf--Galois extension $B=A^{\mathrm{co}H}\subseteq A$ lifts to a
Hopf--Galois extension of graded algebras. Recall that for a bicovariant
$H$-bimodule $\Gamma_H$ we obtain a graded Hopf algebra
$\Omega^{\leqslant 1}_H:=H\oplus\Gamma_H$. The multiplication is given by
that in $H$ and by the $H$-module structure, while the product of degree one
elements is trivial. The
comultiplication $\Delta_{\Omega^{\leqslant 1}_H}\colon\Omega^{\leqslant 1}_H\rightarrow
\Omega^{\leqslant 1}_H\otimes\Omega^{\leqslant 1}_H$ has components on degree zero and
one $\Delta^0:=\Delta\colon H\rightarrow H\otimes H$ and
\begin{equation*}
    \Delta^1:=\Delta_R^{\Gamma_H}+\Delta_L^{\Gamma_H}
    \colon\Gamma_H\rightarrow\Gamma_H\otimes H\oplus H\otimes\Gamma_H~.
\end{equation*} The antipode
$S_{\Omega^{\leqslant 1}_H}\colon\Omega^{\leqslant 1}_H\rightarrow\Omega^{\leqslant 1}_H$
has components $S^0=S\colon H\rightarrow H$ and
\begin{equation*}
\begin{split}
    S^1\colon\Gamma_H&\rightarrow\Gamma_H\\
    \omega&\mapsto-S(\omega_{-1})\omega_0S(\omega_1).
\end{split}
\end{equation*}
\begin{lemma}\label{lemma02}
Consider a principal comodule algebra $A$ with a principal
covariant calculus $(\Gamma_A,\mathrm{d}_A)$. The graded algebra
$\Omega^{\leqslant 1}_A:=A\oplus\Gamma_A$ is a graded right $\Omega^{\leqslant 1}_H$-comodule
algebra with right coaction $\Delta_{\Omega^{\leqslant 1}_A}\colon
\Omega^{\leqslant 1}_A\rightarrow\Omega^{\leqslant 1}_A\otimes\Omega^{\leqslant 1}_H$
given by its
 components $\Delta_{\Omega^{\leqslant 1}_A}^0:=\delta_A\colon A\rightarrow A\otimes H$ and
\begin{equation}\label{deltaom}
    \Delta_{\Omega^{\leqslant 1}_A}^1:=\Delta_R^{\Gamma_A}+\mathrm{ver}
    \colon\Gamma_A\rightarrow\Gamma_A\otimes H\oplus A\otimes\Gamma_H.
\end{equation}
Defining $\Omega^{\leqslant 1}_B:=B\oplus\Gamma_B\subseteq \Omega^{\leqslant
  1}_A$, we further have that
$\Omega^{\leqslant 1}_B=(\Omega^{\leqslant 1}_A)^{\mathrm{co}\Omega^{\leqslant 1}_H}
$ is the subalgebra of $\Omega^{\leqslant 1}_H$-coinvariants.
\end{lemma}
\begin{proof}
By assumption $\Delta_{\Omega^{\leqslant 1}_A}^0=\delta_A\colon A\rightarrow A\otimes H
\subseteq A\otimes\Omega^{\leqslant 1}_H$ is a right  $\Omega^{\leqslant
  1}_H$-coaction and an algebra map. Similarly also 
$\Delta_{\Omega^{\leqslant 1}_A}$ is an $\Omega^{\leqslant
  1}_H$-coaction and an algebra map. For the coaction property we just have to prove 
$(\Delta_{\Omega^{\leqslant 1}_A}\otimes\mathrm{id})\circ 
\Delta_{\Omega^{\leqslant 1}_A}=( \mathrm{id}\otimes\Delta_{\Omega^{\leqslant 1}_H}) \circ \Delta_{\Omega^{\leqslant 1}_A}$
 on
$1$-forms. This is a straightforward computation using the right
$H$-covariance of $(\Gamma_A,\mathrm{d}_A)$, the definition of the
vertical map and the bicovariance of 
$(\Gamma_H,\mathrm{d}_H)$.
 The coaction $\Delta_{\Omega^{\leqslant 1}_A}$ is an algebra map because for all $a,a'\in A$ and $\omega\in\Gamma_A$ we have
\begin{equation*}
    \Delta_{\Omega^{\leqslant 1}_A}(a\cdot\omega\cdot a')
    =\delta_A(a)\Delta_{\Omega^{\leqslant 1}_A}(\omega)\delta_A(a')
    =\Delta_{\Omega^{\leqslant 1}_A}(a)\Delta_{\Omega^{\leqslant 1}_A}(\omega)
    \Delta_{\Omega^{\leqslant 1}_A}(a')~,
\end{equation*}
where we used that $\Gamma_A$ is a right $H$-covariant $A$-bimodule and that the
vertical map satisfies the compatibility condition (\ref{eq133}).

The algebra inclusion $B\subseteq A$ implies the graded algebra one
$\Omega^{\leqslant 1}_B:=B\oplus\Gamma_B\subseteq \Omega^{\leqslant 1}_A$.
The equality $\Omega^{\leqslant 1}_B=(\Omega^{\leqslant 1}_A)^{\mathrm{co}\Omega^{\leqslant 1}_H}
$  is obvious in degree zero. In degree one,
recalling  definition \eqref{deltaom} and that the codomain of $\Delta_{\Omega^{\leqslant 1}_A}$
is a direct sum, the coinvariance condition
$
\Delta_{\Omega^{\leqslant 1}_A}(\omega)=\omega\otimes 1$ implies 
$\Delta_R^{\Gamma_A}(\omega)=\omega\otimes 1$ and
$\mathrm{ver}(\omega)=0$. These equalities respectively imply
$\omega\in \Gamma_A^{\mathrm{co}H}$ and  $\omega\in
\Gamma_A^\mathrm{hor}$, this latter condition following from exactness of the
sequence \eqref{eq134}. 
Then, from Theorem~\ref{prop-ff} we obtain
$\Gamma_B=\Gamma^{\mathrm{hor}}\cap\Gamma^{\mathrm{co}H}$, which
implies
$(\Gamma_A)^{\mathrm{co}\Omega^{\leqslant 1}_H}
=\Gamma_B$. In conclusion,
$(\Omega^{\leqslant 1}_A)^{\mathrm{co}\Omega^{\leqslant 1}_H}
=\Omega^{\leqslant 1}_B$.
\end{proof}

In the following we prove that
principal covariant calculi  are equivalent to graded Hopf--Galois
extensions with compatible differentials.
\begin{theorem}\label{thm03}
Let $A$ be a principal comodule algebra and $(\Gamma_A,\mathrm{d}_A)$ a
principal covariant calculus on $A$ with corresponding bicovariant
FODC $(\Gamma_H,\mathrm{d}_H)$ on $H$.
Then
\begin{equation}\label{GradedHG}
    \Omega^{\leqslant 1}_B=(\Omega^{\leqslant 1}_A)^{\mathrm{co} \Omega^{\leqslant 1}_H}
    \subseteq\Omega^{\leqslant 1}_A
\end{equation}
is a faithfully flat Hopf--Galois extension of graded algebras.
Moreover we have the following commutative diagram
\begin{equation}\label{GradedCoactDiag}
\begin{tikzcd}
\Gamma_A
\arrow{rr}{\Delta^1_{\Omega^{\leqslant 1}_A}}
& & \Gamma_{A\otimes H}\\
A \arrow{u}{\mathrm{d}_A}
\arrow{rr}{\delta_A}
& & A\otimes H
\arrow{u}[swap]{\mathrm{d}_{A\otimes H}}
\end{tikzcd}
\end{equation}
where we endow $A\otimes H$ with the tensor product calculus
$(\Gamma_{A\otimes H},\mathrm{d}_{A\otimes H})$.

Conversely, if $(\Gamma_A,\mathrm{d}_A)$ is a FODC on a right $H$-comodule 
algebra $A$ and 
$(\Gamma_H,\mathrm{d}_H)$ a bicovariant FODC on $H$ such that (\ref{GradedHG}) is
a faithfully flat Hopf--Galois extension of graded algebras and
the diagram in (\ref{GradedCoactDiag}) commutes, then $A$ is a principal comodule algebra
and $(\Gamma_A,\mathrm{d}_A)$ a principal covariant calculus on $A$.
\end{theorem}
\begin{proof}
If $(\Gamma_A,\mathrm{d}_A)$ is a principal covariant calculus on
$A$
$B=A^{\mathrm{co}H}\subseteq A$ is a  faithfully
flat Hopf--Galois extension, the sequence \eqref{eq134} is exact and 
by Lemma~\ref{lemma02} $\Omega^{\leqslant 1}_A$ is a graded
right $\Omega^{\leqslant 1}_H$-comodule algebra.  Then from \cite{SchHGFF}~Cor.~5.9,
$\Omega^{\leqslant 1}_B=(\Omega^{\leqslant 1}_A)^{\mathrm{co}\Omega^{\leqslant 1}_H}
\subseteq\Omega^{\leqslant 1}_A$ is a Hopf--Galois extension which is
faithfully flat as a left $\Omega^{\leqslant 1}_B$-module.
Moreover, the diagram (\ref{GradedCoactDiag}) commutes
by the right $H$-covariance of $(\Gamma_A,\mathrm{d}_A)$ and the
definition (\ref{ver}) of the vertical map.

Conversely, if $A$ is a right $H$-comodule algebra with a FODC
$(\Gamma_A,\mathrm{d}_A)$ and a  bicovariant FODC $(\Gamma_H,\mathrm{d}_H)$ on $H$
such that (\ref{GradedHG}) is a faithfully flat Hopf--Galois extension (so
that in particular  $B\subseteq A$ is a faithfully
flat Hopf--Galois extension) then 
from \cite{SchHGFF}~Cor.~5.9 the sequence
\begin{equation}\label{ProofSeq}
    0\rightarrow A\otimes_B\Gamma_B\rightarrow\Gamma_A\xrightarrow{\Theta}
    A\square_H\Gamma_H\rightarrow 0
\end{equation}
is exact, where $\Theta$ is the projection of
$\Delta^1_{\Omega^{\leqslant 1}_A}\colon \Gamma_A\rightarrow
\Omega^{\leqslant 1}_A\otimes\Omega^{\leqslant 1}_H$ to 
$A\otimes\Gamma_H$. If we assume
commutativity of the diagram in (\ref{GradedCoactDiag}) it follows that
\begin{equation*}
    \Delta_R^{\Gamma_A}(\mathrm{d}_Aa)
    +\Theta(\mathrm{d}_Aa)
    =\mathrm{d}_A(a_0)\otimes a_1
    +a_0\otimes\mathrm{d}_Ha_1
\end{equation*}
for all $a\in A$, which implies
$\Delta_R^{\Gamma_A}(\mathrm{d}_Aa)=\mathrm{d}_A(a_0)\otimes a_1$
and
$\Theta(\mathrm{d}_Aa)=a_0\otimes\mathrm{d}_Ha_1$
for all $a\in A$. Since $\Theta$ and $\Delta_R^{\Gamma_A}$ are $A$-bilinear
(where $A\otimes H$ is an $A$-bimodule via $\delta_A$) the first equality
implies that $(\Gamma_A,\mathrm{d}_A)$ is a right
$H$-covariant calculus, the second one that $\Theta$ is the vertical map.
Thus (\ref{ProofSeq}) is the exact sequence showing the principality
of  the differential calculus  $(\Gamma_A,\mathrm{d}_A)$.
\end{proof}

\section{A sheaf-theoretic approach to noncommutative calculi}
\label{sheaf-sec}

In this section we introduce differential calculi on quantum principal bundles  in a two-step process. First we define a covariant calculus on a sheaf of
comodule algebras $\mathcal{F}$ as a sheaf $\Upsilon$ of
$\mathcal{F}$-bimodules together
with a morphism $\mathrm{d}\colon\mathcal{F}\rightarrow\Upsilon$ of sheaves,
requiring locally Leibniz rule and a surjectivity condition.
In case $\mathcal{F}$ is a quantum principal bundle we demand an additional
local compatibility of the calculi on the total sheaf, the base sheaf and the
structure Hopf algebra.

\subsection{Quantum principal bundles as sheaves}\label{qb-sh}

We start by recalling the main notions and results of \cite{AFL}.
Those concern the (function) algebra level. In the remaining sections
we generalize the definitions and findings to FODCi.
\medskip

A \textit{quantum ringed space} $(M , \mathcal{O}_M)$ is
a pair consisting of a classical topological space $M$ and 
a sheaf $\mathcal{O}_M$ on $M$ of noncommutative algebras.
We call a sheaf of $H$-comodule algebras $\cF$
an $H$-\textit{principal bundle} 
or \textit{quantum principal bundle} 
(QPB) over $(M,\cO_M)$
if there exists an open covering 
$\{U_i\}$ of $M$ such that:
\begin{enumerate}
\item[i.)] 
  $\cF(U_i)^{\mathrm{co}H}=\cO_M(U_i)$,
\item[ii.)] $\cF$ is \textit{locally principal}, that is, 
$\mathcal{O}_M(U_i)\subseteq\cF (U_i )$ is a principal comodule
  algebra. 
\end{enumerate}
If the base ring is a field local principality is equivalent to
faithfully flatness of the local Hopf--Galois extensions.
If these Hopf--Galois extensions on the open cover $\{U_i\}$ are cleft or trivial
we say that the QPB is \textit{locally cleft} or \textit{locally trivial},
respectively.

 \begin{remark}[Principal comodule algebras restrict locally] \label{CiPa}
In \cite{cipa}~Lem.~4.1 it is shown that the pushforward of a
strong connection $\ell: H\to A\otimes A$ by a comodule algebra map
$\phi: A\to A'$ is a strong connection $(\phi\otimes \phi)\circ \ell$
on $A'$. Recalling that if the base ring $\Bbbk$ is a field there is a
bijective correspondence of strong connections on comodule algebras
with principal comodule algebras (cf. the paragraph after Definition
\ref{PrinComAl}) it follows that the existence of a comodule algebra map
$\phi: A\to A'$ implies that if $A$ is a principal comodule algebra, so is $A'$. 

Since $\cF$ is a sheaf of $H$-comodule algebras, for any inclusion
$U\subset U_i$ of an open $U$ of the topology of $M$ in an open 
$U_i$ of the covering
$\{U_i\}$, the restriction $r^{}_{U U_i}:\mathcal{F}(U_i)\to
\mathcal{F}(U)$, which by definition is a comodule algebra map, then
implies that $\mathcal{F}(U)$ is a principal comodule algebra.

Let now $\mathcal{F}_p$ be the stalk of $\mathcal{F}$ at $p\in M$.
Choosing $U_i$ such that $p\in U_i$ and considering the canonical map
$\mathcal{F}(U_i)\to\mathcal{F}_p$ of right $H$-comodule algebras,
defined by $s\mapsto [(U_i,s)]_p$, where the equivalence class $[(U_i,s)]_p\in \mathcal{F}_p$ is the
germ at $p$ of the section $s\in\mathcal{F}(U_i)$, we obtain
that the stalk $\mathcal{F}_p$ is a principal comodule algebra.
\end{remark}
Our main examples of QPBs have a quantum group as total space.
Let $G$ be a complex semisimple algebraic group and
$P$ a parabolic subgroup.
The quotient $G/P$ is a projective variety and the projection
$G \lra G/P$ is a principal bundle. 
Let $\cO_q(G)$ and $\cO_q(P)=\cO_q(G)/I_P$ be Hopf algebras
over $\C_q:=\C[q,q^{-1}]$,
quantizations of $\cO(G)$ and $\cO(P)$, the coordinate
algebras of $G$ and $P$, respectively (see \cite{AFL} Sect.~3,
\cite{cfg} Sect.~3). We shall later specialize $q$ to a complex
  number in order to consider Hopf--Galois extensions with  $\C$ the base ring.
Once we fix a projective embedding for $G/P$, obtained
through the global sections of a very ample line bundle $\cL$, we have
the graded algebra $\cO(G/P)$.
We say that $s\in \cO_q(G)$ is a {\it quantum section} if
$$
(\mathrm{id} \otimes \pi) \Delta(s) = s \otimes \pi(s), \qquad
s \equiv t \quad \hbox{mod} \, (q-1),
$$
where $\pi: \cO_q(G) \lra \cO_q(P)$,  
$\Delta$ is the coproduct of $\cO_q(G)$ and $t$ is a {\sl classical
section}, that is, it lifts to $\cO(G)$ the character of $P$ defining $\cL$
(see \cite{AFL} Def.~3.6).

Denote by $\{s_i\}_{i\in \cI}$ a choice of
linearly independent elements in $\Delta(s)=\sum_{i\in
\cI} s^i \otimes s_i$, that we assume to generate the homogeneous
coordinate ring $\cO_q(G/P)$ quantization of the commutative homogeneous
coordinate ring $\cO(G/P)$ ($\cL$ is very ample).
Assume furthermore that
$S_i=\{s_i^r, r \in \Z_{\geq 0}\}$ is
Ore, denote by $\cO_q(G)S_i^{-1}$ the Ore localization, and 
assume that subsequent Ore localizations do not depend on the order.
In the specialization for $q=1$, the $s_i$'s define an open cover
of $G$, whose projection on $M=G/P$ gives an open cover $\{U_i\}$ of
$M$ (see \cite{AFL}, Sect.~4 for more details). 
Let us define the open sets
$$
U_I:= U_{i_1} \cap \dots \cap U_{i_r}, \qquad I=(i_1, \dots , i_r)
$$
with $r=0,1,2,\ldots n$ (here $n$ is the  cardinality of $\mathcal{I}$
and the empty set $\varnothing$ corresponds to
$r=0$). We consider on $M$ the topology generated by the open sets $U_I$,
which hence form a basis ${\mathcal{B}}=\{U_I\}_{I\in\mathcal{II}}$, where $\mathcal{II}$ is
the set of ordered multi-indices
$I=(i_1,\dots, i_r)$, $1\leq i_1 < \dots < i_r \leq n$ with
  $r=1,2,\ldots n$ (and we also consider the case $r=0$ corresponding to the empty set).

\medskip
The main result in \cite{AFL} is Theorem 4.8 that states the following.

\begin{theorem}\label{main-afl}
Let the notation be as above.
Then
\begin{enumerate}
\item[i.)] the assignment
\begin{equation}\label{eq116}
    U_I \mapsto \cO_M(U_I):=
   {\C}_q[s_{k_1}s_{i_1}^{-1}, \dots, s_{k_r}s_{i_r}^{-1}; 1\leq k_1\leq
    n,\ldots 1\leq k_r\leq n] 
    \subset  \cO_q(G)S_{i_1}^{-1}\dots S_{i_r}^{-1}
\end{equation}
defines a sheaf $\cO_M$ on $M=G/P$,

\item[ii.)]
the assignment
\begin{equation}\label{eq115}
    U_I \mapsto \cF_G(U_I):=\cO_q(G)S_{i_1}^{-1} \dots S_{i_r}^{-1},\
    \qquad I=(i_1, \dots, i_r)
\end{equation}
defines a sheaf $\cF_G$ of right $\cO_q(P)$-comodule algebras 
on the topological space $M$,

\item[iii.)]  
$\cF_G^{\coi \cO_q(P)}=\cO_M$, i.e.,  the subsheaf $\cF_G^{\coi
    \cO_q(P)}: U\to \cF_G(U)^{\coi
    \cO_q(P)}\subset \cF_G(U)$ is canonically 
isomorphic to the sheaf $\cO_M$.
\end{enumerate}
\end{theorem}
\begin{remark} We clarify the sheafification implicitly understood in \cite{AFL}  Thm.~4.8.
The restriction morphisms
\begin{equation}\label{rIJ}
\begin{split}
    r^M_{IJ}\colon&\mathcal{O}_M(U_J)\rightarrow\mathcal{O}_M(U_I)\\
    r_{IJ}\colon&\mathcal{F}_G(U_J)\rightarrow\mathcal{F}_G(U_I),
\end{split}
\end{equation}
$U_I\subseteq U_J$ (i.e. $J\subseteq I$),
for $\mathcal{O}_M$ and $\mathcal{F}_G$ are given by
the natural morphisms. Thus, (\ref{eq116}) and (\ref{eq115}) determine
presheaves $\mathcal{O}_M$ and $\mathcal{F}_G$ on the basis $\mathcal{B}$.

Define
$\cF_G^{\mathrm{co}\cO_q(P)}(U_I):=\cF_G(U_I)^{\mathrm{co}\cO_q(P)}$
for all $I$ and define the restriction morphisms 
\begin{equation}\label{eq130IJ}
    r^{\mathrm{co}H}_{IJ}:=r_{IJ}|_{\mathcal{F}_G^{\mathrm{co}H}(U_J)}
    \colon\mathcal{F}_G^{\mathrm{co}H}(U_J)
    \rightarrow\mathcal{F}_G^{\mathrm{co}H}(U_I)~,
\end{equation}
which are  well-defined
because  $H$-colinearity of the  $r_{IJ}$ in \eqref{rIJ} implies
$r_{IJ}(\mathcal{F}_G^{\mathrm{co}H}(U_J))\subseteq\mathcal{F}_G^{\mathrm{co}H}(U_I)$.
The restriction morphisms $r^{\mathrm{co}H}_{IJ}$ inherit the $\mathcal{B}$-presheaf
properties from the $r_{IJ}$, so that $\cF_G^{\mathrm{co}\cO_q(P)}$ is
a  $\mathcal{B}$-presheaf.
Using the technique of \cite{AFL} Prop.~4.7, one can prove the $\mathcal{B}$-presheaf 
equality $\cF_G^{\mathrm{co}\cO_q(P)}
=\cO_M$.
Using the following observation these 
$\mathcal{B}$-presheaves are extended to sheaves on the full
topological space, these are the sheaves of Theorem \ref{main-afl}.
\end{remark}
\begin{observation}\label{bsheaf}
If we have a presheaf $\cG_\cB$ defined on a basis $\cB$ for the
topology, we can always extend it to the presheaf $\cG$, where
$\cG(U):=\varprojlim \cG_\cB(V)$ for $V \subset U$,
$V \in \cB$, provided the target category admit inverse limits 
(see \cite{egaI} Chpt.~0, Sect.~3.2.1).
If $\mathcal{B}$ is finite the existence of $\varprojlim \cG_\cB(V)$
is granted.
Then, once we have a presheaf, we can always proceed to the sheafification
and obtain a sheaf on the topology with the same target category
(groups, $H$-comodule algebras, etc).
By universality, a morphism of presheaves on a basis
will extend uniquely first to a morphisms of presheaves and then
to a morphisms of their sheafification (cf. \cite{ha}  Chpt.~2 and also \cite{eh} for
more details on this standard construction).
To ease the notation, we shall use the same letter to denote
a presheaf on $\cB$, its extension to a presheaf and the sheafification.
\end{observation}

The topology of $M$ used in Theorem \ref{main-afl} is finite since
it is  induced by a finite cover $\{U_i\}_{i\in\mathcal{I}}$
of $M$ via the  finite basis $\mathcal{B}=\{U_I\}^{}_{I\in \mathcal{I}\mathcal{I}}$.
We can then define the \textit{minimal opens}
\begin{equation}\label{minop}
  U_p:=\mbox{$\bigcap^{}_{U_I\ni p\,}$}U_I\in  \mathcal{B}~.
\end{equation}
Given a presheaf $\mathcal{F}$ on $M$, its
stalk at $p\in M$ coincides with its sections on $U_p$, i.e. 
$\mathcal{F}_p=\mathcal{F}(U_p)$. This is immediate from the
definition of stalk 
since any open neighbourhood of $p$ includes $U_p$.
Now the stalks of the sheafification of $\mathcal{F}$ coincide with
the stalks of the  initial presheaf, and so $\mathcal{F}$ coincides
with its sheafification on the opens $U_p$.

\subsection{Covariant calculi on sheaves of comodule algebras}

We now give the definition of a covariant FODC
on a sheaf of comodule algebras. The bimodule property is entirely captured in
a sheaf-theoretic language, while we demand the Leibniz rule and
surjectivity property
of the differential on stalks.
To account for coinvariance
we have to consider the differential as a morphism of sheaves of comodules. 
We give some fundamental examples of covariant FODCi on sheaves and 
discuss the sheafs of base forms, horizontal forms
and coinvariant forms.

Given a morphism $\varphi\colon\mathcal{F}\rightarrow\mathcal{G}$ of sheaves on $M$ we denote the induced morphism on the stalks at $p\in M$ by
$\varphi_p\colon\mathcal{F}_p\rightarrow\mathcal{G}_p$.
We also recall that, for $\mathcal{F}$ a sheaf of algebras on a topological space $M$, 
a sheaf of $\mathcal{F}$-modules associates to each open $U$ in $M$ an 
$\mathcal{F}(U)$-module with compatible restriction maps.

\begin{definition} \label{fodc-def}
Let $\mathcal{F}$ be a sheaf of algebras on a topological space $M$.
A \textit{FODC on $\cF$} is 
a sheaf $\Upsilon$ of $\cF$-bimodules on $M$
with a morphism
\begin{equation*}
    \mathrm{d}\colon\cF\lra\Upsilon
\end{equation*}
of sheaves, such that for all $p\in M$ the induced maps on stalks $\mathrm{d}_p\colon\mathcal{F}_p\rightarrow
\Upsilon_p$ satisfy
\begin{enumerate}
\item[\it {i.)}] the Leibniz rule
$\mathrm{d}_p(fg)=(\mathrm{d}_pf)g+f\mathrm{d}_pg\hbox{ for all }
f,g\in \cF_p$

\item[\it{ii.)}] the surjectivity condition $\Upsilon_p=\cF_p\mathrm{d}_p\cF_p$.
\end{enumerate}
If $H$ is a Hopf algebra and $\mathcal{F}$ a sheaf of right $H$-comodule
algebras over $M$ we call a FODC $(\Upsilon,\mathrm{d})$ on $\mathcal{F}$
\textit{right $H$-covariant} if $\Upsilon$ is a sheaf of right
$H$-covariant $\mathcal{F}$-bimodules and $\mathrm{d}$ is a morphism of
sheaves of right $H$-comodules.
\end{definition}
Similarly one defines left $H$-covariant and $H$-bicovariant FODCi on
sheaves of left $H$-comodule algebras and $H$-bicomodule algebras,
respectively.

\begin{example}
\begin{enumerate}
\item[i.)] Let $M$ be an algebraic variety and $G$ an
  {(affine)} algebraic group acting on $M$.
Then, $\cO_M$, the structural sheaf of $M$, carries an $H=\cO(G)$ coaction, where
$\cO(G)$ denotes the global sections of the structural sheaf of $G$, which
carries a natural Hopf algebra structure.
Define $\Omega$ as the sheaf of K\"ahler differentials with $ \mathrm{d}\colon\cO_M \lra \Omega$
as in \cite{ha},  Sec. 8, II. As one can readily check, using the results
in \cite{ha}, $(\Omega,\mathrm{d})$ is a FODC on $\cO_M$.
 
\item[ii.)] Let the algebraic variety
$M=G$ be a simple complex algebraic group and let
$P$ be a parabolic subgroup of $G$. 
We can view the principal bundle $\pi\colon G \lra G/P$ in the
sheaf-theoretic language.
An example of FODC is again given by the sheaf of K\"ahler
differentials on $G$. Since $G$ is a principal bundle, we can also
define base forms as the differential forms over $G/P$: they are obtained
by considering horizontal forms invariant under the natural $P$ action. We 
are going to give a quantum version of this K\"ahler differential
construction in Section~\ref{DConSheaves}. The explicit example
of $\mathcal{O}_q(\mathrm{SL}_2(\mathbb{C}))$ is discussed in 
Section~\ref{sectionQDCProjSpace}.
\end{enumerate}
\end{example}

We now turn to examine base, horizontal and coinvariant forms in the
sheaf-theoretic context.

\medskip
Let $\mathcal{F}$ be a sheaf of right $H$-comodule algebras.
As for the $\mathcal{B}$-presheaf
  $\cF_G^{\mathrm{co}\cO_q(P)}$, we define
the presheaf $\mathcal{F}^{\mathrm{co}H}$ of algebras by assigning to
every open $U$ of $M$ the algebra
$
\mathcal{F}^{\mathrm{co}H}(U):=
\mathcal{F}(U)^{\mathrm{co}H}$ with restriction morphisms, cf. \eqref{eq130IJ},
\begin{equation}\label{eq130}
    r^{\mathrm{co}H}_{VU}:=r_{VU}|_{\mathcal{F}^{\mathrm{co}H}(U)}
    \colon\mathcal{F}^{\mathrm{co}H}(U)
    \rightarrow\mathcal{F}^{\mathrm{co}H}(V)~.
\end{equation} 
The presheaf $\mathcal{F}^{\mathrm{co}H}$ is a subsheaf of $\mathcal{F}$,
indeed it is the kernel of the sheaf morphism $\delta_R-\mathrm{id}\otimes 1_H\colon
\mathcal{F}\rightarrow\mathcal{F}\otimes H$.

\begin{definition}
Let $(\Upsilon,\mathrm{d})$ be a FODC on a sheaf 
$\mathcal{F}$ of right $H$-comodule algebras over $M$. For $U$ open in $M$
we define the $\mathcal{F}^{\mathrm{co}H}(U)$-subbimodules of
\begin{alignat*}{2}
  &\textit{base forms} &&\qquad\qquad\qquad\Upsilon_M(U):=\mathcal{F}^{\mathrm{co}H}(U)\mathrm{d}
  \mathcal{F}^{\mathrm{co}H}(U)~,\\
    &\textit{horizontal forms} &&\qquad\qquad\qquad\Upsilon^\mathrm{hor}(U) 
    :=\mathcal{F}(U) \Upsilon_M (U)~.
\end{alignat*}
If $(\Upsilon,\mathrm{d})$ is a right $H$-covariant FODC on
$\mathcal{F}$ we define the $\mathcal{F}^{\mathrm{co}H}(U)$-subbimodule of
\begin{equation*}
    \quad\qquad\qquad\textit{right $H$-coinvariant forms} \qquad\Upsilon^{\mathrm{co}H} (U)
    :=\{\omega\in\Upsilon(U) ~|~\Delta_R^{U}(\omega)=\omega\otimes 1\}~,
\end{equation*}
where  $\Delta_R^{U}\colon\Upsilon(U)
\rightarrow\Upsilon(U)\otimes H$
is the right $H$-coaction on $\Upsilon(U)$.
\end{definition}

\begin{proposition}\label{prop08}
The assignments \begin{equation*}
    \Upsilon_M\colon U\mapsto\Upsilon_M(U)~,
    \qquad
    \Upsilon^\mathrm{hor}\colon U\mapsto\Upsilon^\mathrm{hor}(U)~,
    \qquad
    \Upsilon^{\mathrm{co}H}\colon U\mapsto\Upsilon^{\mathrm{co}H}(U)~,
\end{equation*}
with restriction morphisms
\begin{equation}\label{eq111}
\begin{split}
    r^{\Upsilon_M}_{VU}&:=r^\Upsilon_{VU}|_{\Upsilon_M(U)}
    \colon\Upsilon_M(U)\rightarrow\Upsilon_M(V)~,\\
    r^\mathrm{hor}_{VU}&:=r^\Upsilon_{VU}|_{\Upsilon^\mathrm{hor}(U)}
    \colon\Upsilon^\mathrm{hor}(U)\to\Upsilon^\mathrm{hor}(V)~,\\
    r^{\Upsilon^{\mathrm{co}H}}_{VU}&:=r^\Upsilon_{VU}|_{\Upsilon^{\mathrm{co}H}(U)}
    \colon\Upsilon^{\mathrm{co}H}(U)\rightarrow\Upsilon^{\mathrm{co}H}(V)~,
\end{split}
\end{equation}
where $r^\Upsilon_{VU}\colon\Upsilon(U)\rightarrow\Upsilon(V)$
denotes the restriction morphism of $\Upsilon$,
define sheaves of $\mathcal{F}^{\mathrm{co}H}$-bimodules.
\end{proposition}
\begin{proof}
We first prove that the maps in (\ref{eq111}) are well-defined,
that is they map in the claimed codomains (which are submodules of $\Upsilon(V)$).
We show for example
$r^{\Upsilon_M}_{VU}(\Upsilon_M(U))\subseteq\Upsilon_M(V)$.
Let $f^i,g_i\in\mathcal{F}^{\mathrm{co}H}(U)\subseteq\mathcal{F}(U)
$ we have
$$
r^{\Upsilon_M}_{VU}(f^i\mathrm{d}_Ug_i)
=r^{\Upsilon}_{VU}(f^i\mathrm{d}_Ug_i)
=r_{VU}(f^i)r^{\Upsilon}_{VU}(\mathrm{d}_Ug_i)
=r_{VU}^{\mathrm{co}H}(f^i)\mathrm{d}_Vr_{VU}^{\mathrm{co}H}(g_i)
\in\Upsilon_M(V)~,
$$
where we used that $r^\Upsilon$ is left $\mathcal{F}$-linear, $\mathrm{d}$
is a morphism of sheaves and \eqref{eq130}. Similarly  $r^\mathrm{hor}_{VU}$ is
 well-defined. The proof for $r^{\Upsilon^{\mathrm{co}H}}_{VU}$ is
 straightforward from  $H$-colinearity of $r^{\Upsilon}_{VU}$.
The maps in (\ref{eq111}) are
$\mathcal{F}^{\mathrm{co}H}(U)$-bimodule maps (the
$\mathcal{F}^{\mathrm{co}H}(U)$-bimodule structure on the images being
given via     $r^{\mathrm{co}H}_{VU}$) because they are
restrictions of the  $\mathcal{F}(U)$-bimodule map
$r^\Upsilon_{VU}$.
Finally, the morphisms in  (\ref{eq111}) for $V,U$ open in $M$ define
the presheaves $\Upsilon_M$,  $\Upsilon^\mathrm{hor}$,
$\Upsilon^{\mathrm{co}H}$ since they are restrictions of the morphisms $r^\Upsilon_{VU}$
defining the presheaf $\Upsilon$.
Since $\Upsilon^{\mathrm{co}H}$ is the kernel of the sheaf morphism
$\Delta_R-\mathrm{id}\otimes 1_H\colon\Upsilon\rightarrow\Upsilon\otimes H$, it is a sheaf.
By common abuse of notation
we shall denote by
$\Upsilon_M$ and $\Upsilon^\mathrm{hor}$ the sheafifications of the corresponding presheaves. These are sheaves in the category of $\mathcal{F}^{\mathrm{co}H}$-bimodules.
\end{proof}

\begin{corollary}\label{basef-cor}
Let $(\Upsilon,\mathrm{d})$ be a FODC on a sheaf 
$\mathcal{F}$ of right $H$-comodule algebras over $M$.
Then $(\Upsilon_M,\mathrm{d}_M)$ with 
$\mathrm{d}_M:=\mathrm{d}|_{\mathcal{F}^{\mathrm{co}H}}
\colon\mathcal{F}^{\mathrm{co}H}\rightarrow\Upsilon_M$ is a FODC on
$\mathcal{F}^{\mathrm{co}H}$.
\end{corollary}

\begin{proof}
We already noted that the presheaf $\mathcal{F}^{\mathrm{co}H}$
is a subsheaf of $\mathcal{F}$. Since $\mathrm{d}_M$ is the restriction
of $\mathrm{d}\colon\mathcal{F}\rightarrow\Upsilon$ to $\mathcal{F}^{\mathrm{co}H}$
it is a morphism of sheaves and satisfies the Leibniz
rule on stalks.
On the level of presheaves we have the surjectivity condition on stalks
 $\mathcal{F}^{\mathrm{co}H}_p\,
(\mathrm{d}_M)^{}_p\!\:\mathcal{F}^{\mathrm{co}H}_p=(\Upsilon_M)_p$, for any $p\in M$.
Since the stalks of the presheaf coincide with the stalks of its
sheafification we conclude that the surjectivity condition of the
sheaf morphism $\mathrm{d}_M$ is satisfied (cf. the analogue of  {\it{ii.)}} in Definition
\ref{fodc-def}).
\end{proof}

Let $\Upsilon^{\mathrm{co}H}\cap\Upsilon^\mathrm{hor}$ denote the sheafification
of the presheaf obtained intersecting the subsheaves
$\Upsilon^{\mathrm{co}H}$ and $\Upsilon^\mathrm{hor}$ of $\Upsilon$.

\begin{theorem}\label{thmcharbase}
Let the base ring $\Bbbk$ be a field. For any right $H$-covariant FODC $(\Upsilon,\mathrm{d})$ on a QPB
$\mathcal{F}$ 
we have an isomorphism 
$$
\Upsilon_M\cong\Upsilon^{\mathrm{co}H}\cap\Upsilon^\mathrm{hor}
$$
of sheaves of $\mathcal{O}_M$-bimodules.
\end{theorem}
\begin{proof}
The natural inclusion $\Upsilon_M\to\Upsilon^{\mathrm{co}H}\cap\Upsilon^\mathrm{hor}$
(by which we mean the sheafification of the corresponding morphism of presheaves)
is a morphism of sheaves. Recalling from Remark \ref{CiPa} that stalks are
principal comodule algebras, the induced inclusion on stalks,
$(\Upsilon_M)^{}_p\to\Upsilon^{\mathrm{co}H}_p\cap\Upsilon^\mathrm{hor}_p$,
is an isomorphism thanks to Theorem \ref{prop-ff}. Thus
$\Upsilon_M\to\Upsilon^{\mathrm{co}H}\cap\Upsilon^\mathrm{hor}$ is  an isomorphism of sheaves.
\end{proof}

\subsection{Ore extension of covariant calculi}\label{SectionOre}

As a tool for the following sections, we now discuss the Ore extension
of covariant calculi.
We recall that if $a \in A$ is a (right) Ore element, we denote by
$AS^{-1}$ or by $A[a^{-1}]$ the
Ore localization of $A$ at the multiplicative 
right Ore set $S=\{a^m\}_{m \in \Z_\geq 0}$
(for more details see  \cite{lam} Chpt.~4, Sect.~10 
and \cite{kassel} Sect.~I.7.).

\medskip
Let $(\Gamma,\mathrm{d})$ be a FODC 
on an algebra $A$. 
Let $a$ be an Ore element in $A$. We define the $A[a^{-1}]$-bimodule
\begin{equation}\label{eqGAGA}
\Gamma_a:=A[a^{-1}] \, \Gamma \, A[a^{-1}]:=A[a^{-1}]\otimes_A\Gamma\otimes_AA[a^{-1}]
\end{equation}
and the $\Bbbk$-linear map
\begin{equation}\label{eq110}
    \mathrm{d}_a:A[a^{-1}] \lra \Gamma_a, \qquad \mathrm{d}_a(a')=\begin{cases} \mathrm{d}a' & a' \in A \\
    -a^{-1}\, \mathrm{d}a \, a^{-1} & a'=a^{-1} \end{cases},
\end{equation}
where we extend $\mathrm{d}_a$ to $A[a^{-1}]$ by the Leibniz rule.

\begin{lemma} \label{local-lemma}
$(\Gamma_a,\mathrm{d}_a)$ is a FODC on $A[a^{-1}]$.
\end{lemma}

\begin{proof}
We need to verify the properties of Definition~\ref{def01}.
The first two properties are satisfied by definition. Moreover
by the equality
$$
a'\,\mathrm{d}a''\,a^{-n}=a'  \,\mathrm{d}_a(a''a^{-n})-a' \, a'' \, \mathrm{d}_a(a^{-n}),
$$
where $a',a''\in A$ and $n\in\mathbb{N}$,
we see that we can express any element of $\Gamma_a$ as 
an element in $A[a^{-1}] \, \mathrm{d}_a A[a^{-1}]$.
\end{proof}

With a slight abuse of terminology, we call $(\Gamma_a,\mathrm{d}_a)$ 
the \textit{localization} of $(\Gamma,\mathrm{d})$ via the Ore
element $a\in A$. 

We now show that the covariant properties of the FODC $\Gamma$ are
preserved under localization.

\begin{lemma}\label{local-lemma2}
Let $(\Gamma,\mathrm{d})$ be a right $H$-covariant FODC on a 
right $H$-comodule algebra $(A,\delta_R)$ and $a\in A$ an Ore element.
Assume that $\delta_R(a)\in A\otimes H$ is invertible. Then
\begin{enumerate}
\item[i.)] $\delta_R$ extends to a
right $H$-coaction $\delta^a_R\colon A[a^{-1}]\rightarrow A[a^{-1}]\otimes H$,
structuring $(A[a^{-1}],\delta_R^a)$ as a right $H$-comodule algebra;

\item[ii.)] the FODC $(\Gamma_a,\mathrm{d}_a)$ on $(A[a^{-1}],\delta_R^a)$
is right $H$-covariant;
\end{enumerate}
\end{lemma}

\begin{proof}
\begin{enumerate}
\item[i.)] We set $\delta_R^a(a'):=\delta_R(a')$ for all $a'\in A$, furthermore
$\delta_R^a(a^{-1}):=\delta_R(a)^{-1}$ and extend $\delta_R^a$ as an algebra homomorphism
to $A[a^{-1}]$. From the coaction properties of $\delta_R$ it follows that $\delta_R^a$
is a right $H$-coaction on $A[a^{-1}]$ and by construction it is compatible with the 
algebra structure of $A[a^{-1}]$.

\item[ii.)] There is a canonical right $H$-coaction $\Delta_R^a$ on
$\Gamma_a:=A[a^{-1}]\otimes_A\Gamma\otimes_AA[a^{-1}]$ given by the diagonal coaction.
We show that $\Delta_R^a$ structures $\Gamma_a$ as a right $H$-covariant 
$A[a^{-1}]$-bimodule such that $\mathrm{d}_a$ is right
$H$-colinear. The claim
that $(\Gamma_a,\mathrm{d}_a)$ is a right $H$-covariant FODC on
$A[a^{-1}]$ then
follows from Proposition~\ref{lemma03}. Consider an arbitrary element 
$b^i\otimes_A\omega^i\otimes_Ac^i\in\Gamma_a$ with $b^i,c^i\in A[a^{-1}]$ and
$\omega^i\in\Gamma$, sum over $i$ understood. Since, according to $i.)$, $\delta_R^a$ is an algebra homomorphism, we
obtain 
\begin{align*}
    \Delta_R^a(b\cdot(b^i\otimes_A\omega^i\otimes_Ac^i)\cdot c)
    &=(bb^i)_0\otimes_A\omega^i_0\otimes_A(c^ic)_0\otimes(bb^i)_1\omega^i_1(c^ic)_1\\
    &=b_0(b^i_0\otimes_A\omega^i_0\otimes_Ac^i_0)c_0\otimes b_1(b^i_1\omega^i_1c^i_1)c_1\\
    &=\delta_R^a(b)\Delta_R^a(b^i\otimes_A\omega^i\otimes_Ac^i)\delta_R^a(c)
\end{align*}
for all $b,c\in A[a^{-1}]$. That is $\Gamma_a$ is a right $H$-covariant 
$A[a^{-1}]$-bimodule.
Next, applying $\mathrm{d}_a\otimes\mathrm{id}_H$ to the equation
$1_A\otimes 1_H=\delta_R(a)\delta_R^a(a^{-1})$ we obtain
\begin{align*}
    0&=(\mathrm{d}a_0\otimes a_1)\cdot\delta_R^a(a^{-1})
    +\delta_R^a(a)\cdot(\mathrm{d}_aa^{-1}_0\otimes a^{-1}_1)
    =\Delta_R(\mathrm{d}a)\cdot\delta_R^a(a^{-1})
    +\delta_R^a(a)\cdot(\mathrm{d}_aa^{-1}_0\otimes a^{-1}_1),
\end{align*}
where in the last expression we used right $H$-colinearity of
  $\mathrm{d}$. This implies
\begin{equation}\label{daDD}
\mathrm{d}_aa^{-1}_0\otimes a^{-1}_1
=-\delta_R^a(a^{-1})\cdot\Delta_R(\mathrm{d}a)\cdot\delta_R^a(a^{-1})
=-\Delta_R^a(a^{-1}\mathrm{d}(a)a^{-1})
=\Delta_R^a(\mathrm{d}_a(a^{-1}))~.
\end{equation}
The Leibniz rule of $\mathrm{d}_a$,
the algebra homomorphism property of $\delta_R^a$,
  the $A$-bilinearity of $\Delta_R^a$ and \eqref{daDD}
then imply  that $\mathrm{d}_a$
is right $H$-colinear on any power of $a^{-1}$ and furthermore on any element
of $A[a^{-1}]$.\qedhere
\end{enumerate}
\end{proof}

Let $A$ be an algebra,
$a_i\in A$, $i=1,\ldots,n$ be Ore elements and  $S_{i}=\{a^r_{i}, r\in \Z_{\geq 0}\}$.  For a multi-index $I=(i_1,\ldots,i_k)\in\{1,\ldots,n\}^k$ we use the notation
$A_I=AS_{i_1}^{-1}\ldots S_{i_k}^{-1}$, or equivalently the  notation
\begin{equation*}
    A_I=A_{(i_1,\ldots,i_k)}:=A[a_{i_1}^{-1},\ldots,a_{i_k}^{-1}]:=(\cdots(A[a_{i_1}^{-1}])
    \cdots )[a_{i_k}^{-1}]
\end{equation*}
for the $k$-fold Ore localization.
Given a FODC $(\Gamma,\mathrm{d})$ on $A$ we similarly define the $k$-fold Ore localization of the
  $A$-bimodule $\Gamma$ by
$$
\Gamma_I=\Gamma_{(i_1,\ldots,i_k)}:=(\cdots(\Gamma_{a_{i_1}})_{a_{i_2}}\cdots)_{a_{k}}:=A_I\cdots
\otimes_{A_{(i_1,i_2)}}(A_{(i_1,i_2)}\otimes_{A_{i_1}}(A_{i_1}\otimes_A\Gamma\otimes_A
A_{i_1})\otimes_{A_{i_1}}A_{(i_1,i_2)})\otimes_{A_{(i_1,1_2)}}\cdots A_I$$
and of the exterior derivative by
\beq\label{ore-der}
\mathrm{d}_I=\mathrm{d}_{(i_1,\ldots,i_k)}:=(\cdots(\mathrm{d}_{a_{i_1}})_{a_{i_2}}\cdots)_{a_{i_k}}\colon A_{I}~
\rightarrow\Gamma_{I}~.
\eeq

\begin{lemma}\label{local-lemmaII}
Let $A$ be an algebra and
$a_i\in A$, $i=1,\ldots,n$ be Ore elements such that subsequent localizations
do not depend on the order. Then, subsequent localizations of a FODC on $A$ are
independent of the order, as well.
\end{lemma}
\begin{proof}
Consider a FODC $(\Gamma,\mathrm{d})$ on $A$.
For two Ore elements $a_i,a_j\in A$, using that  $A_{(i,j)}=A_{(j,i)}$
we have
\begin{align*}
    (\Gamma_{a_i})_{a_j}
    &=A_{(i,j)}\otimes_{A_i}(A_i\otimes_{A}\Gamma\otimes_{A}A_i)\otimes_{A_i}A_{(i,j)}\\
    &=A_{(i,j)}\otimes_A\Gamma\otimes_AA_{(i,j)}\\
    &=A_{(j,i)}\otimes_{A_j}(A_j\otimes_A\Gamma\otimes_AA_j)\otimes_{A_j}A_{(j,i)}\\
    &=(\Gamma_{a_j})_{a_i}
\end{align*}
as an equation of $A_{(i.j)}$-bimodules. In particular,
\begin{align*}
\mathrm{d}_{(i,j)}, \mathrm{d}_{(j,i)}\colon A_{(i.j)}
\rightarrow\Gamma_{(i,j)}    
\end{align*}
have the same domain and codomain. They coincide
on $A$ and by definition (\ref{eq110}) they also agree on $a_i^{-1}$ and $a_j^{-1}$. Both
differentials are extended by the Leibniz rule, therefore
$\mathrm{d}_{(i,j)}=\mathrm{d}_{(j,i)}$. An easy inductive argument on
the number of Ore localization then proves the lemma.
\end{proof}

This
lemma shows in particular that the $k$-fold Ore localization of the
$A$-bimodule $\Gamma$ is \mbox{$\Gamma_I=A_I\otimes_A\Gamma\otimes_AA_I$,}
that for short we frequently write as $\Gamma_I=A_I\Gamma A_I$.
Recalling also Lemma \ref{local-lemma2} the following proposition is straightforward.

\begin{proposition}\label{ore-calc}
Let $(\Gamma,\mathrm{d})$ be a FODC on a 
right $H$-comodule algebra $(A,\delta_R)$ and
$a_i\in A$, $i=1,\ldots,n$ be Ore elements such that $\delta_R(a_i)$ are
  invertible elements of $A_i\otimes H$. Then
\begin{equation*}
    \Gamma_I=A_I\otimes_A\Gamma\otimes_AA_I, \qquad
\mathrm{d}_I\colon A_I\rightarrow\Gamma_I
\end{equation*}
is a right $H$-covariant FODC on $A_I$. If subsequent Ore
localizations of $A$ are independent on the order  also the subsequent
localizations of the FODC $(\Gamma, d)$ are order independent.
\end{proposition}

\begin{remark}
We have considered FODCi on  Ore
extended algebras. In Section \ref{pc-sec}, associated with a
FODC $(\Gamma_A,\rm{d}_A)$ there is the graded algebra of 0- and 1-forms $\Omega^{\leqslant 1}_A=A\oplus\Gamma_A$. It is natural to
consider also Ore extensions of this graded algebra (which are
necessarily by degree zero elements). In this case an Ore element $a$
is in particular Ore in the zero degree subalgebra $A$ and we also require to
be able to order on the right (and on the left) of 1-forms the
negative powers $a$. 
Thus, if $a\in A$ is Ore in $\Omega^{\leqslant 1}_A$ the graded algebra $A_a\oplus\Gamma_a$ with $\Gamma_a$ defined in \eqref{eqGAGA} is the Ore extension of $A\oplus\Gamma$ corresponding to $a$.
This Ore condition on 1-forms is met in the examples we shall consider.
\end{remark}

\subsection{Covariant differential calculi on sheaves over quantum 
projective varieties}\label{DConSheaves}

In this section we consider all modules over the complex numbers
$\mathbb{C}$. When considering the quantizations  $\mathcal{O}_q(G)$,
$\mathcal{O}_q(P)$ we therefore specialize $q$ to be in $\mathbb{C}$.
Fix a complex semisimple algebraic group $G$ with parabolic subgroup $P$,
quantizations $\mathcal{O}_q(G)$, $\mathcal{O}_q(P)$ and the sheaves
$\mathcal{F}_G$ and $\mathcal{O}_M$ as in Theorem~\ref{main-afl},
where the QPB property of $\cF_G$  is
with respect to the finite open cover $\{U_i\}_{i\in \mathcal{I}}$, $\mathcal{I}=\{1,\dots,n\}$ determined by a quantum section.
Consider the (finite) topology generated by the open cover; it
will have topological basis
${\mathcal{B}}=\{U_I\}_{I\in\mathcal{II}}$, where we recall that $\mathcal{II}$ is
the set of ordered multi-indices
$I=(i_1,\dots, i_r)$, $1\leq i_1 < \dots < i_r \leq n$ with
  $r=1,2,\ldots n$ (and we also consider the case $r=0$ corresponding to the empty set) and $U_I:=\cap_{i \in I} U_i$.

Given a right $\mathcal{O}_q(P)$-covariant FODC $(\Gamma,\mathrm{d})$ on
$\mathcal{O}_q(G)$ we induce a $\mathcal{O}_q(P)$-covariant FODC 
on $\mathcal{F}$ according to 
Section~\ref{SectionOre} as follows. We set on the basis $\mathcal{B}$ of open sets $U_I$
\begin{equation}\label{cactusGI}
    \Upsilon_G(U_I):=
\mathcal{F}_G(U_I)\otimes_{\mathcal{O}_q(G)}
    \Gamma\otimes_{\mathcal{O}_q(G)}\mathcal{F}_G(U_I)~, 
\end{equation}
and define restriction morphisms
$r^\Upsilon_{JI}\colon\Upsilon_G(U_I)\rightarrow\Upsilon_G(U_J)$ for all
$U_J\subseteq U_I$ by
\begin{equation}\label{eq108}
    r^\Upsilon_{JI}(f\otimes_{\mathcal{O}_q(G)}
    \omega\otimes_{\mathcal{O}_q(G)}g)
    :=r_{JI}(f)\otimes_{\mathcal{O}_q(G)}
    \omega\otimes_{\mathcal{O}_q(G)}r_{JI}(g)~,
\end{equation}
where $f,g\in\mathcal{F}_G(U_I)$ and $\omega\in\Gamma$.
Since $r_{JI}$ are algebra homomorphisms which are the identity on 
$\mathcal{O}_q(G)$ the expression (\ref{eq108}) is well-defined on the
algebraic tensor product. The right $\mathcal{O}_q(P)$-comodule algebra map 
$r_{JI}$ defines an $\mathcal{F}_G(U_I)$-bimodule structure on $\Upsilon_G(U_J)$ 
and $r^\Upsilon_{JI}\colon\Upsilon_G(U_I)\rightarrow\Upsilon_G(U_J)$ is a
map of right $\mathcal{O}_q(P)$-covariant
$\Upsilon_G(U_J)$-bimodules, i.e., it is  $\mathcal{O}_q(P)$-colinear and
$r^\Upsilon_{JI}(f\cdot\theta\cdot g)
=r_{JI}(f)r^\Upsilon_{JI}(\theta)r_{JI}(g)$ for all $f,g\in\mathcal{F}_G(U_I)$
and $\theta\in\Upsilon_G(U_I)$. The equality
$r^\Upsilon_{KJ}\circ r^\Upsilon_{JI}=r^\Upsilon_{KI}$
for any three opens $U_K\subseteq U_J\subseteq U_I$
then follows from that for the restriction morphisms of
$\mathcal{F}_G$. We have defined the $\mathcal{B}$-presheaf
  $\Upsilon_G$  of right  $\mathcal{O}_q(P)$-covariant $\mathcal{F}_G$-bimodules.
We denote its sheafification by the same symbol $\Upsilon_G$.

\begin{theorem}[Induced calculus on the sheaf $\mathcal{F}_G$]\label{thm02}
Let $(\Gamma,\mathrm{d})$ be a right covariant FODC on the Hopf
algebra $\cO_q(G)$ and
$\mathcal{F}_G$,  $\Upsilon_G$ the sheaves of
 right $\mathcal{O}_q(P)$-comodule algebras and
of  right $\mathcal{O}_q(P)$-covariant $\cF_G$-bimodules
defined above.

\begin{enumerate}
\item[i.)] The linear maps 
$$
{\mathrm{d}}_I:\mathcal{F}_G(U_I) \lra\Upsilon_G(U_I)~,
$$
defined in (\ref{ore-der}) for $A_I=\mathcal{F}_G(U_I)$
and $\Gamma_I=\Upsilon_G(U_I)$, define a morphism of sheaves of right
$\cO_q(P)$-comodules $\mathrm{d}\colon
\mathcal{F}_G\to \Upsilon_G$. This gives a
right $\mathcal{O}_q(P)$-covariant FODC
$(\Upsilon_G,\mathrm{d})$ 
on the sheaf $\mathcal{F}_G$.
\item[ii.)] The FODC $(\Upsilon_G,\mathrm{d})$ on the sheaf 
$\cF_G$  
induces a FODC $(\Upsilon_M,\mathrm{d}_M)$ on the sheaf
$\cF_G^{\coi \cO_q(P)}=\cO_M$.
\item[iii.)] If $\mathcal{F}_G$ is a QPB the sheaf of base forms is isomorphic, as a sheaf of $\cO_M$-bimodules, to the intersection
of that of horizontal and $H$-coinvariant forms: $\Upsilon^{}_M\cong\Upsilon_G^\mathrm{hor}
\cap\Upsilon_G^{\mathrm{co}\:\!\mathcal{O}_q(P)}$. 
\end{enumerate}
\end{theorem}

\begin{proof}
Let $\mathcal{F}_G$ be the $\mathcal{B}$-presheaf defined in
\eqref{eq115} and $\Upsilon_G$ the $\mathcal{B}$-presheaf defined in
\eqref{cactusGI}, \eqref{eq108}.
We  show that the assignment $U_I\mapsto {\mathrm{d}_I}$
defines a morphism of
$\mathcal{B}$-presheaves  $\mathrm{d}:\mathcal{F}_G\to\Upsilon_G$. 
That is, $\mathrm{d}_J(r_{JI}(f))=
r^\Upsilon_{JI}(\mathrm{d}_If)$ for all $f\in\mathcal{F}_G(U_I)$
and $U_J\subseteq U_I$. This is clear if $f\in\mathcal{O}_q(G)\subseteq
\mathcal{F}_G(U_I)$. Also for $f=s_{i_\ell}^{-1}$, where
$1\leq\ell\leq n$, $i_\ell\in I=(i_1,\ldots,i_r)$,
we obtain
\begin{align*}
    r^\Upsilon_{JI}(\mathrm{d}_Is_{i_\ell}^{-1})
    &=-r^\Upsilon_{JI}(s_{i_\ell}^{-1}\otimes_{\mathcal{O}_q(G)}
    \mathrm{d}s_{i_\ell}\otimes_{\mathcal{O}_q(G)}s_{i_\ell}^{-1})\\
    &=-r_{JI}(s_{i_\ell}^{-1})\otimes_{\mathcal{O}_q(G)}
    \mathrm{d}s_{i_\ell}\otimes_{\mathcal{O}_q(G)}r_{JI}(s_{i_\ell}^{-1})\\
    &=-s_{i_\ell}^{-1}\otimes_{\mathcal{O}_q(G)}
    \mathrm{d}s_{i_\ell}\otimes_{\mathcal{O}_q(G)}s_{i_\ell}^{-1}\\
    &= ~\mathrm{ d}_J(s_{i_\ell}^{-1})\\
    &= ~\mathrm{ d}_J(r_{JI}(s_{i_\ell}^{-1}))\,.
\end{align*}
By the Leibniz rule, cf. Proposition~\ref{ore-calc}, this
$\mathcal{B}$-presheaf property extends to any element
$f\in\mathcal{F}_G(U_I)$, $\mathrm{d}_J(r_{JI}(f))=
r^\Upsilon_{JI}(\mathrm{d}_If)$. 

Recalling from Proposition~\ref{ore-calc} that $(\Upsilon_G(U_I),
\mathrm{d}_I)$ is a right $\mathcal{O}_q(P)$-covariant FODC and that $\mathcal{F}_G$
and $\Upsilon_G$ are  $\mathcal{B}$-presheaves of right $\mathcal{O}_q(P)$-comodules,
we see that the $\mathcal{B}$-presheaf morphism
$\mathrm{d}\colon\mathcal{F}_G\to\Upsilon_G$ is compatible
with the right $\mathcal{O}_q(P)$-comodule structure.
Furthermore, the Leibniz rule and surjectivity condition are satisfied for all opens
$U_I$ in $\mathcal{B}$. We sheafify  according to Observation~\ref{bsheaf} and obtain
the sheaf $\Upsilon_G$ of right $\mathcal{O}_q(P)$-covariant
$\mathcal{F}_G$-bimodules with the sheaf morphism
$\mathrm{d}\colon\mathcal{F}_G\to\Upsilon_G$. Recalling that stalks are preserved under sheafification
and that the
stalks at $p$ are respectively  $(\cF_G)_p =\cF_G(U_p)$,
$(\Upsilon_G)_p=\Upsilon_G(U_p)$ we infer that the Leibniz rule and
the surjectivity condition  for the sheafified morphism
$\mathrm{d}\colon\mathcal{F}_G\to\Upsilon_G$  are satisfied on the 
stalks.
This proves {\it{i.)}}.
Recalling that the sheaf equality $\cF_G^{\coi \cO_q(P)}=\cO_M$ was proven in Theorem~\ref{main-afl} we have that {\it{ii.)}} follows directly from Corollary~\ref{basef-cor}. The last point {\it{iii.)}} follows from Theorem~\ref{thmcharbase}.
\end{proof}

\subsection{Principal covariant calculi on quantum principal  bundles}\label{sectionQDCProjSpace}
A principal covariant
calculus on a QPB $\mathcal{F}$ is
 a right $H$-covariant FODC $(\Upsilon,\mathrm{d})$ on the sheaf $\mathcal{F}$
 (see Definition~\ref{fodc-def}) with an exact sequence  on the
 stalks.

\begin{definition}\label{DConQPB}
Consider a QPB $\mathcal{F}$ on a quantum ringed space $(M,\cO_M)$
and a FODC $(\Upsilon,\mathrm{d})$ on $\mathcal{F}$.
Given a left covariant FODC $(\Gamma_H,\mathrm{d}_H)$ on $H$
we call $(\Upsilon,\mathrm{d})$ a \textit{principal
calculus} on $\mathcal{F}$ if for all $p\in M$ we have the exact
sequence of stalks
\begin{equation}\label{eq109}
    0\rightarrow\mathcal{F}_p\otimes_{(\mathcal{O}_M)^{}_p}(\Upsilon_M)^{}_p
    \rightarrow\Upsilon_p\xrightarrow{\mathrm{ver}_p}
    \mathcal{F}_p\square_H
    \Gamma_H\rightarrow 0~,
\end{equation}
where $\mathrm{ver}_p$ is the vertical map defined in Section~\ref{pc-sec} and $\Upsilon_M$ denotes the sheaf of base forms
defined in Proposition~\ref{prop08}.
If, in addition, $(\Upsilon,\mathrm{d})$ is right $H$-covariant
and $(\Gamma_H,\mathrm{d}_H)$ is bicovariant, we say that
$(\Upsilon,\mathrm{d})$ is
a \textit{principal covariant calculus} on the sheaf $\mathcal{F}$.
\end{definition}

We next apply Theorem~\ref{thm02} and consider three FODCi on QPBs over 
 $\mathbf{P}^1(\mathbb{C})$, the second one being a principal calculus and the third one a principal covariant calculus. The first two have total space algebra $\mathcal{O}_q(\mathrm{SL}_2(\mathbb{C}))$
with total space calculus the 4D and 3D ones of  Examples~\ref{sl2}
and \ref{ex01}. The third example has total space algebra $\mathcal{O}_q(\mathrm{GL}_2(\mathbb{C}))$ with total space calculus the 
4D one of Example~\ref{eq04}.

\subsubsection{First order differential calculus via Ore localization of 4D Calculus on \texorpdfstring{$\mathcal{O}_q(\mathrm{SL}_2(\mathbb{C}))$}{Oq(SL2(C))}}

Consider the principal bundle 
$\wp: \rSL_2(\C) \lra \rSL_2(\C)/P \simeq \bP^1(\C)$,
where $P$ is the upper Borel in $\rSL_2(\C)$. 
Take $V_i$ to be the subset of matrices in $\rSL_2 (\C)$ with entry $(i,1)$ not 
equal to zero. 
 We observe that $\{V_1,V_2\}$ is an open cover of $\rSL_2 (\C)$. Define $U_i=\wp(V_i)$
and observe that
$\{U_1,U_2\}$ is an open cover of $\bP^1(\C)$ since $\wp$ is an open
map.
This is the cover obtained via the quantum section $\alpha\in\mathcal{O}(\mathrm{SL}_2(\mathbb{C}))$ with coproduct
$
\Delta(\alpha)=\alpha\otimes\alpha + \beta\otimes\gamma.
$
In fact $V_1=\{g \in \rSL_2(\C) \,|\, \alpha(g) \neq 0\}$ and
$V_2=\{g \in \rSL_2(\C) \,|\, \gamma(g) \neq 0\}$.
Explicitly, $U_1,U_2$ are the opens obtained by removing the north or
south pole. 
The coaction 
$$
\delta_R\colon\mathcal{O}(\mathrm{SL}_2(\mathbb{C}))\rightarrow\mathcal{O}(\mathrm{SL}_2(\mathbb{C}))\otimes\mathcal{O}(P)~,~~~~~~
\begin{pmatrix}
\alpha & \beta\\
\gamma & \delta
\end{pmatrix}\mapsto\begin{pmatrix}
\alpha & \beta\\
\gamma & \delta
\end{pmatrix}\dot\otimes\begin{pmatrix}
t & p\\
0 & t^{-1}
\end{pmatrix}
$$
uniquely extends to coactions on the localizations
$\cO(\rSL_2)[\ap^{-1}]$ and $\cO(\rSL_2)[\cm^{-1}]$ by defining
$\delta_R \ap^{-1}=\ap^{-1}\otimes t^{-1}$ and
$\delta_R \cm^{-1}=\cm^{-1}\otimes t^{-1}$, respectively.
The coinvariant subalgebras for $\cO(\rSL_2)[\ap^{-1}]$ 
and $\cO(\rSL_2)[\cm^{-1}] $ respectively are isomorphic to 
$\C[u]$ and  $\C[v]$ with $u:=\ap^{-1}\cm$ and $v:=\ap\cm^{-1}$ and they 
are the coordinate rings of the affine algebraic
varieties corresponding to the opens $U_1, U_2$ in $\rSL_2(\C)/P\simeq \bP^1(\C)$.

The quantum deformation of this construction has been studied in \cite{AFL} as an
example of QPB over projective base. In summary, starting
with the Hopf algebra $A:=\cO_q(\rSL_2(\mathbb{C}))$ and its parabolic
quotient Hopf algebra $H:=\mathcal{O}_q(P)$ as in (\ref{quotient})
we define the noncommutative
localizations $A_{1}:=A[\ap^{-1}]$ and $A_2:=A[\cm^{-1}]$ and structure them
as right $H$-comodule algebras in complete analogy to the classical setting.
The subalgebras of right $H$-coinvariants are given by
\begin{equation*}
    B_1=\C_q[\cm\ap^{-1}] \simeq \C_q[u], \qquad B_2=\C_q[\ap \cm^{-1}] \simeq \C_q[v]~
\end{equation*}
and the ringed space $(\bP^1(\C),
\cO_{\bP^1(\C)})$ can then be constructed as
\begin{equation}\label{OqP1C}
    \cO_{\bP^1(\C)}(U_i):=B_{i}~,~~\cO_{\bP^1(\C)}(U_{12}):=B_{12}
:=B_{1}[u^{-1}]~,~~\cO_{\bP^1(\C)}(\bP^1(\C)):=\C~,
\end{equation}
$i\in\{1,2\}$,
with the only non-trivial restriction map given by $r_{{12},2}:B_{\,2}\to B_{12}$, $v\mapsto u^{-1}$.
Moreover we observe that
$$
\cF_{\mathrm{SL}_q}(U_i):=A_{i}~, \quad \cF_{\mathrm{SL}_q}(U_{12}):=A_{12}:=A_{1}[\cm^{-1}]=A_{2}[\ap^{-1}]~, \quad \cF_{\mathrm{SL}_q}(\bP^1(\C))=A~,\quad \cF_{\mathrm{SL}_q}(\emptyset):=\{0\}~,
$$
$i\in\{1,2\}$, defines a sheaf of right
$\cO_q(P)$-comodule algebras on $\bP^1(\C)$, which we denote by
$\mathcal{F}_{\mathrm{SL}_q}$.
Endowed with the cleaving maps,
\begin{equation}\label{cleavmaps}
j_1\colon H\rightarrow A_1~,~~~~\begin{pmatrix}
t & p\\
0 & t^{-1}
\end{pmatrix}\mapsto\begin{pmatrix}
\alpha & \beta\\
0 & \alpha^{-1}
\end{pmatrix}~,~~~~~~~~~~~
j_2\colon H\rightarrow A_2~,~~~~\begin{pmatrix}
t & p\\
0 & t^{-1}
\end{pmatrix}\mapsto\begin{pmatrix}
\gamma & \delta\\
0 & \gamma^{-1}
\end{pmatrix}~,
\end{equation}
$\mathcal{F}_{\mathrm{SL}_q}$ becomes a QPB on $\bP^1(\C)$ (on $A_{12}$ we
can choose $j_{12}=j_1\colon H\rightarrow A_1\subset A_{12}$ as a cleaving map
or alternatively $j_2$). Notice that the stalk
$(\mathcal{F}_{\mathrm{SL}_q})_p$ of $\mathcal{F}_{\mathrm{SL}_q}$ at
$p\in\mathbf{P}^1(\mathbb{C})$ equals $\mathcal{F}_{\mathrm{SL}_q}(U_{12})$
if $p\in U_{12}$ and $\mathcal{F}_{\mathrm{SL}_q}(U_i)$ if 
$p\in U_i\setminus U_{12}$. Further notice that
$\mathcal{B}=\{\emptyset,U_1,U_2,U_{12}\}$.
Recalling the induced calculus on a sheaf of
Theorem~\ref{thm02} we have

\begin{proposition}\label{Prop4D}
Let $(\Gamma,\mathrm{d}):=(\Gamma^+_{\mathrm{SL}},\mathrm{d}^+_{\mathrm{SL}})$ be the $4$-dimensional bicovariant FODC on $\cO_q(\rSL_2(\mathbb{C}))$
of Example~\ref{sl2}. The induced right $\mathcal{O}_q(P)$-covariant FODC
$(\Upsilon_{\mathrm{SL}_q},\mathrm{d})$ on $\mathcal{F}_{\mathrm{SL}_q}$
is a free left $\mathcal{F}_{\mathrm{SL}_q}$-module of dimension $4$.
\end{proposition}
\begin{proof}
The $A_I$-bimodules $\Upsilon_{\mathrm{SL}_q}(U_I)=\Gamma_I=A_I\Gamma A_I$
of the induced
calculus of Theorem~\ref{thm02} are constructed in Proposition~\ref{ore-calc}.
We show that $\Gamma_I=A_I\Gamma$. 
The proof then follows using that $\Gamma$ is a free left
$A$-module and ordering the invertible elements of $A_I$ on the left. The only nontrivial cases are $\Gamma_i$, $i\in\{1,2\}$, and 
$\Gamma_{12}$.
From the commutation relations 
(\ref{commrel}) (recall that the differential and the bimodule structure of the $4$-dimensional bicovariant calculi on $\mathrm{SL}_q(2)$ and $\mathrm{GL}_q(2)$ coincide up to the identification $\mathrm{det}_q=1$)
it immediately follows that
\begin{equation}\label{comalp}
\begin{split}
    &\omega^1\alpha^{-1}=q^{-1}\alpha^{-1}\omega^1,
    \qquad\qquad\qquad~~~\,\omega^2\alpha^{-1}=\alpha^{-1}\omega^2,\\
    &\omega^3\alpha^{-1}=\alpha^{-1}\omega^3
    +q^{-3}\lambda\alpha^{-2}\beta\omega^1,
    \qquad\omega^4\alpha^{-1}=q\alpha^{-1}\omega^4
    +\lambda\alpha^{-2}\beta\omega^2
\end{split}
\end{equation}
in $\Gamma_1$ and
\begin{equation}\label{comalp'}
\begin{split}
&    \omega^1\gamma^{-1}=q^{-1}\gamma^{-1}\omega^1,\qquad\qquad\qquad~~~\,
    \omega^2\gamma^{-1}=\gamma^{-1}\omega^2,\\
    &\omega^3\gamma^{-1}=\gamma^{-1}\omega^3
    +q^{-3}\lambda\gamma^{-2}\delta\omega^1,\qquad
    \omega^4\gamma^{-1}=q\gamma^{-1}\omega^4
    +\lambda\gamma^{-2}\delta\omega^2
\end{split}
\end{equation}
in $\Gamma_2$.
Recalling that a general element of $A_1$ is of the form
$\sum_ja^j\alpha^{-n_j}\in A_1$ for $n_j\in\mathbb{N}$, $a^j\in A$
and that $\Gamma$ is an $A$-bimodule which is freely generated as a
left $A$-module by $\omega^1, \omega^2, \omega^3, \omega^4$, 
the commutation relations (\ref{comalp}) imply 
$\Gamma A_1
\subseteq A_1\Gamma$.
Thus, $\Gamma_1=A_1\Gamma A_1\subseteq A_1\Gamma$ which gives
$\Gamma_1=A_1\Gamma$ (the inclusion $\Gamma_1\supseteq A_1\Gamma$  being trivial).
Similarly we prove $\Gamma_2=A_2\Gamma$, and
$\Gamma_{12}=A_{12}\Gamma$ using both (\ref{comalp}) and (\ref{comalp'}).
\end{proof}
In the proof of Proposition~\ref{Prop4D} it was shown that for all 
$U_I\in\mathcal{B}$ the right $H$-covariant
$A_I$-bimodule
$\Upsilon_{\mathrm{SL}_q}(U_I)
=\mathrm{span}_{A_I}\{\omega^1,\omega^2,\omega^3,\omega^4\}$
is freely generated from the left. 
This implies that $\Upsilon_{\mathrm{SL}_q}(U_1\cup U_2)=\Gamma$, i.e. that the
presheaf constructed by the Ore extensions
is already a sheaf and there is no sheafification required.
\begin{proposition}\label{PropBase4D}
The induced FODC 
$(\Upsilon_{\mathbf{P}^1(\mathbb{C})},\mathrm{d}_{\mathbf{P}^1(\mathbb{C})})$ of base forms on $\mathcal{O}_q(\mathbf{P}^1(\mathbb{C}))$
(described in Corollary~\ref{basef-cor}) is determined by
\begin{equation}\label{Base4D}
\begin{split}
    &\Upsilon_{\mathbf{P}^1(\mathbb{C})}(U_1):=\Gamma_{B_1}=
    B_1\mathrm{d}_1B_1=\mathrm{span}_\mathbb{C}\{u^k\mathrm{d}_1u~|~k\in\mathbb{N}\}~,\\
    &\Upsilon_{\mathbf{P}^1(\mathbb{C})}(U_2):=\Gamma_{B_2}=
    B_2\mathrm{d}_2B_2=
    \mathrm{span}_\mathbb{C}\{v^k\mathrm{d}_2v~|~k\in\mathbb{N}\}~,\\
    &\Upsilon_{\mathbf{P}^1(\mathbb{C})}(U_{12}):=\Gamma_{B_{12}}=
    B_{12}\mathrm{d}_{12}B_{12}=
    \mathrm{span}_\mathbb{C}\{u^k\mathrm{d}_{12}u~|~k\in\mathbb{Z}\}~.
\end{split}
\end{equation}
Moreover, $\Gamma_{B_1}=\mathrm{span}_{B_1}\{\alpha^{-2}\omega^2\}$,
$\Gamma_{B_2}=\mathrm{span}_{B_2}\{\gamma^{-2}\omega^2\}$
and $\Gamma_{B_{12}}=\mathrm{span}_{B_{12}}\{\alpha^{-2}\omega^2\}$
are free left modules and we have the sheaf isomorphism
\begin{equation}\label{slseq}
\Upsilon_{\mathbf{P}^1(\mathbb{C})}\cong\Upsilon_{\mathrm{SL}_q}^{\mathrm{co}H}\cap\Upsilon_{\mathrm{SL}_q}^\mathrm{hor}~.
\end{equation}
\end{proposition}
\begin{proof}
Using (\ref{4D+}), (\ref{comalp}) and
$\delta=\alpha^{-1}+q^{-1}\alpha^{-1}\beta\gamma$ we obtain
\begin{align*}
    \mathrm{d}_1u
    &=\mathrm{d}_1(\gamma\alpha^{-1})\\
    &=\mathrm{d}(\gamma)\alpha^{-1}-\gamma\alpha^{-1}\mathrm{d}(\alpha)\alpha^{-1}\\
    &=\bigg(\frac{q-1}{\lambda}\gamma\omega^1+\frac{q^{-1}-1}{\lambda}\gamma\omega^4
    -\delta\omega^2\bigg)\alpha^{-1}
    -\gamma\alpha^{-1}\bigg(\frac{q-1}{\lambda}\alpha\omega^1
    +\frac{q^{-1}-1}{\lambda}\alpha\omega^4-\beta\omega^2
    \bigg)\alpha^{-1}\\
    &=-\delta\omega^2\alpha^{-1}+\gamma\alpha^{-1}\beta\omega^2\alpha^{-1}\\
    &=-\alpha^{-2}\omega^2-q^{-1}\alpha^{-1}\beta\gamma\alpha^{-1}\omega^2
    +\gamma\alpha^{-1}\beta\alpha^{-1}\omega^2\\
    &=-\alpha^{-2}\omega^2~.
\end{align*}
Similarly $\mathrm{d}_2v=\mathrm{d}_2(\alpha\gamma^{-1})=-\gamma^{-2}\omega^2$.
The generators $u=\gamma\alpha^{-1}$ and $v=\alpha\gamma^{-1}$ of $B_1$ and $B_2$ have
therefore the commutation relations
\begin{equation}\label{eq129}
    (\mathrm{d}_1u)u=q^2u\mathrm{d}_1u~,~~~~~
    (\mathrm{d}_2v)v=q^{-2}v\mathrm{d}_2v~.
\end{equation}
Then the equalities in (\ref{Base4D}) follow, the last one also
recalling that $v=u^{-1}$ in $B_{12}$.

From $\mathrm{d}_1u=-\alpha^{-2}\omega^2$,
$\mathrm{d}_2v=-\gamma^{-2}\omega^2$ 
and the freeness of the left modules $\Gamma_1$, $\Gamma_2$ and $\Gamma_{12}$
it further
follows that
$\Gamma_{B_1}=\mathrm{span}_{B_1}\{\alpha^{-2}\omega^2\}$,
$\Gamma_{B_2}=\mathrm{span}_{B_2}\{\gamma^{-2}\omega^2\}$
and $\Gamma_{B_{12}}=\mathrm{span}_{B_{12}}\{\alpha^{-2}\omega^2\}$
are free left modules.  The sheaf equality \eqref{slseq} follows directly from
Theorem~\ref{thmcharbase}.
\end{proof}
Note that the commutation relations (\ref{eq129}) 
agree with those obtained by Chu, Ho and Zumino in \cite{zumino96}.

\begin{remark}
Consider the right $H$-covariant FODC
$(\Upsilon_{\mathrm{SL}_q},\mathrm{d})$ on $\mathcal{F}_{\mathrm{SL}_q}$
together with the bicovariant quotient calculus $(\Gamma_H,\mathrm{d}_H)$ on $H$
induced from the $4$-dimensional bicovariant FODC $(\Gamma,\mathrm{d})$
on $A$. The sequences
\begin{equation}\label{ses4D}
    0\rightarrow A_I\otimes_{B_I}\Gamma_{B_I}\rightarrow\Gamma_{I}
    \xrightarrow{\mathrm{ver}_I}A_I\square_H\Gamma_H
    \rightarrow 0
\end{equation}
are well-defined for all $U_I\in\mathcal{B}$ with the
vertical map given by $\mathrm{ver}_I\colon\Gamma_{I}
\rightarrow A_I\square_H\Gamma_H$, $a^i\omega^i\mapsto
a^i_0\otimes a^i_1[\omega^i]$ for all $a^i\in A_I$ (the proof is as in Lemma~\ref{lemma-ver}).
However $(\Upsilon_{\mathrm{SL}_q},\mathrm{d})$
is \textit{not} a principal calculus on the
QPB $\mathcal{F}_{\mathrm{SL}_q}$. For example the sequence (\ref{ses4D}) for $I=1$
is \textit{not} exact. Indeed first observe that, since $\Gamma_1$
and $\Gamma_{B_1}$ are free modules, $\omega^1\notin A_1\Gamma_{B_1}=\mathrm{span}_{A_1}\{\omega^2\}$.
Then, since $\omega^1$ is in the kernel of the projection
$\Gamma\rightarrow\Gamma_H$ it follows that
$\mathrm{ver}_1(\omega^1)=1\otimes[\omega^1]=0$
and therefore
$A_1\Gamma_{B_1}\subsetneq\ker\mathrm{ver}_1$.
\end{remark}

\subsubsection{Principal calculus via Ore localization of 3D calculus on \texorpdfstring{$\mathcal{O}_q(\mathrm{SL}_2(\mathbb{C}))$}{Oq(SL2(C))}}

In this section we show that the
failure of (\ref{ses4D}) to be an exact sequence can be cured by
considering a $3$-dimensional left covariant FODC instead of a $4$-dimensional
bicovariant FODC on $A=\mathcal{O}_q(\mathrm{SL}_2(\mathbb{C}))$ as the global
calculus of $\mathcal{F}_{\mathrm{SL}_q}$. We choose a left covariant
calculus on $A$ so to canonically obtain a left covariant calculus on
the quotient $H=\mathcal{O}_q(P)$. As in the previous section
we consider the QPB $\mathcal{F}_{\mathrm{SL}_q}$ with 
the cleaving maps (\ref{cleavmaps}).
\begin{lemma}\label{lem3D}
Let $(\Gamma,\mathrm{d})=(\Gamma_\mathrm{SL},\mathrm{d}_\mathrm{SL})$
be the $3$-dimensional left covariant FODC on $A$
of Example~\ref{ex01}. The induced FODC
$(\Upsilon_{\mathrm{SL}_q},\mathrm{d})$ on $\mathcal{F}_{\mathrm{SL}_q}$
is a free left $\mathcal{F}_{\mathrm{SL}_q}$-module of dimension $3$.
\end{lemma}
\begin{proof}
We first
provide the commutation relations of $\alpha^{-1}$ and $\gamma^{-1}$ with the
basis $\{\omega^0,\omega^1,\omega^2\}$ of the free $A$-module $\Gamma$.
From those in (\ref{eq125'}) we deduce
\begin{align}\label{3DCRalpha}
&\omega^0\alpha^{-1}=q^{-3}\alpha^{-1}\omega^0,\quad
& &\omega^1\alpha^{-1}=q^{-2}\alpha^{-1}\omega^1
+(q^{-6}-q^{-4})\alpha^{-2}\beta\omega^2,\quad
& &\omega^2\alpha^{-1}=q^{-3}\alpha^{-1}\omega^2\,,\\
&\omega^0\gamma^{-1}=q^{-3}\gamma^{-1}\omega^0,\quad
& &\omega^1\gamma^{-1}=q^{-2}\gamma^{-1}\omega^1
+(q^{-6}-q^{-4})\gamma^{-2}\delta\omega^2,\quad
& &\omega^2\gamma^{-1}=q^{-3}\gamma^{-1}\omega^2~.\nonumber
\end{align}
Then, proceeding as in the proof of Proposition~\ref{Prop4D}, we
conclude that $\Gamma_I=A_I\Gamma$ is a free left module (with
$I=1,2,12$).
\end{proof}
As in Proposition~\ref{PropBase4D} 
the sheaf of base forms $(\Upsilon_M,\mathrm{d}_M)$ is determined by
the free left modules 
\begin{align*}
    \Gamma_{B_1}=B_1\mathrm{d}_1B_1=\mathrm{span}_{B_1}\{\alpha^{-2}\omega^2\}\,,~~
    \Gamma_{B_2}=B_2\mathrm{d}_2B_2=\mathrm{span}_{B_2}\{\gamma^{-2}\omega^2\}\,,~~
    \Gamma_{B_{12}}=B_{12}\mathrm{d}_{12}B_{12}=\mathrm{span}_{B_{12}}\{\alpha^{-2}\omega^2\}\,.
\end{align*}
With  $u=\gamma\alpha^{-1}\in B_1$ and $v=\alpha\gamma^{-1}\in
  B_2$ we have $\mathrm{d}_1u=q^{-3}\alpha^{-2}\omega^2$
and $\mathrm{d}_2v=-q^{-2}\gamma^{-2}\omega^2$.
The $\Gamma_{B_I}$ bimodule relations are then determined by
\begin{equation*}
    (\mathrm{d}_1u)u=q^2u\:\!\mathrm{d}_1u~,~~~~~
    (\mathrm{d}_2v)v=q^{-2}v\mathrm{d}_2v~.
\end{equation*}
We prove for example the properties of  $\Gamma_{B_1}$. We have
\begin{equation*}
\begin{split}
\mathrm{d}_1u&=    \mathrm{d}_1(\gamma\alpha^{-1})\\
    &=(\mathrm{d}\gamma)\alpha^{-1}
    -\gamma\alpha^{-1}\mathrm{d}(\alpha)\alpha^{-1}\\
    &=\gamma\omega^1\alpha^{-1}+\delta\omega^2\alpha^{-1}
    -\gamma\alpha^{-1}(\alpha\omega^1+\beta\omega^2)\alpha^{-1}\\
    &=q^{-3}(\delta\alpha^{-1}-\gamma\alpha^{-1}\beta\alpha^{-1})\omega^2\\
    &=q^{-3}\alpha^{-2}\omega^2
\end{split}
\end{equation*}
where we used (\ref{eq124}), (\ref{3DCRalpha}) and
$\delta\alpha^{-1}-q^{-1}\alpha^{-1}\beta\gamma\alpha^{-1}=\alpha^{-2}$.
Again by (\ref{3DCRalpha}) the commutation relation
\begin{align*}
    (\mathrm{d}_1u)u
    =q^{-3}\alpha^{-2}\omega^2\gamma\alpha^{-1}
    =q^{-3}\alpha^{-2}\gamma\alpha^{-1}\omega^2
    =q^{-1}\gamma\alpha^{-3}\omega^2
    =q^2u\mathrm{d}_1u
\end{align*}
follows. Thus, 
$\Gamma_{B_1}=B_1\mathrm{d}_1B_1=\mathrm{span}_{B_1}\{\alpha^{-2}\omega^2\}$. Similarly
for $\Gamma_{B_2}$ and $\Gamma_{B_{12}}$. 

\begin{proposition}\label{Principal3D}
Let $(\Upsilon_{\mathrm{SL}_q},\mathrm{d})$ be the FODC on the locally trivial QPB $\mathcal{F}_{\mathrm{SL}_q}$ of Lemma~\ref{lem3D}.
Let $(\Gamma_H,\mathrm{d}_H)$ be the left covariant FODC on $H$ 
induced from the $3$-dimensional left covariant FODC $(\Gamma,\mathrm{d})$ on $A$ of Example~\ref{ex01}.
The sequences
\begin{equation}\label{ses3D}
\qquad\qquad    0\rightarrow A_I\otimes_{B_I}\Gamma_{B_I}\rightarrow\Gamma_{I}
    \xrightarrow{\mathrm{ver}_I}A_I\square_H\Gamma_H
    \rightarrow 0\,, \quad \quad U_I\in\mathcal{B}=\{\emptyset,U_1,U_2,U_{12}\}
\end{equation}
are exact and
$(\Upsilon_{\mathrm{SL}_q},\mathrm{d})$ is a principal calculus on $\mathcal{F}_{\mathrm{SL}_q}$.
\end{proposition}

\begin{proof}
Recalling Lemma~\ref{lemma-ver},
that $\Gamma_I$ are free left $A_I$-modules
and the left linearity of the vertical map (\ref{eq133}), we see that the vertical maps
in \eqref{ses3D} are given by $\mathrm{ver}_I\colon\Gamma_{I}
\rightarrow A_I\square_H\Gamma_H$, $a^\ell\omega^\ell\mapsto
a^\ell_0\otimes a^\ell_1[\omega^\ell]$, where $a^\ell\in A_I$.
In particular, they are well-defined. 
We first show that the second arrow in (\ref{ses3D}) is injective
in the case of $U_1$. The results for the other opens in
$\mathcal{B}$ follow analogously.
Recall that $\Gamma_{B_1}$ is generated  as a left  $B_1$-module by
$\alpha^{-2}\omega^2$ (cf. the discussion following Lemma~\ref{lem3D}). An
arbitrary element of $A_1\otimes_{B_1}\Gamma_{B_1}$ is therefore of the
form $a^i\otimes_{B_1}b^i\alpha^{-2}\omega^2$ with
$a^i\in A_1$ and $b^i\in B_1$ (sum understood). Since, by Lemma~\ref{lem3D},
$\Gamma_1$ is a free left $A_1$-module with basis $\{\omega^0,\omega^1,\omega^2\}$, 
it follows that $a^ib^i\alpha^{-2}\omega^2=0$ if and only if
$a^ib^i\alpha^{-2}=0$, i.e., if and only if, $a^ib^i=0$, which proves injectivity.
We next prove exactness in $\Gamma_{I}$.
The condition
$0=\mathrm{ver}_I(a^\ell\omega^\ell)
=a^\ell_0\otimes a^\ell_1[\omega^\ell]$ (sum understood) holds if and only if $a^0=a^1=0$,
indeed $[\omega^2]=0$, while $[\omega^0]$ and $[\omega^1]$ form a
basis of the
free left $H$-module $\Gamma_H$, cf. Example~\ref{ex01}.
This proves 
$\ker\mathrm{ver}_I={\rm{span}}_{A_I}\{\omega_2\}=A_I\Gamma_{B_I}$. In
order to prove surjectivity of $\mathrm{ver}_I\colon
\Gamma_{I}\rightarrow A_I\square_H\Gamma_H$ let $a^{i\ell}\otimes
h^i[\omega^\ell]$ (sums understood) be an arbitrary element in
$A_I\square_H\Gamma_H$. Since the forms $[\omega^\ell]$ are
left $H$-coinvariant, by  definition of 
cotensor product  we have $a^{i\ell}_0\otimes a^{i\ell}_1
\otimes h^i=a^{i\ell}\otimes h^i_1\otimes h^i_2$.  Then
$$
\mathrm{ver}_I(a^{i\ell}\epsilon(h^i)\omega^\ell)
=a^{i\ell}_0\otimes a^{i\ell}_1\epsilon(h^i)[\omega^\ell]
=a^{i\ell}\otimes h^i[\omega^\ell]
$$
and the sequences in  (\ref{ses3D}) are exact.
\end{proof}
As in the previous section the freeness of the presheaf $\Upsilon_{\mathrm{SL}_q}$
implies that there is no sheafification required and in particular
$\Upsilon_{\mathrm{SL}_q}(\mathbf{P}^1(\mathbb{C}))=\Gamma$.

We have seen that the differential calculus $(\Upsilon_{\mathrm{SL}_q},\mathrm{d})$ is principal.
Note that it is not principal covariant
since the induced left covariant FODC $(\Gamma_H,\mathrm{d}_H)$
on $H$ is not bicovariant. This is the case because the
ideal $I=\mathrm{span}_H\{(t-1)(t-q^2)
, p^2,(t-q^2)p\}$ characterizing
$(\Gamma_H,\mathrm{d}_H)$ is not closed under the adjoint
right coaction, e.g. $\mathrm{Ad}_R(p^2)\notin I\otimes H$
(it contains a polinomial in $t$ with roots different from $1$ and
$q^2$. See \cite{wor} for the classification of bicovariant FODCi on
Hopf algebras in terms of ideals closed under the right adjoint coaction).\

\subsubsection{Principal covariant calculus via Ore localization of 4D calculus on \texorpdfstring{$\mathcal{O}_q(\mathrm{GL}_2(\mathbb{C}))$}{Oq(GL2(C))}}

Consider the principal bundle $\mathrm{GL}_2(\mathbb{C})\rightarrow\mathrm{GL}_2(\mathbb{C})/P_\mathrm{GL}\cong\mathbf{P}^1(\mathbb{C})$
with $P_\mathrm{GL}$ the upper Borel in $\mathrm{GL}_2(\mathbb{C})$. As before,
localizing with respect to $\alpha$ and $\gamma$ gives the opens $U_1$ and $U_2$ with corresponding topology $\{\emptyset,U_1,U_2,U_{12},\mathbf{P}^1(\mathbb{C})\}$ on $\mathbf{P}^1(\mathbb{C})$.
Although $\mathrm{GL}_2(\mathbb{C})$ is not semisimple,
a quantum deformation of this bundle, which is 
a QPB
$\mathcal{F}_{\mathrm{GL}_q}$ on $\mathbf{P}^1(\mathbb{C})$,
is obtained via
Ore extensions of the quantum group $A=\mathcal{O}_q(\mathrm{GL}_2(\mathbb{C}))$.
Explicitly, we define the sheaf  of algebras by $\mathcal{F}_{\mathrm{GL}_q}(\emptyset):=\{0\}$,
$\mathcal{F}_{\mathrm{GL}_q}(\mathbf{P}^1(\mathbb{C})):=A$,
\begin{equation*}
    \mathcal{F}_{\mathrm{GL}_q}(U_1):=A_1:=A[\alpha^{-1}]~,~~~
    \mathcal{F}_{\mathrm{GL}_q}(U_2):=A_2:=A[\gamma^{-1}]~,~~~
    \mathcal{F}_{\mathrm{GL}_q}(U_1\cap U_2):=A_{12}:=A[\alpha^{-1},\gamma^{-1}]~.
\end{equation*}
The Hopf algebra quotient
$A\rightarrow H=\mathcal{O}_q(P_{\mathrm{GL}})=A/\langle\gamma\rangle$ defined from Example~\ref{eq04}
induces a natural comodule algebra structure on $\mathcal{F}_{\mathrm{GL}_q}$.
This is a locally trivial QPB 
with cleaving maps $j_1\colon H\rightarrow A_1$,
$j_2\colon H\rightarrow A_2$ and
$j_{12}=j_1\colon H\rightarrow A_1\subset A_{12}$ given by
\begin{align}\label{trivializations}
    j_1\begin{pmatrix}
    t & p\\
    0 & s
    \end{pmatrix}
    =\begin{pmatrix}
    \alpha & \beta\\
    0 & \alpha^{-1}\mathrm{det}_q
    \end{pmatrix}~,~j_1(r')=r~,~~~~~~
    j_2\begin{pmatrix}
    t & p\\
    0 & s
    \end{pmatrix}
    =\begin{pmatrix}
    \gamma & \delta\\
    0 & \gamma^{-1}\mathrm{det}_q
    \end{pmatrix}~,~j_2(r')=r~,
\end{align}
where $r\in A$ and $r'\in H$ are the inverses of the quantum determinants.
It is straightforward to verify that the sheaf $\mathcal{O}_{\mathbf{P}^1(\mathbb{C})}$ of coinvariants of the QPB $\mathcal{F}_{\mathrm{GL}_q}$ coincides with the one defined in (\ref{OqP1C}).
\begin{proposition}
The Ore extension of the bicovariant $4$-dimensional FODC $(\Gamma,\mathrm{d}):=(\Gamma_{\mathrm{GL}},\mathrm{d}_\mathrm{GL})$
of Example~\ref{eq04} gives a right $H$-covariant FODC $(\Upsilon_{\mathrm{GL}_q},\mathrm{d})$
on the QPB $\mathcal{F}_{\mathrm{GL}_q}$.
Explicitly, the sheaf $\Upsilon_{\mathrm{GL}_q}$
of right $H$-covariant $\mathcal{F}_{\mathrm{GL}_q}$-bimodules is defined by
$\Upsilon_{\mathrm{GL}_q}(\emptyset)=\{0\}$,
$\Upsilon_{\mathrm{GL}_q}(\mathbf{P}^1(\mathbb{C}))=\Gamma$ and
\begin{equation}\label{GLSheaf}
    \Upsilon_{\mathrm{GL}_q}(U_1)=A_1\otimes_A\Gamma\otimes_AA_1~,~~~
    \Upsilon_{\mathrm{GL}_q}(U_2)=A_2\otimes_A\Gamma\otimes_AA_2~,~~~
    \Upsilon_{\mathrm{GL}_q}(U_1\cap U_2)=A_{12}\otimes_A\Gamma\otimes_AA_{12}~.
\end{equation}
Together with the pullback calculus on the sheaf $\mathcal{O}_{\mathbf{P}^1(\mathbb{C})}$ of coinvariants
and the bicovariant quotient calculus on $H$
this gives a principal covariant calculus on $\mathcal{F}_{\mathrm{GL}_q}$.
\end{proposition}
\begin{proof}
Clearly (\ref{GLSheaf}) defines a presheaf $\Upsilon_{\mathrm{GL}_q}$
of right $H$-covariant $\mathcal{F}_{\mathrm{GL}_q}$-bimodules with
restriction morphisms as in (\ref{eq108}).
Repeating the computations of Proposition~\ref{Prop4D} and Proposition~\ref{PropBase4D} one shows that
\begin{enumerate}
\item[i.)] $\Upsilon_{\mathrm{GL}_q}(U_I)$ is a free left $\mathcal{F}_{\mathrm{GL}_q}(U_I)$-module
generated by $\{\omega^1,\omega^2,\omega^3,\omega^4\}$,

\item[ii.)] the base forms $(\Upsilon_{\mathbf{P}^1(\mathbb{C})},\mathrm{d}_{\mathbf{P}^1(\mathbb{C})})$ are determined by
$\Gamma_{B_1}=\mathrm{span}_{B_1}\{-\alpha^{-2}\omega^2\}$,
$\Gamma_{B_2}=\mathrm{span}_{B_2}\{-\gamma^{-2}\omega^2\}$
and $\Gamma_{B_{12}}=\mathrm{span}_{B_{12}}\{-\alpha^{-2}\omega^2\}$.
\end{enumerate}
Using the freeness of the presheaf
$\Upsilon_{\mathrm{GL}_q}$ it is not hard to check that
$\Upsilon_{\mathrm{GL}_q}(U_1\cup U_2)=\Upsilon_{\mathrm{GL}_q}(\mathbf{P}^1(\mathbb{C}))$, proving that $\Upsilon_{\mathrm{GL}_q}$ is a sheaf
of right $H$-covariant $\mathcal{F}_{\mathrm{GL}_q}$-bimodules.
By Proposition~\ref{ore-calc} the extended differential $\mathrm{d}_I\colon A_I
\rightarrow\Upsilon_{\mathrm{GL}_q}(U_I)$ defines a right 
$H$-covariant FODC for each open  $U_I\in\mathcal{B}$ and it is
straightforward to verify that this determines a morphism 
$\mathrm{d}\colon\mathcal{F}_{\mathrm{GL}_q}\rightarrow\Upsilon_{\mathrm{GL}_q}$
of sheaves of right $H$-comodules.

We next prove that the sequences 
\begin{equation}\label{ProofSequence}
    0\rightarrow A_I\otimes_A\Gamma_{B_I}\rightarrow\Gamma_{I}
\xrightarrow{\mathrm{ver}_I}A_I\square_H\Gamma_H\rightarrow 0
\end{equation}
are exact for all $U_I\in\mathcal{B}$.
From Theorem \ref{prop-ff} we have injectivity of the second arrow. Recall from Example~\ref{eq04} that the bicovariant quotient calculus $(\Gamma_H,\mathrm{d}_H)$
on $H$ is $3$-dimensional (instead of $2$-dimensional in the framework of
Proposition~\ref{Prop4D}) with left coinvariant basis $\{[\omega^1],[\omega^3],[\omega^4]\}$ and that by Lemma~\ref{lemma-ver} the
vertical map $\mathrm{ver}_I\colon\Gamma_I\rightarrow A_I\square_H\Gamma_H$
is well-defined for all $U_I\in\mathcal{B}$.
Now take an arbitrary element $a^\ell\omega^\ell\in\Gamma_I$ (sum understood).
Because of freeness 
$0=\mathrm{ver}_I(a^\ell\omega^\ell)=a^\ell_0\otimes a^\ell_1[\omega^\ell]$ if and only if
$a^1=a^3=a^4=0$.
Thus $\ker\mathrm{ver}_I=\mathrm{span}_{A_I}\{\omega^2\}=A_I\Gamma_{B_I}$.
Surjectivity of $\mathrm{ver}_I$ is verified as in Proposition~\ref{Principal3D}.
Hence the sequences (\ref{ProofSequence}) are exact and $(\Upsilon_{\mathrm{GL}_q},\mathrm{d})$ is a principal covariant calculus on $\mathcal{F}_{\mathrm{GL}_q}$.
\end{proof}
As a corollary, applying Theorem~\ref{thmcharbase} we have that
$\Upsilon_{\mathbf{P}^1(\mathbb{C})}\cong\Upsilon_{\mathrm{GL}_q}^{\mathrm{co}H}\cap\Upsilon_{\mathrm{GL}_q}^\mathrm{hor}$
are isomorphic sheaves.

\medskip

The sheaf $\mathcal{F}_{\mathrm{GL}_q}$ is locally trivial since
$A_I\cong B_I\# H$ according to \eqref{trivializations}. Considering the 
FODC on
$(\Upsilon_{\mathbf{P}^1(\mathbb{C})},\mathrm{d}_{\mathbf{P}^1(\mathbb{C})})$
and the bicovariant calculus $(\Gamma_H,\mathrm{d}_H)$ on $H$,
we can construct locally the smash product calculus
$(\Gamma_{B_I\#H}, \mathrm{d}_{B_I\#H})$  of Section \ref{smash}, which
can be pulled back to $A_I$ via $\varphi_I:A_I\to B_I\# H$ (the module
of one forms being $\Gamma_{B_I\#H}$ since $\varphi_I$ are isomorphisms). However this local
data does not give a FODC on  $\mathcal{F}_{\mathrm{GL}_q}$. Indeed
for that to be the case the restriction morphisms $r_{IJ}^{\!\Upsilon_\#}$ of the would be
sheaf ${\Upsilon_\#}$ of covariant 
$\mathcal{F}_{\mathrm{GL}_q}$-bimodules would have to be left 
$\mathcal{F}_{\mathrm{GL}_q}$-linear and compatible with the
differentials: $\mathrm{d}_{B_I\#H}\circ \varphi_I\circ
r_{IJ}= r_{IJ}^{\!\Upsilon_\#}\circ \mathrm{d}_{B_J\#H}\circ \varphi_J$. Recalling that 
$\Gamma_{B_I\#H}$ are free left ${B_I\#H}$-bimodules it follows form
straightforward computations that there is not such map $
r_{12,2}^{\!\Upsilon_\#}:\Gamma_{B_2\#H}\to \Gamma_{B_{12}\#H}$.

As a consequence we have that in this example the Ore extended differential calculi $(\Gamma_{A_I}, \mathrm{d}_I)$ on
   $A_I\cong B_I\#H$ are not isomorphic to the smash product calculi
   $(\Gamma_{B_I\#H}, \mathrm{d}_{B_I\#H})$. While the latter, from a bottom
   up approach, could seem to be the canonical choice, it turns out
   that it  is the first one to provide a FODC on the sheaf
   $\mathcal{F}_{\mathrm{GL}_q}$.
 This is yet another instance of the non uniqueness of the
   differential calculus construction.
\bigskip

{\bf Acknowledgments}.
We wish to thank A. Carotenuto, F. Gavarini, A. Krutov, R. O'Buachalla, C. Pagani, F. Zanchetta
 for helpful conversations. 
R.F. and E.L. wish to thank the DISIT at the 
Universit\`a del Piemonte Orientale for the warm hospitality during
the completion of this work. 
All authors wish to thank the Dept. of Mathematics at Charles University,
Prague, for the warm hospitality during the completion of this work.

The work of P.A. and T.W. is partially
supported by INFN, CSN4, Iniziativa Specifica GSS, and by Universit\`a del
Piemonte Orientale. P.A. is also
affiliated to INdAM, GNFM (Istituto Nazionale di Alta Matematica,
Gruppo Nazionale di Fisica Matematica). 
R.F. and E.L. are affiliated to INFN, GAST, Bologna.
 R.F. is also affiliated to INdAM, GNSAGA.
This research was funded by COST:  CaLISTA CA 21109.

\end{document}